\documentclass[11pt,twoside]{article}
\usepackage{bbm}
\usepackage{mathrsfs}
\usepackage{amssymb}
\usepackage{amsmath}
\usepackage{amsthm}
\usepackage{amsfonts}
\usepackage{color}
\usepackage{graphicx}
\usepackage[active]{srcltx}
\usepackage{CJK}
\usepackage{geometry}
\usepackage[cp1252]{inputenc}

\usepackage{mathrsfs}
\usepackage{graphicx}

\usepackage[active]{srcltx}

\allowdisplaybreaks

\usepackage{titletoc}
\titlecontents{section}[0pt]{\addvspace{2pt}\filright}
              {\contentspush{\thecontentslabel\ }}
              {}{\titlerule*[8pt]{.}\contentspage}

\def\rr{{\mathbb R}}
\def\rn{{{\rr}^n}}

\def\rrm{{{\rr}^m}}

\def\nn{{\mathbb N}}

\def\cS{{\mathcal S}}

\def\csgv{\mathcal {S}_{g;V}}
\def\csgu{\mathcal {S}_{g;U}}

\def\bi{{\bf I}}

\def\odl{\overline {d}_{\lambda}}

\def\cf{{\mathcal F}}

\def\dcc{{d_{CC}}}

\def\dl{d_\lambda}
\def\dlu{d^{U}_\lambda}

\def\rl{\rho_\lambda}

\def\fz{\infty}
\def\az{\alpha}

\def\dist{{\mathop\mathrm{\,dist\,}}}
\def\loc{{\mathop\mathrm{\,loc\,}}}

\def\lip{{\mathop\mathrm{\,Lip}}}

\def\lz{\lambda}
\def\dz{\delta}

\def\ez{\epsilon}

\def\gz{{\gamma}}

\def\Oz{{\Omega}}

\def\tz{\theta}

\def\pa{\partial}
\def\pau{\partial{U}}

\def\wz{\widetilde}

\def\bl{\bigg (}
\def\br{\bigg )}

\def\bbbl{\bigg \{}
\def\bbbr{\bigg \}}

\def\bint{{\ifinner\rlap{\bf\kern.35em--}
\int\else\rlap{\bf\kern.45em--}\int\fi}\ignorespaces}

\def\bbint{{\ifinner\rlap{\bf\kern.35em--}
\hspace{0.078cm}\int\else\rlap{\bf\kern.45em--}\int\fi}\ignorespaces}

\def\esup{\mathop\mathrm{\,esssup\,}}

\def\Lip{\mathop\mathrm{Lip}}

\def\r{\right}
\def\lf{\left}
\def\la{\langle}
\def\ra{\rangle}

\newtheorem{thm}{Theorem}[section]
\newtheorem{lem}[thm]{Lemma}
\newtheorem{prop}[thm]{Proposition}
\newtheorem{rem}[thm]{Remark}
\newtheorem{cor}[thm]{Corollary}
\newtheorem{defn}[thm]{Definition}
\newtheorem{ques}[thm]{Problem}
\numberwithin{equation}{section}

\newtheorem{Assumption}[]{Assumption}

\begin{document}

\arraycolsep=1pt

\title{\Large\bf
A Rademacher type theorem for Hamiltonians $H(x,p)$ and {\color{black} an } application to absolute minimizers
 \footnotetext{\hspace{-0.35cm}
\endgraf
 2020 {\it Mathematics Subject Classification:} Primary 49J10 $\cdot$ 49J52; Secondary 35F21 $\cdot$ 51K05.
\endgraf The first author is supported by the Academy of Finland via the projects: Quantitative rectifiability in Euclidean and non-Euclidean spaces, Grant No. 314172, and Singular integrals, harmonic functions, and boundary regularity in Heisenberg groups,
Grant No. 328846.
 The second author is supported by the National Natural Science Foundation of China (No.  12025102 \& No. 11871088)  and   by the Fundamental Research Funds for the Central Universities.
\endgraf Data sharing not applicable to this article as no datasets were generated or analysed during the current study.}
}
\author{Jiayin Liu, Yuan Zhou}
\date{ }
\maketitle

\begin{center}
\begin{minipage}{13.5cm}\small
{\noindent{\bf Abstract.}\quad
We establish  a Rademacher type theorem  involving Hamiltonians $H(x,p)$ under  very weak conditions in both of Euclidean and Carnot-Carath\'eodory spaces.
In particular,  $H(x,p)$ is assumed to be only  measurable in the variable $x$,
and  to be quasiconvex and lower-semicontinuous in the variable $p$. {\color{black}Without the lower-semicontinuity in the variable $p$, we provide a counter example showing the failure of such   a Rademacher type theorem.}
Moreover, {\color{black} by applying such a Rademacher type theorem}  we  build up an  existence result of absolute minimizers for {\color{black}the} corresponding $L^\infty$-functional.   These improve or extend several known results in the literature.

}
\end{minipage}
\end{center}

 \tableofcontents
 \contentsline{section}{References\numberline{ }}{50}

\section{Introduction}\label{s1}

Let $ \Omega\subset \rn$ be a   domain (that is, an open and connected {\color{black}subset}) of $\rn$  with $n\ge2$.     {\color{black}We first recall  Rademacher's theorem in Euclidean spaces. See Appendix for some of its consequence related to Sobolev and Lipschitz spaces. }
\begin{thm} \label{rade}
  If $u:{\color{black}\Omega}\to\rr$ is a Lipschitz function,
that is,
 \begin{equation}\label{lip}|u(x)-u(y)|\le \lz|x-y|\quad\forall x,y\in \Omega \quad \mbox{for some $0\le \lz<\fz$,}
 \end{equation}
 then,   at almost all $x\in\Omega$, $u$ is differentiable  and
 $|\nabla u(x)| =\lip u(x)$.  Here $|\nabla u(x)|$ is the Euclidean length of the derivative $\nabla u(x)$ at $x$, and $\lip u (x)$ is  the pointwise Lipschitz {\color{black}constant   at $x$ defined by
 }
 \begin{equation}\label{pointlip} \lip u(x):=\limsup_{y\to x}\frac{|u(y)-u(x)|}{|y-x|} . \end{equation}
\end{thm}

{\color{black}The above  Rademacher's theorem was extended to      Carnot-Carath\'{e}odory spaces $(\Omega, X)$}, where $X$ is  a family  of smooth vector fields in $\Oz$ satisfying the H\"ormander condition (See Section 2).  {\color{black} Denote by  $Xu$  the distributional horizontal derivative of $u\in L^1_\loc(\Omega)$. Write   $\dcc$  as the Carnot-Carath\'{e}odory distance with respect to $X$.
One then has the following; see \cite{gn,FSS97,fhk99,ksz} and, for the better result in
Carnot group and Carnot type vector field, see \cite{monti,p89}.
\begin{thm} \label{radcc}
  If $u:\Omega\to\rr$ is a Lipschitz function with respect to $\dcc$,
that is,
 \begin{equation}\label{lip01}|u(x)-u(y)|\le \lz \dcc(x,y)\quad\forall x,y\in \Omega \quad \mbox{for some $0\le \lz<\fz$,}
 \end{equation}
 then, $Xu \in L^\fz(\Omega,\rr^m)$ and, for almost all $x\in\Omega$, the length
 $|X u(x)| \le \lz$.

 Under the additional assumption that   $X$ is a Carnot type vector field in $\Omega$, 
or in particular,
$(\Omega, X)$ is a domain in some Carnot group, 
one further has  $|X u(x)| = \lip_{d_{CC}}u(x)
$ for almost all $x\in\Omega$, where $\lip_\dcc u(x) $   is
 defined by \eqref{pointlip} with $|y-x|$ replaced by $d_{CC}(y,x)$.
\end{thm}
}

This paper aims to build up some Rademacher type theorem involving Hamiltonians $H(x,p)$ in both of Euclidean and Carnot-Carath\'eodory spaces. {\color{black}  Throughout this paper, the following assumptions  are always held for $H(x,p)$.}

\begin{Assumption} \label{ham}
Suppose that $H:\Omega\times\rr^m\to{\color{black}[0,+\fz)}$ is measurable  and satisfies
{\color{black}  \begin{enumerate}
  \vspace{-0.15cm}\item [(H1)]   For each $x\in\Omega$,
  $H(x,\cdot)$  is quasi-convex,  that is,
 \begin{equation*}
 H(x,t p + (1 - t)q) \le \max \{H(x,p), H(x,q)\}, \quad \forall p, q \in \rr^m, \ \forall \, t\in[0,1] \ \text{and} \  \forall x\in\Omega.
\end{equation*}
\item [(H2)] \vspace{-0.15cm} For each $  x\in \Omega$, $H(x,0)=\min_{p\in\rrm}H(x,p)=0$.
 \vspace{-0.15cm} \item [(H3)] It holds that
  $R_\lz < \fz$ for all $\lz \ge 0$, and $\lim_{\lz \to \fz}R_\lz' = \fz$, where and in below,
 $$R_\lz:= \sup \{|p| \ | \ (x,p)\in \Omega \times \rr^{m},H(x,p)\le \lz \} $$
   and
  $$ R'_\lz:= \inf \{|p| \ | \ (x,p)\in \Omega \times \rr^{m},H(x,p)\ge \lz \}.$$
  \end{enumerate}}
\end{Assumption}

For any $\lz\ge 0$, we define
\begin{equation}\label{d311}
d _{\lambda}  (x,y):= \sup\{u(y)-u(x) \ | \ u \in   \dot W^{1,\infty}_{X }(\Omega) \
\mbox{ with    }\  \|H(\cdot, Xu)\|_{L^\infty(\Omega)} \le \lz
  \}\quad\forall \mbox{$x,y \in \Omega$}.
\end{equation}
Recall that $\dot W^{1,\fz}_{X }(\Omega)$ denotes the set of
 all functions $u\in L^\fz (\Omega)$ whose
 distributional horizontal derivatives $Xu\in L^\fz(\Omega;\rr^m)$.  {\color{black}
 It was known that any function $ u\in \dot W^{1,\fz}_{X }(\Omega)$   admits a  continuous representative $\wz u$; see \cite[Theorem 1.4]{FSS97} and also Theorem \ref{holder} and Remark \ref{conrep} below. In this paper, in particular, in \eqref{d311} above,  we  always
identify functions in $\dot W^{1,\fz}_X(\Omega)$ as their continuous representatives. We remark that $d_\lz$ is  not a distance  necessarily, but    Lemma \ref{lem311}   says that $d_\lz$ is  always a pseudo-distance  as defined in Definition \ref{pseudo} below.}

Given any $\lz\ge 0$, by the definition \eqref{d311},
 if  $u\in  \dot W_{X}^{1,\fz}(U)$ and $\|H(\cdot,Xu)\|_{L^\fz(U)}\le \lz$,
 then  $
u(y)-u(x)\leq d_{\lambda}(x,y) \ \forall\, x,y\in \Omega.$
It is natural to ask whether the converse is true or not.
However, such converse is not necessarily true as  witted by   the Hamiltonian
\begin{equation}\label{[p]}
\lfloor|p|\rfloor =\max\{t\in \nn \ | \  |p|-t\ge 0\}\quad \forall x\in \Omega,\, p\in\rrm;
\end{equation}
 for details see Remark \ref{eg} below. The point is that $ \lfloor |p|\rfloor$  is not lower-semicontinuous.
Below, the converse is shown to be true if $H(x,p)$ is assumed additionally to be  lower-semicontinuous in the variable $p$, that is,
\begin{enumerate}\vspace{-0.15cm}
\em\item[(H0)] For almost all $x\in\Omega$, 
$H(x,p)\le\liminf_{q\to p}H(x,q) \quad\forall p\in\rr^m.$
\end{enumerate}

\begin{thm} \label{rad}
Suppose that $H$ satisfies (H0)-(H3).
Given  any $\lz \ge 0 $ and any function $u:\Omega\to \rr$,
 the following are equivalent:

  \begin{enumerate}
  \vspace{-0.15cm}
  \item[(i)] $u\in \dot W_{X}^{1,\fz}(\Omega) \ \text{and}  \ \|H(\cdot,Xu)\|_{L^\fz(\Omega)}\le\lz$;
  \vspace{-0.15cm}
  \item[(ii)]
$u(y)-u(x)\le d_\lz (x,y)$   $\forall x,y\in \Omega$;
\vspace{-0.15cm}
\item[(iii)]
For any $x\in\Omega$, there exists a neighborhood $N(x)\subset\Omega$ such that
$$u(y)-u(z)\le  d_\lz (z,y)\quad \forall y,z \in N(x).$$
\end{enumerate}

\noindent In particular, {\color{black}if $u:\Omega\to \rr$ satisfies any one of (i)-(iii), then}
\begin{align}\label{ident}\|H(\cdot,Xu)\|_{L^\fz(\Omega)}&={\color{black}\inf}\{\lz \ge 0 \ | \ \lz \ \mbox{satisfies  (ii)}\}={\color{black}\inf}\{\lz \ge 0 \ | \ \lz \ \mbox{satisfies   (iii)}\}.
\end{align}
\end{thm}


Using Theorem \ref{rad}, when $\lz\ge\lz_H$ we prove that $d_\lz$ has a pseudo-length property, which allows us to get the following.
Here and below, define
\begin{equation}\label{lh}
  \lz_H:=\inf\{\lz\ge0 \ |  \ R'_\lz>0\}.
\end{equation}
Since $R_\lz'$ defined in (H3) of Assumption \ref{ham} is always  nonnegative and increasing in $\lz\ge0$ and tends to $\fz $ as $\lz\to\fz$, we  know that $0\le \lz_H<\fz$, and moreover,
$\lz>\lz_H$ if and only if $R'_\lz>0$.

\begin{thm} \label{radiv}
Suppose that $H$ satisfies (H0)-(H3).
Given  any $\lz \ge \lz_H$ and any function $u:\Omega\to \rr$,
  the statement (i)  in Theorem \ref{rad} is equivalent to the following
  \begin{enumerate}
\item[(iv)]
For any $x\in\Omega$, there exists a neighborhood $N(x)\subset\Omega$ such that
$$u(y)-u(x)\le  d_\lz (x,y)\quad \forall y  \in N(x).$$
\end{enumerate}

\noindent In particular, {\color{black}if $u:\Omega\to \rr$ satisfies   (iv), then}
\begin{align}\label{ident2}\max\{\lz_H,\|H(\cdot,Xu)\|_{L^\fz(\Omega)}\}
 =\min\{\lz \ge \lz_H \ | \ \lz \ \mbox{satisfies   (iv)}\}.
 \end{align}
\end{thm}

As a consequence of Theorem \ref{rad} and Theorem \ref{radiv}, we have
the following Corollary \ref{global}.
Associated to Hamiltonian $H(x,p)$, we introduce some  notion and notations.
Denote by $\dot W^{1,\fz}_H(\Omega)$ the collection of all
$u\in \dot W^{1,\fz}_X(\Omega)$ with $\|H(\cdot,Xu)\|_{L^\fz(\Omega)}<\fz$.
Denote by $\lip_H(\Omega,X)$  the class of functions $u:\Omega\to\fz$ {\color{black}satisfying (ii) for some $\lz>0$} equipped with (semi-)norms
\begin{equation}\label{liph}
  \lip_H(u,\Omega)={\color{black}\inf}\{\lz \ge \lz_H \ | \ \lz \ \mbox{satisfies  (ii)}\}.
\end{equation}
Denote by $\lip_H^\ast( \Omega)$ the collection of all functions $u$ with
$$\lip_H^\ast(u,\Omega)=\sup_{x\in\Omega}\lip_Hu(x)<\fz,$$
where  we write the  pointwise  ``Lipschitz" constant
\begin{equation}\label{ptliph}
  \lip_Hu(x)={\color{black}\inf}\{\lz\ge\lz_H \ | \ \lz \ \mbox{satisfies   (iv)}\}.
\end{equation}
{\color{black}Thanks to the right continuity of the map $\lz\in[\lz_H,\fz)\mapsto d_\lz(x,y)$ as given in Lemma \ref{dpro},
the infima in \eqref{liph} and \eqref{ptliph} are actually minima.}

\begin{cor} \label{global}
Suppose that $H$ satisfies (H0)-(H3) with $\lz_H=0$.
Then $\dot W^{1,\fz}_H(\Omega)=\lip_H(\Omega)=\lip_H^\ast( \Omega)$ and
 $$\|H(\cdot,Xu)\|_{L^\fz(\Omega)}=\lip_H(u,\Omega)=\lip_H^\ast(u, \Omega).$$
\end{cor}


Next, we apply  the above Rademacher type property {\color{black}to  study    a minimization problem} for    $L^\fz$-functionals corresponding {\color{black}to the above Hamiltonian} $H(x,p)$ in both Euclidean and Carathe\'odory spaces:
$$\cf(u,U):=\big\|H(\cdot, Xu )\big\|_{L^\fz(U)}\  \mbox{for any $u\in  W^{1,\infty}_{X,\loc}(U)$ and  domain $ U\subset \Omega$}.$$
Aronsson \cite{a1,a2,a3,a4} in 1960's  initiated the study in this direction
via introducing  absolute minimizers.  
A  function $u\in  W^{1,\fz}_{X,\loc}(U)$  is called an absolute minimizer  in $U$ for $H$ and $X$
(write $u\in AM(U;H,X)$ for short)  if for any domain $V\Subset U$, it holds that
$${\color{black}\cf(u,V) \le \cf(v,V)}   
\mbox{\ whenever}\ v\in \dot W^{1,\fz}_X(V)\cap C(\overline V)\ \mbox{and} \ u\big|_{\partial V}=v\big|_{\partial V}.$$
{\color{black}Here and throughout this paper,
for domains $A$ and $B$, the notation $A\Subset B$ stands for that   $A$ is a bounded subdomain of  $B$ and its closure $\overline A\subset B$.}

{\color{black}
The existence of absolute minimizers with a given boundary value  has been extensively studied. Apart from the pioneering work by Aronsson mentioned above, we refer the readers to
\cite{acj,bjw,cp,cpp,c03,gxy,j98} and the references therein in the Euclidean setting. For the existence results in Heisenberg groups, Carrot-Carath\'{e}odory spaces and general metric spaces with special type of Hamiltonians, we refer the readers to \cite{b1,j02,ksz,kz,ksz2,wcy}.
Usually, there are two major approaches to obtain the existence of absolute minimizers. When dealing with $C^2$ Hamiltonians, one usually transfers the study of absolute minimizers into the study of viscosity solutions of the Aronsson equation (the Euler-Lagrange equation of the $L^\fz$-functional $\cf$). Thus, to get the existence of absolute minimizers, it suffices to show the existence of the corresponding viscosity solutions. This approach was employed, for instance, in \cite{a3,a68,bjw,c03,gwy,j98,wcy}. To study the the existence of absolute minimizers for Hamiltonians $H(x,p)$ with less regularity, one efficient way is to use Perron's method to first get the existence of absolute minimizing Lipschitz extensions (ALME), and then show the equivalence between ALMEs and absolute minimizers. This idea was adopted in \cite{a1,a2,j02,acj,kz,ksz2,gxy}. To see the close connection between ALMEs and absolute minimizers, we refer the readers to \cite{js,ksz,dmv} and references therein.

 Theorem \ref{rad}, Theorem \ref{radiv} and Corollary \ref{global} allow us  to apply  Perron's method directly and then to establish the  following existence result of absolute minimizers.
 This is  partially motivated by \cite{cpp}.
However, since we are faced with measurable Hamiltonians,
there are  several new barriers to be  overcome as illustrated  at the end of this section.
}

\begin{thm}\label{t231}
Suppose that  $H$  satisfies (H0)-(H3) with $\lz_H=0$.
Given any domain $U\Subset\Oz$ and $g\in \lip_{d_{CC}}(\partial U)$, there must be a function $u\in AM(U;H,X)\cap \lip_{d_{CC}}(\overline U)$
  so that $u|_{\partial U}=g$.
\end{thm}

\medskip
Theorem \ref{rad}, Theorem \ref{radiv}, Corollary \ref{global} and Theorem \ref{t231} improve or extend several previous studies in the literature including Theorem \ref{rade}
and Theorem \ref{radcc} above; see Remark \ref{re1} and  Remark \ref{re2} below.

 \begin{rem} \label{re1}  \rm
 \begin{enumerate}
 \item[ (i)]   In Euclidean spaces, that is,  $X=\{\frac{\partial }{\partial x_i} \}_{1\le i\le n}$, if $H(x,p)=|p|$, then
 Corollary \ref{global} coincides with Lemma \ref{rade'}, which is a consequence of Theorem \ref{rade}.
  In Carnot-Carath\'{e}odory  spaces $(\Omega, X)$, if $H(x,p)=|p|$, then Corollary \ref{global} coincides with Lemma \ref{radcclem}, which is a consequence of Theorem \ref{radcc}.

 \item[ (ii)] In Euclidean spaces,  
if $H(x,p)$ is lower semi-continuous in $U \times \rn$,
   $H(x,\cdot)$ is quasi-convex for each $x \in U$,  and  satisfies (H2) and (H3),
     (ii)$\Leftrightarrow$(i) in Theorem \ref{rad} was proved in
  Champion-De Pascale \cite{cp} (see also \cite{c03,acjs} for convex $H(p)$ in Euclidean spaces).
  The proof in \cite{cp} relies on the lower semi-continuity  in both of $x$ and $p$ heavily, which {\color{black} allows for  approximation $H(x,p)$ via a continuous} Hamiltonian in $x$ and $p$.
     But such an approach fails under the weaker assumptions (H0)\&(H1)  here.
We refer to Section 7 for more details and related further discussions.

 \item[ (iii)]{\color{black}In both Euclidean and Carrot-Carath\'{e}odory} spaces, if $H(x,p)=\sqrt{\langle A(x)p,p\rangle}$, where {\color{black}$A(x)$ is a measurable symmetric matrix-valued function} satisfying uniform ellipticity, then Theorem \ref{rad} was established  in \cite{ksz,ksz2}.
       The proofs therein rely on the inner product structure, and also do  not work here.
For measure spaces {\color{black}endowed with strongly regular} nonlocal Dirichlet energy forms, where the Hamiltonian is given by the square root of Dirichlet form, {\color{black}we refer to} \cite{sp,kz,ksz2,flw} for some corresponding Rademacher type property.

 \item[ (iv)] Under merely (H0)-(H3), one can not expect that
$\lip_H(x)=H(x,Xu(x)) $ almost everywhere.  Recall that  in Euclidean spaces,
there does exist  $A(x)$, which satisfying  Remark \ref{re1}(iii) above, so that such pointwise property fails for the Hamiltonian
$\sqrt {\langle A(x)p,p\rangle}$. For more details see   \cite{s97,ksz}.
\end{enumerate}
       \end{rem}

 \begin{rem}\rm \label{re2}
 \begin{enumerate}
 \item[ (i)] {\color{black}In Euclidean spaces}, that is,   $X=\{\frac{\partial }{\partial x_i}\}_{1\le i\le n}$,
 if $H(x,p)$ is given by  Euclidean norm   and  also any Banach norm,
the existence of  absolute minimizers was established in   Aronsson \cite{a1,a2} and   Aronsson et al \cite{acj}.
If $H(x,p)$ is given by  $\sqrt{\langle A(x)p,p\rangle}$ with $A$ being as in Remark \ref{re1} (iii) above,  existence of  absolute minimizers is given by \cite{ksz} with the aid of
 \cite{js}.
In a similar way, with the aid of \cite{js},
Guo et al \cite{gxy}   also obtained the existence of absolute minimizers
if 
 $H(x,p)$ is a measurable function in $\Omega\times\rn$, and satisfies that $\frac1C< H(x,p)<C$ for all $x\in\Omega$ and $p\in S^{n-1}$, where  $C\ge1$ is a constant, and that $H(x,\eta p) = |\eta| H(x,p)$ for all $x\in \Omega$, $p\in\rn$ and $\eta\in\rr$.

 \item[ (ii)] In Euclidean spaces, if $H(x,p)$ is continuous in both variables $x,p$  and quasi-convex in the variable $p$,  with an additional growth assumption in $p$,
Barron et al \cite{bjw} built up the existence of absolute minimizers.
If $H(x,p)$ is lower semi-continuous in $x,p$ and quasi-convex in $p$ and  satisfies (H1)-(H3), Champion-De Pascale \cite{cp} established the existence of absolute minimizers with the help of their Rademacher type theorem  in Remark \ref{re1}(ii). Recall that   the lower semi-continuity of $H$ plays a key role in \cite{cp} to obtain the pseudo-length property for $d_\lz$.

 \item[ (iii)] In  Heisenberg groups,  if $H(x,p)=\frac12|p|^2$ we refer to \cite{b1} for the existence of absolute minimizers.
In any Carnot group, 
 if  $H\in C^2(\Omega\times \rr^m)$, $D^2_{pp}H(x, \cdot)$ is positive definite,  and
    there exists $ \alpha\ge1$ such that
\begin{equation}\label{e1.2}
H(x,\eta p) = \eta^{\az}H(x,p)\quad\forall x\in \Omega,\ \eta>0,\ p\in\rn,
\end{equation}
then the existence of absolute minimizers was obtained by
Wang   \cite{wcy}    via considering viscosity solutions to the corresponding Aronsson equations.
\end{enumerate}
     \end{rem}

The following  remark    explains that,   
without the assumption (H0), Theorem \ref{rad} does not necessarily hold.
\begin{rem} \label{eg}\rm
 The Hamiltonian $H(x,p)=\lfloor|p|\rfloor$ given in $\eqref{[p]}$  satisfies (H1)-(H3)  but does not satisfy (H0).
Given any $\lz\in(0,1)$,  we have
  \begin{align*}
    d_{\lz} (x,y) & =\sup\{u(y)-u(x) \ | \ u \in \dot W^{1,\fz}_X(\Omega)\ \mbox{with} \  \|H(\cdot,X u)\|_{L^\fz(\Omega)}=\|\lfloor|X u|\rfloor\|_{L^\fz(\Omega)} \le \lz\} \\
     & =\sup\{u(y)-u(x) \ | \ u \in \dot W^{1,\fz}_X(\Omega) \ \mbox{with} \ \| |X u| \|_{L^\fz(\Omega)} \le 1\} \\
     &=d_{CC}(x,y)\quad\forall x,y\in\Omega.
  \end{align*}
Fix any $ z\in\Omega$ and write $u(x)=d_\lz(z,x)\ \forall x\in\Omega$. By the triangle inequality we have
$$u(y)-u(x)=d_\lz(z,y)-d_\lz(z,x)\le d_\lz(x,y)$$
Recall that, when $X=\{\frac{\partial}{\partial x_j}\}_{1\le j\le n}$   or  when  $X$ is given by Carnot type H\"ormander vector fields,
one always has $|X d_{CC} (z,\cdot) | =1$ almost everywhere; see  \cite{monti}. For such $X$, we conclude that
 $$\|H(\cdot,X u)\|_{L^\fz(\Omega)}=1>\lz.$$  Thus Theorem \ref{rad} fails.
\end{rem}


The following  remark   explains the reasons why we need $\lz\ge \lz_H $   in Theorem \ref{radiv}, and why we assume $\lz_H=0$ in Theorem \ref{t231}.
Note that,
in Theorem \ref{rad} where $\lz_H$ maybe not $0$, we do get the equivalence among (i), (ii) and (iii)  for any $\lz\ge 0$.

\begin{rem}\label{rlambda}\rm
\begin{enumerate} {\color{black}
 \item[ (i)]
  To prove  (iv) in Theorem \ref{radiv}  $\Rightarrow$ (i) in Theorem \ref{rad},
 we need a pseudo-length property for $d_\lz$ as in Proposition \ref{len}.
  When $\lz>\lz_H$ (equivalently, $R'_\lz>0$),  to get such pseudo-length property for $d_\lz$,  our proof does use $R_\lz'>0$  so to guarantee that the topology induced by $\{d_\lz(x ,\cdot)\}_{x \in \Omega}$
 (See Definition \ref{pseudo}) is the same as the Euclidean topology; see Remark \ref{exp1}.
 When $\lz=\lz_H$,  we get such pseudo-length property for $d_{\lz_H}$ via approximating by $d_{\lz_H+\ez}$ with sufficiently small $\ez>0$.

 \item[ (ii)]
 When  $ \lz_H>0$ and $0\le\lz<\lz_H$, we do not know whether $d_\lz$ enjoys such pseudo-length property.
We remark that there does exist some Hamiltonian $H(x,p)$ which satisfies
the assumptions (H0)-(H3) with $\lz_H>0$ (that is,  for some $\lz>0$,
 $R_\lz'=0$) ; but for $0<\lz<\lz_H$,  the topology induced by $\{d_\lz(x ,\cdot)\}_{x \in \Omega}$ does not coincide with the Euclidean topology; see Remark \ref{d1} (ii).

 \item[ (iii)] To get the existence of absolute minimizers,
 our approach does need Theorem \ref{radiv} and also several properties of $d^U_\lz$, whose proof relies heavily on the  pseudo-length property for $d_\lz$ and $R_\lz'>0$.
 In Theorem 1.6, we assume $\lz_H=0$ so that we can work
  with  all Lipschitz boundary  $g$ so to get existence of absolute minimizer.

In the case $\lz_H>0$, our approach  will give the existence of of absolute minimizer
when the boundary   $g:\partial U\to\rr$ satisfies $\mu(g,\partial
 U)>\lz_H$, but does not work when   $\mu(g,\partial
 U)\le  \lz_H$.
Here  $ \mu(g,\partial  U)$ is the  infimum of $\lz$ so that $g(y)-d(x)\le d^U_\lz(x,y)$ $forall x,y\in\pa U$.
 }
\end{enumerate}
\end{rem}

{\color{black}
The paper is organized as below,
where we also clarify the ideas and  main novelties to prove Theorem \ref{rad}, Theorem \ref{radiv} and also Theorem \ref{t231}.
We emphasize that in our results from Section 2 - Section 6, $X$ is a fixed smooth vector field in a domain $\Oz$  and
  satisfies the H\"ormander condition;  and that
  the Hamiltonian $H(x,p)$  always enjoys (H0)-(H3).
  In all results from Section 5 and 6, we  further assume $\lz_H=0$.      }

In Section 2, we state several facts about the analysis and geometry in Carnot-Carath\'{e}odory spaces employed in the proof.



In Section 3, we prove (i)$\Leftrightarrow$(ii)$\Leftrightarrow$(iii) in Theorem \ref{rad}. 
Since (i)$\Rightarrow$(ii)  follows from the definition and that
 (ii)$\Rightarrow $(iii) is obvious, it suffices to    prove (iii)$\Rightarrow $(i).
To this end, we borrow some ideas from \cite{sk,sp,flw,ksz}, which were designed for  nonlocal Dirichlet energy forms originally.

The key is that,
 by employing assumptions (H0), (H1) and Mazur's theorem, we  are able to prove that if $v_j\in\dot W^{1,\fz}_X(\Omega)$ with $\|H(\cdot, Xv_j)\|_{L^\fz(\Omega)}\le\lz$
  and $v_j\to v$ in $C^0(\Omega)$ as $j\to\fz$, then $v \in\dot W^{1,\fz}_X(\Omega)$
  with $\|H(\cdot, Xv)\|_{L^\fz(\Omega)}\le\lz$; see  Lemma \ref{l6.1} for details.
Thanks to this,
 choosing a suitable sequence of approximation functions via the definition of $d_\lz$,  we  then show that
 $$\mbox{$ d _{\lz}(x,\cdot) \in  W_{X}^{1,\fz}(\Omega)$ and $\|H(\cdot,Xd _{\lz}(x,\cdot))\|_{L^\fz(\Omega)}\le\lz$ for all $\lz>0$ and $x\in \Omega$}.$$
See Lemma \ref{lem3.2}.
Given any $u$ satisfying (iii),  we construct  approximation functions $u_j$ from $d_\lz$
and use Lemma \ref{l6.2} to show  $H(x,Du_j)\le \lz$. That is (i) holds.


In Section 4, we prove (iii) in Theorem \ref{rad} $\Leftrightarrow $ (iv) in Theorem \ref{radiv}. Since (iii)$\Rightarrow $(iv) is obvious, it suffices to show (iv)$\Rightarrow$(iii).
This  follows from the pseudo-length property  of pseudo metric $d_\lz$ in Proposition \ref{len}.
To get such a length property we find some special functions
which fulfill the assumption {\color{black}Theorem \ref{rad}(iii)}, and hence
we  show that the pseudo metric $d_\lz$ has a pseudo-length property.

In Section 5, we introduce  McShane extensions and
 minimizers, and  then gather several properties of them
 and pseudo-distance  in Lemma \ref{property} to Lemma \ref{lem5.10}, which are
required in Section 6.
These properties also have their own interests.

 Given any domain $U\Subset\Oz$,
 via the intrinsic distance $d_\lz^U$ induced from $U$, we introduce
  McShane extensions $\csgv^\pm$ of  any $g\in \lip_{d_{CC}}(\partial U)$ in $U$.
There are several reasons to use   $d^U_\lz$  other than $d_\lz$, for example,
 $d^U_\lz$ has {\color{black}the pseudo-length property in
 $U$} but the restriction of $d_\lz$ may not have; moreover, Theorem \ref{rad},
  and Theorem \ref{radiv} holds if $(\Omega,d_\lz)$ therein is
 replaced by $ (U,d^U_\lz)$, but not necessarily hold if $(\Omega,d_\lz)$ therein is
 replaced by $ (U,d _\lz)$.
 However,  the use of $d_\lz^U$  causes several difficulties.  For example,
 $d_\lz^U$ may be infinity  when  extended to $\overline U$.
 This  makes it quite implicit to see the continuity of McShane extensions around $\pa U$ from the definition. In Lemma \ref{cont}, we get such continuity by analyzing the behaviour of  $d_\lz^U$ near $\pa U$.  Moreover, as required in Section 6, we have  to study the relations between    $d_\lz^U$ and $d^V_\lz$ for subdomains $V$ of $U$ in Lemma \ref{ll} and Lemma  \ref{lem5.9}.

In Section 6, we prove Theorem \ref{t231} in a constructive way by using above Rademacher type property and Perron's approach, where we borrow some ideas from  \cite{bjw,cp,cpp}.

 The proof consists of   crucial Lemma \ref{lem234}, Lemma \ref{lem235} and Proposition \ref{p42}.    Lemma \ref{lem234}  says that McShane extensions $\csgu^\pm$ in $U$ of   function $g$ in $\pa U$ are local super/sub minimizers in $U$.
Since $\csgu^\pm$ are the maximum/minimum  minimizers,
the proof of  Lemma \ref{lem234} is reduced to showing that for any subdomain $V \subset U$, the McShane extensions $\cS^\pm_{h^\pm, V}$ in $V$ with boundary   $h^\pm=\csgu^\pm|_{\pa V}$ satisfy
$$\mbox{$\cS^\pm_{h^\pm; V}(y) - \cS^\pm_{h^\pm; V}(x) \le d_\lz^U(x,y)$ for all $x,y \in \overline V$};$$  see Lemma \ref{lem5.7} and Lemma \ref{lem5.10} and the proof of Lemma \ref{lem234}.
However,  since Lemma \ref{cont} only gives
$$\mbox{$\cS^\pm_{h^\pm; V}(y) - \cS^\pm_{h^\pm; V}(x) \le d_\lz^V(x,y)$ for all $x,y \in \overline V$,}$$
  we must improve $d_\lz^V(x,y)$ here to the smaller one $d_\lz^U(x,y)$.  To this end, we show $d_\lz^V= d_\lz^U$ locally in Lemma \ref{ll}, and also use the pseudo-length property of $ d_\lz^U$ heavily.
 Lemma \ref{lem235} says that a function which   is both of
local superminimizers and subminimizer must be an absolute minizimzer.
To get the required local minimizing property, we use McShane extensions  to construct approximation functions
and also need the fact $d_\lz^V= d_\lz^{V\setminus\{x_i\}_{1\le i\le m}}$ in $\overline V\times \overline V$ as
in Lemma \ref{lem5.9}.
Proposition \ref{p42}  says that
   the supremum of local subminimizers are absolute minimizers.
Due to  Lemma \ref{lem235}, it suffices  to prove the local super/sub minimizing property
of such a supremum. We do prove this via using Lemma \ref{lem234} and  Lemma \ref{lem5.10} repeatedly and a contradiction argument.

In Section 7, we aim at explaining some obstacles in using previous approach to establish the Rademacher type theorem and the existence of absolute minimizers.
Indeed, in the literature to study Hamiltonians with better regularity or homogeneity, for instance, \cite{cp,gxy}, another intrinsic distance $\bar d_\lz$ is more common used in those proof which is hard to fit our setting.

In Appendix, we revisit the Rademacher's theorem in the Euclidean space to show that Theorem \ref{rad} and Corollary \ref{global} are indeed an extension of the Rademacher's theorem.

{\color{black}
\begin{center}

\textsc{Acknowledgement
}
\end{center}
The authors would like to thank Katrin F\"{a}ssler for the reference \cite{ls16,l16} and her comment on this paper. The authors are also grateful to the anonymous referee for many valuable suggestions which make the paper more self-contained and improve the final presentation of the paper.}

\section{Preliminaries}

In this section, we introduce the background and some known results related to Carnot-Carath\'{e}odory  spaces employed in the proof.

Let $X:=\{X_1,  ... ,X_m\}$  for some $m\le n$ be  a family  of smooth vector fields  in $\Oz$
which satisfies the H\"ormander condition, that is,   there is a step $k \ge 1$ such that, at  each point,
  $\{X_i\}^{m}_{i=1}$ and all their commutators up to at most order $k$ generate the whole $\rn$.  Then for each $i=1,\cdots,m$, $X_i$ can be written as
$$  X_i = \sum_{l=1}^n b_{il} \frac{\pa}{\pa x_l}  \mbox{ in $\Omega$}  $$
with $b_{il} \in C^\fz(\Omega)$ for all $i=1,\cdots,m$ and $l=1,\cdots,n$.

Define the
 Carnot-Carath\'{e}odory distance corresponding to $X$ by
\begin{equation} \label{dcc0}
  d_{CC}(x,y):=\inf\{\ell_\dcc(\gz) \ | \ \mbox{$\gz\in \mathcal {ACH}(0,1;x,y;\Omega)$}\}
\end{equation}
Here and below, we write $\gz\in \mathcal {ACH}(0,1;x,y;\Omega)$  if $\gz:[0,1]\to\Omega$ is    absolutely continuous, 
$\gz(0)=x,\gz(1)=y$, and there exists measurable functions $c_{i}:[0,1] \to \rr$ with $1 \le i \le m$ such that $\dot\gz(t) = \sum_{i=1}^m c_i(t)X_i(\gz(t))$ whenever $\dot\gz(t)$ exists.
The length of $\gz$ is $$\ell_\dcc(\gz):=\int_0^1|\dot\gz(t)|\,dt=\int_0^1 \sqrt{\sum_{i=1}^m c_i^2(t)}\,dt.$$ In the Euclidean case, we have the following remark.
\begin{rem}\label{dccde} \rm In  Euclidean case, that is, $X=\{\frac{\partial }{\partial {x_i}}\}_{1\le i\le n}$, one has
$d_{CC}$ coincides with the intrinsic distance $d_E^\Omega$ as given in (A.2).
In particular, $d_{CC}(x,y)=d^\Omega_E(x,y)$ for all $ x,y\in\Omega$ with $|x-y|<\dist(x,\partial\Omega)$. When $\Omega$ is convex, one further have
and $d_{CC}(x,y)=|x-y|$ for all $x,y\in\Omega$; however, when $\Omega$ is not convex, this    is not necessarily true.
See  Lemma A.4 in Appendix for more details.
\end{rem}

Since $X$ is a H\"ormander vector field in $\Omega$,
for any compact set $K \subset \Omega$, there exists a constant $C(K) \ge 1$ such that
\begin{equation*}
  C(K)^{-1} |x - y | \le \dcc(x,y) \le C(K) |x - y |^{\frac{1}{k}} \quad   {\color{black}\forall x, y \in K},
\end{equation*}
see for example \cite{nsw85} and \cite[Chapter 11]{HK}.
This shows that the topology induces by  $(\Omega,d_{CC})$ is exactly {\color{black}the Euclidean topology}.
%

Given a function $u \in L^1_{\loc}(\Omega)$, its distributional derivative along $X_i$ is defined by the identity
$$   \la X_i u , \phi \ra = \int_\Omega u X_i^\ast \phi \, dx \mbox{ for all $\phi \in C_0^\fz(\Omega)$,}  $$
where $X_i^\ast = -\sum_{l=1}^n \frac{\pa}{\pa x_l}(b_{il} \cdot)$ denotes the formal adjoint of $X_i$. Write $X^\ast= (X_1^\ast, \cdots,X_m^\ast)$. We call $Xu:=(X_1u, \cdots,X_m u)$ the horizontal distributional derivative for $u \in L^1_{\loc}(\Omega)$ and the norm $|Xu|$ is defined by
$$ |Xu|= \sqrt{\sum_{i=1}^m|X_i u|^2}  .$$
For $1\le p\le \fz$, denote by $\dot W^{1,p}_X(\Omega)$ the $p$-th integrable horizontal Sobolev space, that is,
the collection of all functions $u\in L^1_\loc(\Omega)$ with its distributional derivative $Xu\in L^p(\Omega)$. Equip $\dot W^{1,p}_X(\Omega)$ with the semi-norm
$\|u\|_{\dot W^{1,p}_X(\Omega)}=\||Xu|\|_{L^p(\Omega)}$.


%

The following was proved in  \cite[Lemma 3.5 (II)]{gn96}.
\begin{lem}\label{uplus}
 If  $ u \in \dot W^{1,p}_X(U)$ with $1\le p<\fz$ and
 $ U\Subset \Omega$, then  $u^+=\max\{u,0\} \in \dot W^{1,p}_X(U)$ with $Xu^+=(Xu)\chi_{\{x\in U|u>0\}}$ almost everywhere.
\end{lem}

 We recall   the following  imbedding of  horizontal Sobolev spaces from  \cite[Theorem 1.4]{gn}. For any set $U\subset\Omega$, the Lipschitz class
  $\lip_{d_{CC}}(U)$ is   the collection of all functions $u:U\to\rr$ with   its seminorm
 $$\lip_{d_{CC}}(u,U)  := \sup_{x \ne y, x,y \in U} \frac{|u(x)-u(y)|}{\dcc(x,y) } < \fz.   $$

 \begin{thm}  \label{holder}
  For any subdomain $U \Subset \Omega$, if $u \in \dot W^{1,\fz}_X(U)$,   then there is a continuous function $\wz u \in \lip_{d_{CC}}(U)$ with $\wz u=u$ almost everywhere and
  $$\lip_{d_{CC}}(\wz u,U)\le C(U,\Omega)  [\|u\|_{L^\fz(U)}+\|u\|_{\dot W^{1,\fz}_X(U)}].$$
 \end{thm}

\begin{rem}\label{conrep}\rm
For any $u \in\dot W^{1,\fz}_X(U)$,  we call  above $\wz u$ given in Theorem \ref{holder} as the continuous representative of $u$. Up to considering $\wz u $, in this paper  we always   assume that $u$ itself is continuous.
\end{rem}





We have the following dual formula of $\dcc$.
\begin{lem} \label{jerison}
For any $x,y\in \Omega$, we have
     \begin{equation} \label{dcc}
  d_{CC}(x,y) =\sup\{u(y)-u(x) \ | \ u\in  \dot W^{1,\fz}_X(\Omega) \mbox{ with 
  $ \||Xu|\|_{L^\fz(\Omega)}\le1$}\}.
\end{equation}
\end{lem}

To prove this we need the following bound for the norm of horizontal derivative of smooth
 approximation of functions in $\dot W^{1,\fz}_{X}(\Omega)$, see for example  \cite[Proposition 11.10]{HK}.
Denote by $\{\eta_\ez\}_{\ez\in(0,1)}$ the standard smooth mollifier, that is,
$\eta_\ez(x) = \ez^{-n} \eta(\frac{x}{\ez}) \quad\forall x\in\rn$,
where $\eta\in C^\fz(\rn)$ is supported in unit ball of  $\rn$ (with Euclidean distance),
  $\eta \ge 0$ and $\int_{\rn} \eta\,dx=1$.

\begin{prop}\label{mollif2}
Given  any compact set $K\subset \Omega$,  there is $\ez_K\in(0,1)$ such that
for any $\ez<\ez_K$ and $u \in \dot W^{1,\fz}_{X}(\Omega)$ one has
\begin{equation}\label{hk1}
  |X(u\ast \eta_\ez)(x)| \le \||Xu|\|_{L^\fz(\Omega)}
 +A_\ez(u)\quad\forall x\in K,
\end{equation}
where  $A_\ez(u)\ge0$ and $\lim_{\ez \to 0}A_\ez(u) \to 0$  in $K$.
\end{prop}
%

\begin{proof}[Proof of Lemma \ref{jerison}]
Recall that it was shown by \cite[Proposition 3.1]{js87}   that
\begin{equation}\label{dcc2}
  d_{CC}(x,y) =\sup\{u(y)-u(x) \ | \ u\in  C^\fz(\Omega) \mbox{ with
  $ \||Xu|\|_{L^\fz(\Omega)}\le1$}\} \quad\forall x,y\in \Omega.
\end{equation}
It then suffices to show that for any $u\in  \dot W^{1,\fz}_X(\Omega) $ with
  $ \||Xu|\|_{L^\fz(\Omega)}\le1$, we have
 $$\mbox{$ u(y)-u(x)\le d_{CC}(x,y) \quad\forall x,y\in\Omega$.}$$   Note that $u$ is assumed to be continuous as in Remark \ref{conrep}.

To this end, given any $x,y\in\Omega$,   for any $ \ez>0$
there exists a curve $ \gz_\ez  \subset \mathcal {ACH}(0,1;x,y;\Omega)$ such that
$
  \ell_{d_{CC}}(\gz_\ez)\le (1+\ez)d_{CC}(x,y).
$
We can find a   domain $U\Subset \Omega$ such that $\gz_\ez  \subset U$.
 It is standard that $u\ast \eta_t\to u$ uniformly in $\overline U$ and hence
   $$u(y)-u(x) =\lim_{t\to 0}[u\ast \eta_t(y)-u\ast \eta_t(x)] . $$

 Next,  by Proposition \ref{mollif2},
 for   $0<t<t_{\overline U}$ one has
 $$|X(u\ast \eta_t)(z)| \le \||Xu|\|_{L^\fz(\Omega)}
 +A_t u(z)\quad\forall z\in \overline U ,$$
 and  moreover, $A_t u(z) \to 0$ uniformly in $\overline U $    as $t  \to 0$.
Obviously,
 we  can find $t_{\ez,\overline U}<t_{\overline U}$
  such that for any $0<t<t_{\ez,\overline U}$,
   we have $A_t u(x)\le \ez$ and hence,  by $\||Xu|\|_{L^\fz(\Omega)}\le1$,
  $|X(u\ast \eta_t)(z)| \le 1+\ez ,$
    for all $z\in\overline U$. Therefore
\begin{align*}u\ast \eta_t(y)-u\ast \eta_t(x)&=\int_0^1 [(u\ast \eta_t)\circ\gz_t]'(s)\,ds\\
&=  \int_0^1  X(u\ast \eta_t)(\gz_t(s))\cdot \dot \gz_t(s)\,ds\\
&\le (1+\ez)\ell_{d_{CC}}(\gz_\ez)\\
&\le  (1+\ez)(1+\ez)d_{CC}(x,y).
\end{align*}
Sending $t\to0$ and $\ez\to0$, one concludes
 $ u(y)-u(x)\le d_{CC}(x,y) $ as desired.
\end{proof}

As a consequence of Rademacher type theorem (that is, Theorem  \ref{radcc}),
we  have the following, which is an analogue of Lemma \ref{rade'}.  Denote by $\lip^\ast_{d_{CC}}(\Omega)$ the  collection of all functions  $u$ in $\Omega$
  with
\begin{equation}\label{suplip1}\lip^\ast_{d_{CC}}(u,\Omega):=\sup_{x \in\Omega }\lip_{d_{CC}} u(x)<\fz.
\end{equation}
\begin{lem}\label{radcclem} We have $
\dot W^{1,\fz}_X(\Omega)=\lip_{d_{CC}}(\Omega)=\lip^\ast_{d_{CC}}(u,\Omega)$ with
\begin{equation}\label{1.2}
 \||X u|\|_{L^\fz(\Omega)} = \lip_{d_{CC}}( u,\Omega)= \lip^\ast_{d_{CC}}( u,\Omega)
\end{equation}
\end{lem}

\begin{proof}
First, we show $\lip_{d_{CC}}(\Omega)= \lip^\ast_{d_{CC}}(\Omega)$ and $\lip_{d_{CC}}( u,\Omega)= \lip^\ast_{d_{CC}}( u,\Omega)$. Notice that $\lip_{d_{CC}}(u,\Omega)\subset \lip^\ast_{d_{CC}}(u,\Omega)$ and $\lip^\ast_{d_{CC}}(u,\Omega) \le \lip_{d_{CC}}(u,\Omega)$ are obvious. We prove
$$\mbox{$\lip^\ast_{d_{CC}}(u,\Omega) \subset \lip_{d_{CC}}(u,\Omega)$  and $\lip_{d_{CC}}(u,\Omega) \le \lip_{d_{CC}}^\ast(u,\Omega).$}$$

Let $u  \in \lip^\ast_{d_{CC}}(u,\Omega)$. Given any $x,y\in\Omega$, and $\gz\in \mathcal {ACH}(0,1;x,y;\Omega)$, parameterise $\gz$ such that $|\dot\gz(t)| = \ell_{\dcc}(\gz)$ for almost every $t \in [0,1]$.
Since $$A_{x,y}:=\sup_{t\in[0,1]} \lip u(\gz(t))<\fz,$$
 for each $t\in[0,1]$ we can find $r_t>0$ such that
 $$\mbox{$|u(\gz(s))-u(\gz(t))|\le A_{x,y}|\gz(s)-\gz(t)|=A_{x,y}\ell_{\dcc}(\gz)|s-t|$ whenever $|s-t|\le r_t$ and $s \in [0,1]$.}$$
 Since $[0,1]\subset\cup_{t\in[0,1]}(t-r_t,t+r_t)$, we can find an increasing sequence
  $t_i\in[0,1]$ with $t_0 = 0$ and $t_N=1$ such that $$[0,1]\subset \cup_{i=1}^N(t_i-\frac12r_{t_i},t_i+\frac12 r_{t_i}).$$ Write $x_i=\gz(t_i)$ for $i =0 , \cdots , N.$ We have
 \begin{align*}|u(x)-u(y)|&=|\sum_{i=0}^{N-1}[u(x_i)-u(x_{i+1})]|\\
&\le
 \sum_{i=0}^{N-1}|u(x_i)-u(x_{i+1})| \\&\le A_{x,y}\ell_{\dcc}(\gz)\sum_{i=0}^{N-1}|t_i-t_{i+1}|\\&=A_{x,y}\ell_{\dcc}(\gz).
\end{align*}
 Noticing that $A_{x,y}  \le  \lip^\ast(u,\Omega) < \fz$ for all $x,y \in \Omega$, we deduce that
 \begin{equation}\label{dcc3}
   |u(y)-u(x)|\le \lip_{d_{CC}}^\ast(u,\Omega)\ell_{d_{CC}}(\gz) \quad \forall x,y \in \Omega.
 \end{equation}
For any $\ez>0$, recalling the definition of $\dcc$ in \eqref{dcc0}, there exists $\{\gz_\ez\}_{\ez >0} \subset \mathcal {ACH}(0,1;x,y;\Omega)$ such that
\begin{equation}\label{dcc4}
  \ell_{d_{CC}}(\gz_\ez)\le (1+\ez)d_{CC}(x,y) \quad \forall x,y \in \Omega.
\end{equation}
Combining \eqref{dcc3} and \eqref{dcc4}, we have
$$ \frac{|u(y)-u(x)|}{d_{CC}(x,y)} \le \lim_{\ez \to 0} \frac{|u(y)-u(x)|}{(1+\ez)d_{CC}(x,y)} \le \lip_{d_{CC}}^\ast(u,\Omega) \quad \forall x,y \in \Omega.$$
Taking supremum among all  $x,y \in \Omega$ in the above inequality, we deduce that $u \in \lip_{d_{CC}}(u,\Omega)$ and $\lip_{d_{CC}}(u,\Omega) \le \lip_{d_{CC}}^\ast(u,\Omega)$. Hence the second equality in \eqref{1.2} holds.

Next, we show $\dot W^{1,\fz}_X(\Omega)=\lip_{d_{CC}}(\Omega)$ and $\lip_{d_{CC}}(u,\Omega) = \| | Xu| \|_{L^\fz(\Omega)}$.
By Theorem \ref{radcc}, we know $\lip_{\dcc}(\Omega) \subset \dot W^{1,\fz}_X (\Omega)$ and $\| | Xu| \|_{L^\fz(\Omega)} \le\lip_{d_{CC}}(u,\Omega)$.

  To see $\dot W^{1,\fz}_X (\Omega) \subset \lip_{\dcc}(\Omega)$ and $\lip_{d_{CC}}(u,\Omega) \le \| | Xu| \|_{L^\fz(\Omega)}$, let $u \in \dot W^{1,\fz}_X (\Omega)$.  Then $\||Xu|\|_{L^\fz(\Omega)} =:\lz < \fz.$ If $ \lz >0$, then $\lz^{-1}u \in \dot W^{1,\fz}_X (\Omega)$ and $\||X (\lz^{-1} u )|\|_{L^\fz(\Omega)} =1$. Hence $\lz^{-1}u$ could be the test function in \eqref{dcc}, which implies
  $$    \lz^{-1} u(y)- \lz^{-1}u(x) \le  \dcc (x,y) \ \forall x,y \in \Omega,  $$
 or equivalently,
 $$  \frac{|u(y)- u(x) |}{\||Xu|\|_{L^\fz(\Omega)}} \le  \dcc (x,y) \ \forall x,y \in \Omega.  $$
 Therefore, $u \in \lip_{\dcc}(\Omega)$ and $ \lip_{\dcc} (u,\Omega) \le \||Xu|\|_{L^\fz(\Omega)}$. If $\lz=0$, then similar as the above discussion, we have for any $\lz'>0$
 $$  \frac{|u(y)- u(x) |}{\lz'} \le  \dcc (x,y) \ \forall x,y \in \Omega.  $$
 Therefore, $u \in \lip_{\dcc}(\Omega)$ and $ \lip_{\dcc} (u,\Omega) \le \lz'$ for any $\lz'>0$. Hence $ \lip_{\dcc} (u,\Omega) =0 = \||Xu|\|_{L^\fz(\Omega)}$ and  we complete the proof.
\end{proof}

Next, we recall some concepts from metric geometry. First we recall the notion of pseudo-distance.

\begin{defn}
\label{pseudo}
We say that $\rho$ is a pseudo-distance in a set $\Omega \subset \rn$ if $\rho $ is a function in $\Omega \times \Omega$ such that
\begin{enumerate}
 \vspace{-0.3cm}  \item[(i)]  $\rho(x,x)=0$ for all $x \in \Omega$ 
 and
 $\rho(x,y) \ge 0$ for all $x,y \in  \Omega$;
  \vspace{-0.3cm} \item[(ii)] $\rho(x,z) \le \rho(x,y) + \rho(y,z)$ for all $x,y,z \in  \Omega$.
\end{enumerate}
We call $(\Omega,\rho)$ as a pseudo-metric space. The topology induced by $\{\rho(x, \cdot)\}_{x \in \Omega}$ (resp. $\{\rho(\cdot, x)\}_{x \in \Omega}$) is the weakest topology on $\Omega$ such that $\rho(x, \cdot)$ (resp. $\rho(\cdot, x)$)
is continuous for all $x \in \Omega$.
\end{defn}

We remark that since the above pseudo-distance $\rho$ may not have symmetry, the topology induced by $\{\rho(x, \cdot)\}_{x \in \Omega}$ in $\Omega$ may be different from that induced by $\{\rho(\cdot, x)\}_{x \in \Omega}$.

Suppose that $H(x,p)$ is an Hamiltonian in $\Omega$ satisfying  assumptions (H0)-(H3).
 Let $\{d_\lz\}_{\lz \ge 0}$ be defined as in \eqref{d311}. Thanks to the convention in Remark \ref{conrep}, one has
  \begin{equation} \label{dlz}
d _{\lambda} (x,y):= \sup\{u(y)-u(x) \ | \ u \in \dot W^{1,\infty}_{X }(\Omega) \ \mbox{with}\ \|H(\cdot, Xu)\|_{L^\infty(\Omega)} \le \lz
  \}\quad\forall \mbox{$x,y \in \Omega$}.
\end{equation}

The following properties holds for $\dl$.

\begin{lem} \label{lem311}
The following holds.
  \begin{enumerate}
    \item[(i)] For any $\lz\ge0$, $\dl$ is  a pseudo distance on $\Omega$.
    \item[(ii)]  For any $\lz\ge0$,
    \begin{eqnarray}\label{e311}
 R'_\lz\dcc(x,y) \leq \dl (x,y) \leq   R_\lz\dcc(x,y)<\fz\quad\forall x,y\in\Omega.
\end{eqnarray}
    \item[(iii)] If $H(x,p)=H(x,-p)$ for all $p \in \rr^m$ and almost all $x \in \Omega$, then
     $d_\lz(x,y)=d_\lz(y,x)$ for all $x, y\in \Omega$.
  \end{enumerate}
\end{lem}

\begin{proof}
  To see Lemma \ref{lem311} (i),
by choosing constant functions as test functions in \eqref{dlz}, one has  $\rho(x,y)\ge0$ for all $x,y\in\Omega$.  Obviously, one has $d_\lz(x,x)=0$ for all $x\in\Omega$.
Besides, 
\begin{align*}
 d _{\lambda} (x,z)&= \sup\{u(z)-u(x) : u \in \dot W^{1,\infty}_{X }(\Omega),\ \|H(\cdot, Xu)\|_{L^\infty(\Omega)} \le \lz
  \}   \\
   & \le \sup\{u(y)-u(x) : u \in \dot W^{1,\infty}_{X }(\Omega),\ \|H(\cdot, Xu)\|_{L^\infty(\Omega)} \le \lz
  \} \\
  & \quad + \sup\{u(z)-u(y) : u \in \dot W^{1,\infty}_{X }(\Omega),\ \|H(\cdot, Xu)\|_{L^\infty(\Omega)} \le \lz
  \}  \\
  & = d _{\lambda} (x,y) + d _{\lambda} (y,z).
\end{align*}
By Definition \ref{pseudo}, $\dl$ is a pseudo distance.

To see Lemma \ref{lem311} (ii), by (H3), we have
\begin{align*}
  \{u\in \dot W^{1,\infty}_{X }(\Omega ) \ | \  \||Xu|\|_{L^\fz(\Omega )} \le R'_\lz\} &\subset \{u\in \dot W^{1,\infty}_{X }(\Omega ) \ | \ \|H(x,Xu)\|_{L^\fz(\Omega )} \le \lz\} \\
  & \subset \{u\in \dot W^{1,\infty}_{X }(\Omega ) \ |  \ \||Xu|\|_{L^\fz(\Omega )} \le R_\lz\}.
\end{align*}
From this and the definitions of $\dcc$ in \eqref{dcc} and $\dl$ in \eqref{dlz},
we deduce  \eqref{e311} as desired.

Finally  we show  Lemma \ref{lem311} (iii),  since $H(x,p)=H(x,-p)$ for all $p \in \rr^m$ and almost all $x \in \Omega$, then
$$ \|H(x,Xu)\|_{L^\fz(\Omega )} = \|H(x,X(-u))\|_{L^\fz(\Omega )} \mbox{ for all $u \in \dot W^{1,\infty}_{X }(\Omega ) $}.$$
Hence for any $x, y\in \Omega$, $u$ can be a test function for $d_\lz(x,y)$ in the right hand side of \eqref{dlz} if and only if $-u$ can be a test function for $d_\lz(y,x)$ in the right hand side of \eqref{dlz}. As a result, $d_\lz(x,y)=d_\lz(y,x)$, which completes the proof.
\end{proof}

As a consequence of Lemma \ref{lem311}, we obtain the following.

\begin{cor} \label{lem311-1}
 For any $\lz> \lz_H$, $d_\lz$ is comparable with $d_{CC}$, that is ,
\begin{eqnarray}\label{e311-1}
 0<R'_\lz \leq \frac{\dl (x,y)} {\dcc(x,y)} \leq   R_\lz <\fz\quad\forall x,y\in\Omega.
\end{eqnarray}
Consequently, the topology induced by $\{d_\lz(x ,\cdot)\}_{x \in \Omega}$ and $\{d_\lz(\cdot,x)\}_{x \in \Omega}$ coincides with the one induced by
   $\dcc$ in $\Omega$, and hence, is the Euclidean topology.
\end{cor}

\begin{rem} \label{d1}
\rm
(i) We remark that in  Lemma \ref{lem311} (iii),  without  the assumption $H(x,p)=H(x,-p)$ for all $p \in \rr^m$ and almost all $x \in \Omega$, $d_\lz$ may not be symmetric, that is, $d_\lz(x,y)=d_\lz(y,x)$ may not hold for all $x, y\in \Omega$.

(ii)   If $R'_\lz=0$ for some $\lz>0$, then the topology induced by $\{\dl(x,\cdot)\}_{x \in  \Omega}$ may be different from the Euclidean topology.
To wit this, we construct an Hamiltonian $H(p)$ in Euclidean disk $\Omega=\{x\in\rr^2||x|<1\}$  with $X=\{\frac{\partial}{\partial x_1},
\frac{\partial}{\partial x_2} \}$, which  satisfies (H0)-(H3) with $\lz_H>0$.
Define $H: \Omega \times \rr^2 \to [0,\fz)$ by
$$ H(p) = H(p_1,p_2) =\max\{ |p| ,2 \}\chi_{\{p \in \rr^2 | p_1<0\}} + |p|\chi_{\{p \in \rr^2 | p_1 \ge 0\}}$$
where $\chi_E$ is the characteristic function of the set $E$. One can check that $H(p)$ satisfies (H0)-(H3).  We omit the details.

Now we show
$$ R_{1}' = \inf\{ |p| \ | \  H(p) \ge 1  \}=0 ,$$
and thus $\lz_H \ge 1 >0$.
Indeed,  for any $ p=(p_1,p_2)$ and $p_1\ge 0$, one has $H(p )=|p|$ and hence $H(p)\ge 1$ implies $|p|\ge1$.  On the other hand,  for any $ p=(p_1,p_2)$ and $p_1< 0$,
we always have $H(p)=\max\{|p|,2\}\ge2$, and hence
$$ R_{1}' = \inf\{ |p|   | \   p=(p_1,p_2)\in\rr^2 \ \mbox{with $ |p|<1$ and $ p_1<0$} \} =0.$$

%

Writing $e_1=(1,0)$, we claim that
\begin{equation}\label{a1}
  d_1(x,x+ae_1) = 0 \quad \forall x, x+ae_1\in\Omega\ \mbox{with}\ a\in (-1,0] .
\end{equation}
This claim implies that the topology induced by $\{d_1(x,\cdot)\}_{x\in\Omega}$ is different with the Euclidean topology. 

To see the  claim \eqref{a1},
writing $o=(0,0)$, we only need to show that
\begin{equation}\label{a1-1}
  d_1(o,ae_1) = 0 \quad \forall a\in (-1,0).
\end{equation}
It then suffices to show that $
  u(ae_1)- u(0)   \le  0
$ for all $u \in W^{1,\fz}(\Omega) \mbox{ with } \|H(\nabla u)\|_{L^\fz(\Omega)} \le 1$.
 Given such a function $u$, observe that $ \|H(\nabla u)\|_{L^\fz(\Omega)} \le 1$ implies
 $\frac{\partial  u}{\partial x_1}(x) \ge 0 $ and  $|\nabla u(x)|\le 1$ for almost all   $x \in \rr^2$.
Let $\{\gz_\dz\}_{0\le \dz \ll 1}$ be the line segment joining $\dz e_2$  and
$ae_1+\dz e_2$ with $e_2=(0,1)$, that is,
$$\gz_\dz (t) := t(a,\dz) + (1-t)(0,\dz) \quad \forall t \in [0,1].$$
Since $u \in W^{1,\fz}(\Omega)$ implies that $u$ is ACL (see \cite[Section 6.1]{hkst}), there exists $\{\dz_n\}_{n \in \nn}$ depending on $u$ such that $u$ is differentiable almost everywhere on $\gz_{\dz_n}$.
Noting that
 $ \dot \gz_{\dz_n}(t) =-e_1 $  and by
 $\frac{\partial  u}{\partial x_1}(x) \ge 0 $, one has  $$   \nabla u (\gz_{\dz_n}(t)) \cdot \dot \gz_{\dz_n}(t) =  -\frac{\pa u}{\pa x_1}(\gz_{\dz_n}(t))\le0 \quad \forall t \in [0,1],$$
 and hence
 \begin{align*}
 u((a,\dz_n)) - u((0,\dz_n))
  & =   \int_0^1 \nabla u (\gz_{\dz_n}(t)) \cdot \dot \gz_{\dz_n}(t) \, dt  = \int_0^1 -\pa_1 u(\gz_{\dz_n}(t))  \, dt    \le  0.
\end{align*}
Thus
\begin{align*}
  u(ae_1)- u(0) &  = \lim_{n \to \fz} [u((a,\dz_n)) - u((0,\dz_n))] \le  0
\end{align*}
 as desired.
\end{rem}

Finally, we introduce the pseudo-length property.

\begin{defn}\label{leng}\rm We say a pseudo-metric space $(\Omega,\rho)$ is a pseudo-length space if for all $x,y \in \Omega$,
  \begin{equation*}
  \rho (x, y) := \inf \{\ell_{\rho} (\gz) \  | \ \gz \in \mathcal C(a, b; x, y;  \Omega )\}
\end{equation*}
where $ \mathcal C(a, b; x, y;  \Omega)$ denotes the class of all continuous curves $\gz:[a,b]\to {\color{black}\Omega}$ with
$\gz(a)=x$ and $\gz(b)=y$, and $$\ell_{\rho} (\gz):= \sup \bbbl \sum_{i=0}^{N-1} \rho (\gz(t_{i}), \gz(t_{i+1}))\ \bigg | \ a=t_0 < t_1 < \cdots < t_N=b  \bbbr.$$
\end{defn}

\section{Proof of Theorem \ref{rad}}

In this section,  we  always suppose that
  the Hamiltonian $H(x,p)$    enjoys assumptions  (H0)-(H3).
To prove Theorem \ref{rad}, we first need several auxiliary lemmas.

\begin{lem} \label{l6.1} 
Suppose that $ \{u_j\}_{j\in\nn}\subset \dot W^{1,\fz}_X(\Omega )$,
 and  there exists $\lz\ge0$ such that
 $$\mbox{$\|H(x,Xu_j)\|_{L^\fz(\Omega )}\le \lz$ for  all $j\in\nn$.}$$
 If $u_j\to u$ in $C^0(\Omega )$, then  $u\in \dot W^{1,\fz}_X(\Omega )$ and
  $\|H(x,Xu)\|_{L^\fz(\Omega )}\le \lz$.
  {\color{black} Here and in below, for any open set $V \subset \Omega$, $u_j\to u$ in $C^0(V)$ refers to for any $K \Subset V$, $u_j\to u$ in $C^0(\overline K)$.
  }
\end{lem}

\begin{proof} By Lemma \ref{radcclem}, one has
$$|u(x)-u(y)|=\lim_{j\to\fz}|u_j(x)-u_j(y)|\le\limsup_{j\to\fz} \||Xu_j|\|_{L^\fz(\Omega )}d_{CC}(x,y).$$
By (H3) and $\|H(\cdot,Xu_j)\|_{L^\fz(\Omega )}\le \lz$, we have
 $\||Xu_j|\|_{L^\fz(\Omega )}\le R_\lz$ for all $j$, and hence
 $$|u(x)-u(y)| \le R_\lz d_{CC}(x,y),$$
 that is, $ u\in\lip_{d_{CC}}(\Omega)$.
By Lemma 2.7 again,  we have $u\in {\color{black}\dot W^{1,\fz}_X(\Omega )}$.

Next we show that   $\|H(x,Xu)\|_{L^\fz(\Omega )}\le \lz$.
It suffices to show that
$\|H(x,X(u|_U))\|_{L^\fz(U )}\le \lz$ for any $U\Subset\Omega$.
{\color{black}Given any $U\Subset\Omega$, we claim that    $X(u_j|_{ U })$ converges to $X(u|_{ U })$ weakly in $L^2( U ,\rr^m)$, that is, for all $1\le i\le m$,  one has
$$  \lim_{j \to \fz}\int_{ U }u_j X_i^\ast \phi \, dx   =\int_{ U } \phi Xu \, dx\ \ \forall \phi\in L^2( U ).$$
}
   To  see  this claim,   note that
  $\||Xu_j|\|_{L^\infty(\Omega )}\le R_\lz$ for all $j\in\nn$, and hence
  $\||Xu_j|\|_{L^2( U  )}\le R_\lz {\color{black}| U  |^{1/2}}$
  for all $j\in\nn$. In other words, for each $k\in\nn$,  the set  $\{X(u_{j }|_{ U })\}_{j\in\nn}$ is bounded  in $L^2( U ,\rr^m)$.
  {\color{black}
      By the weak compactness of $L^2( U ,\rr^m)$,    any   subsequence
      of $\{Xu_{j }\}_{j\in\nn}$ admits a subsubsequence which converges weakly in $L^2( U ,\rr^m)$.  Therefore, to get the above claim,  by a contradiction argument we only need to show that        for any subsequence
  $\{Xu_{j_s}\}_{s\in\nn}$  of $\{Xu_{j }\}_{j\in\nn}$, if $
     Xu_{j_s} \rightharpoonup q_k$  weakly in $L^2( U , \rr^m )$ as $s\to\fz$,
     then $X(u|_{ U }) =q_k$.
     For  such $\{Xu_{j_s}\}_{s\in\nn}$,
 recalling that $u_j\to u$ in $C^0(\Omega )$ as $j\to\fz$,  for all $1\le i  \le m$  one has
   $$ \int_{ U } u X_i^\ast \phi \, dx  =\lim_{j \to \fz}\int_{ U }u_j X_i^\ast \phi \, dx =\lim_{s \to \fz}\int_{ U }u_{j_s} X_i^\ast \phi \, dx
   =\lim_{s \to \fz}\int_{ U }(X_i u_{j_s})\phi \, dx =\int_{ U }q_k\phi \, dx$$
   for any  $\phi\in C^\fz_c( U )$.   This implies that $Xu|_{ U } =q_k$ as desired.

 }

 By Mazur's Theorem, for any $l>0$, we can find a finite  convex combination $w_l$ of $\{X(u_j|_{ U })\}_{j=1}^{\fz}$ so that  $\|w_l-X(u|_{U})\|_{L^2( U  )}\to0 $ as $l\to\fz$.
 Here $w_l$ is a finite  convex combination of $\{X(u_j|_{ U })\}_{j=1}^{\fz}$ if
 there exist $\{\eta_{j}\}_{j=1}^{k_l}$ for some $k_l$  such that
 $$ \sum_{i=1}^{k_l}\eta_i=1\ \mbox{ and}\   w_l=\sum_{j=1}^{k_l}X(u_j|_{ U })$$
 By the quasi-convexity of $H(x,\cdot)$ as in {\color{black}(H1)},
 we have
 $$H(x,w_l)=H(x,\sum_{j=1}^{k_l}\eta_j X(u_j|_{ U }))\le \sup_{1\le j\le k_l}H(x, X(u_j|_{ U }))\le\lz\quad\mbox{for almost all $x\in  U  $}.$$
 Up to considering subsequence we may assume that $w_l\to X(u|_{ U })$ almost everywhere in $ U $.
By the lower-semicontinuity of $H(x,\cdot)$ as in {\color{black}(H0)},
 we conclude that
 $$H(x,X(u|_{ U }))\le {\color{black}\liminf}_{l\to\fz}H(x,w_l)\le  \lz\quad\mbox{for almost all $x\in  U  $}.$$
The proof is complete.
\end{proof}

\begin{lem} \label{l21}
 If $v  \in \dot W^{1,\infty}_{X }(\Omega)$, then
  \begin{equation}\label{claimx}\mbox{$v^+=\max\{v,0\}\in  \dot W^{1,\infty}_{X }(\Omega)$
 and $Xv^+=(Xv)\chi_{\{x\in\Omega, v>0\}} $ almost everywhere.}
 \end{equation}
 Consequently,   let $\{v_i\}_{1\le i\le j} \subset \dot W^{1,\infty}_{X }(\Omega)$  for some $j\in\nn$.
If  $ u=\max_{1\le i \le j}\{v_i\}$ or $ u=\min_{1\le i \le j}\{v_i\}$, then   \begin{equation}\label{claimxx}\mbox{$u\in \dot W^{1,\infty}_{X }(\Omega)$ and $\|H(x,Xu )\|_{L^\fz(\Omega)} \le \max_{1\le i\le j}\{\|H(x,Xv_i)\|_{L^\fz(\Omega)} \}$.}\end{equation}
\end{lem}
\begin{proof}

{\color{black}
First we prove    \eqref{claimx}. Let  $v  \in \dot W^{1,\infty}_{X }(\Omega)$.
By Lemma \ref{radcclem},   $v  \in \lip_{\dcc}(\Omega)$.  Observe that
 $$|v^+(x)-v^+(y)| \le  |v(x)-v(y)| \le \lip_{\dcc}(v,\Omega) \dcc(x,y) \quad \forall x,y \in \Omega,$$
 that is,  $v^+\in  \lip_{\dcc}(\Omega)$.  By Lemma \ref{radcclem} again, $v^+ \in \dot W^{1,\infty}_{X }(\Omega)$. To get
  $Xv^+=(Xv)\chi_{\{x\in\Omega, v>0\}}$ almost everywhere, it suffices to consider
 the restriction  $v|_U$ of $v$
in  any bounded domain $U \Subset \Omega$,  that is, to prove $X(v|_U)^+=(Xv|_U)\chi_{\{x\in U, v>0\}}$ almost everywhere.  But this always holds thanks to  Lemma \ref{uplus}
and the fact
$v|_U \in \dot W^{1,p}_{X}(U)$ for any $1 \le p < \fz$.%
}

Next we prove \eqref{claimxx}.
If $u  =\max\{v_1,v_2\}$, where $v_i \in \dot W^{1,\infty}_{X }(\Omega)$ for $i=1,2$,
 then
   $u  =v_2+(v_1-v_2)^+.$
By \eqref{claimx}, $u \in  \dot W^{1,\infty}_{X }(\Omega)$ and
   \begin{align*}
Xu &=Xv_2+X(v_1-v_2)^+ \\
&=Xv_2+[(X(v_1-v_2)]\chi_{\{x\in\Omega, v_1>v_2\}}\\
&=(Xv_2)\chi_{\{x\in\Omega, v_1\le v_2\}}+ (Xv_1 ) \chi_{\{x\in\Omega, v_1>v_2\}}.
\end{align*}
Thus  $$H(x,Xu(x))= H(x, Xv_2)\chi_{\{x\in\Omega, v_1\le v_2\}}+ H(x,  Xv_1 )
\chi_{\{x\in\Omega, v_1>v_2\}}  \quad \mbox{for almost all $x\in\Omega$}.$$
{\color{black}A similar argument}  holds for $u=\min\{v_1,v_2\}$.
This gives \eqref{claimxx} when $j=2$.
 By an induction argument, we get  \eqref{claimxx} for all $j\ge2$.
\end{proof}

\begin{lem}\label{lem3.2}
For any $ \lz \ge 0$ and $x \in \Omega $, we have
$ d_\lz(x, \cdot), \dl(\cdot,x)\in \dot W^{1,\infty}_{X }(\Omega )$ and
 $$\|H(\cdot,Xd_\lz(x, \cdot))\|_{L^\fz(\Omega )} \le \lz\quad\mbox{and}\quad\|H(\cdot,-X\dl(\cdot,x))\|_{L^\fz(\Omega )} \le \lz.$$
\end{lem}

\begin{proof}
Given any $x\in \Omega $ and $\lz \ge 0$, write $v(z)=d_\lz(x,z)$ for all $z\in \Omega $.
To see $H(\cdot,Xv)\le \lz$ almost everywhere, by Lemma \ref{l6.1},
 it suffices to find a sequence of function $u_j\in {\color{black} \dot W_X^{1,\fz}(\Omega )}$ so that
$H(\cdot,Xu_j)\le \lz$ almost everywhere  and $u_j\to v$ in $C^0(\Omega )$ as $j\to\fz$.

To this end, let $\{K_j\}_{j \in \nn}$ be a sequence of compact subsets in $\Omega $ with
$$\mbox{$\Omega  = \bigcup_{j \in \nn} K_j$ and $K_j \subset K_{j+1}^{\circ}$.}$$
For $j \in \nn$ and $y \in K_j$, by definition of $\dl$ we can find a function
$v_{y,j} \in\dot W^{1,\fz}_X(\Omega)$ such that $H(\cdot,Xv_{y,j})\le\lz$ almost everywhere,
$$d_\lz(x,y)-\frac{1}{2j}\le v_{y,j}(y) -v_{y,j}(x).$$
Since Lemma \ref{lem311} implies that  $d_\lz(x, \cdot)$ is continuous,
 there exists an open neighbourhood $N_{y,j}$ of $y$ with
$$d_\lz(x, z) - \frac{1}{j} \le v_{y,j}(z)-v_{y,j}(x),  \quad\forall z\in N_{y,j}. $$
Thanks to the compactness of $K_j$,
there exist $y_1, \cdots , y_l\in K_j$  such that
$K_j \subset \bigcup_{i=1}^{l} N_{y_i,j}$.  
Write
$$u_j(z) := \max\{v_{y_i,j}(z)-v_{y_i,j}(x) : i = 1, \cdots , l\}\quad\forall z\in K_j.$$
Then $u_j(x)=0$, and
$$d_\lz(x,z) \le u_j(z)+ \frac{1}{j} \ \text{for all} \ z\in K_j.$$
Moreover by Lemma \ref{l21} we have
$$H(\cdot,Xu_j)\le \lz\quad \mbox{in}\ \Omega .$$
Since
  $$d_\lz(x,z)  \ge u_j(z) \ \text{for all} \ z\in K_j$$
  is clear, we conclude that $u_j\to {\color{black}v} $ in $C^0(K_i)$ as $j\to\fz$ for all $i$, and hence {\color{black}$u_j\to {\color{black}v} $   uniformly in any compact subset of $\Omega$} as $j\to\fz$.

  Similarly, we can show $\|H(x,-Xd_\lz(\cdot, x ))\|_{L^\fz(\Omega )} \le \lz$ which finishes the proof.
\end{proof}

In general, for any $E \subset \Omega $, we define
\begin{equation*}
  d_{\lz,E}(z) := \inf_{x \in E} \{d_\lz(x,z)\}.
\end{equation*}

\begin{lem}\label{l6.2}
  For any set $E\subset\Omega $, we have  $d_{\lz,E} \in \dot W^{1,\infty}_{X }(\Omega )$ and $\|H(x,Xd_{\lz,E} )\|_{L^\fz(\Omega )}\le \lz$.
\end{lem}

\begin{proof}
   Let $\{K_j\}_{j \in \nn}$ be a sequence of compact subsets in $\Omega $ with $\Omega  = \bigcup_{j \in \nn} K_j$ and $K_j \subset K_{j+1}^{\circ}$.
   For each $j$ and $y\in K_j$, we can find $z_{y,j}\in E$ such that
   $$d_{\lz,E}(y)\le d_{\lz,z_{y,j}}(y)\le d_{\lz,E}(y)+1/2j.$$
   Thus there exists a neighborhood $N(y)$ of $y$ such that
    $$d_{\lz,E}(y)\le d_{\lz,z_{y,j}}(y)\le d_{\lz,E}(z)+1/ j\quad\forall z\in N(y).$$
   So we can find $\{y_1,\cdots, y_{l_j}\}$ such that
   $K_j=\cup_{i=1}^{l_j}N(y_i) $ for all $i=1,\cdots l_j$.
     Write $$u_j(z)={\color{black}\min_{i=1,\cdots l_j}}\{d_{\lz,z_{y_i,j}}(z)\} \quad\forall z\in K_j.$$
    $$d_{\lz,E}\le u_j(z)\le d_{\lz,E}+1/j\quad \mbox{in}\ K_j .$$
    This means that $u_j\to d_{\lz,E}$ in $C^0(\Omega)$ as $j\to\fz$.
    Note that $$\|H(\cdot, Xu_j )\|_{L^\fz(\Omega)}\le \lz.$$
    By Lemma {\color{black}\ref{l6.1}} we have $d_{\lz,E}\in \dot W^{1,\fz}_X(\Omega)$ and $\|H(\cdot, Xd_{\lz,E})\|_{L^\fz(\Omega)}\le \lz$ as desired.
\end{proof}


We are able to prove (i)$\Rightarrow$(ii)$\Rightarrow$(iii)$\Rightarrow$(i) in Theorem \ref{rad} as below.

\begin{proof} [Proofs of (i)$\Rightarrow$(ii)$\Rightarrow$(iii)$\Rightarrow$(i) in Theorem \ref{rad}.]
The definition of $d_\lz$ directly gives (i)$\Rightarrow$(ii). Obviously, (ii)$\Rightarrow$(iii).
Below we  prove (iii)$\Rightarrow$ (i). Recall that (iii) says that
$u(y)-u(z)\le  d_\lz (z,y)$  for all $x \in \Omega$ and $y,z \in N(x)$, where $N(x)\subset \Omega$ is   a neighbourhood of $x$.
To get (i), since   $\Omega=\cup_{x\in \Omega}N(x)$,
it  suffices to show that  for all  $x\in \Omega$, one has
  $u\in  W_{X}^{1,\fz}(N(x))$ and $\|H(\cdot,Xu)\|_{L^\fz(N(x))}\le\lz$.

Fix any $x$, and write $ U=N(x)$. Without loss of generality we assume that $U$ is bounded.
Notice that   $u\in L^\fz(U)$. Let  $M\in \nn$ so that $M\ge  \sup_{U} |u|$.
  For each $k \in \nn$ and $l \in\{ -Mk, \cdots , Mk\}$,
   set $$u_{k,l}(z) := l/k + {\color{black}d_{\lz,F_{k,l}} (z)}\quad\forall z\in U $$
   where $$F_{k, l} := \{ y \in U  \ | \ u(y) \le \frac{l}{k} \}.$$
 By Lemma \ref{l6.2}, one has  $$\|H(\cdot,{\color{black}Xu_{k,l}})\|_{L^\fz(U)}\le\lz.$$
   Set
  $$ u_k (z):= \min_{l \in\{ -Mk, \cdots , Mk\}}u_{k,l}(z)\quad\forall z\in U.$$
  {\color{black} To get  $u\in \dot W_{X}^{1,\fz}(U)\ \text{and}  \ \|H(\cdot,Xu)\|_{L^\fz(U)}\le\lz$,
  thanks to Lemma \ref{l6.1} with $\Omega=U$, we only need to  show $u_k \to u$  in $C^0(U)$ as $k\to\fz$.

  To see   $u_k \to u$  in $C^0(U)$ as $k\to\fz$,  note that, for   any $k \in \nn$, $-M\le u \le M$ in $U$   implies  $-Mk\le ku \le Mk$ in $U$.
   Thus, at  any $z \in U$,    we can   find  $j\in\nn$ with $-k\le j\le k$, which depends on $z$,
 such that $Mj\le ku(z)\le M(j+1)$. Letting $l=Mj$, we have
  $ \frac{l}{k}\le u(z)\le \frac{l+M}k$.
  We claim that  $u_k(z) \in [u(z),\frac{l+M}{k}]$.
  Obviously,  this claim  gives
$$|u(z)-u_k(z)| \le \frac{M}{k}\quad\forall z\in U,$$
and hence,  the desired convergence $u_k \to u$  in $C^0(U)$ as $k\to\fz$.

 Below we prove the above claim  $u_k(z) \in [u(z),\frac{l+M}{k}]$.  Recall that
  $u(z)\in [\frac{l}{k},\frac{l+M}{k}]$ for some
  $l =Mj$ with $-k\le j\le k$. 

  First, we prove   $u_k(z) \le \frac{l+M}{k}$.
 If $l+M > Mk$, then  $M < (l+M)/k$.
 Since  $$F_{k, Mk}  = \{ y \in U  \ | \ u(y) \le M \}=U,$$  we have
 $d_{\lz,F_{k,Mk}} (z)=0$ and hence,
 $$u_{k,Mk}(z)=M+d_{\lz,F_{k,Mk}} (z)=M < \frac{l+M}{k}.$$
Therefore,
   \begin{equation}\label{upper2}
    u_k (z) \le   u_{k,Mk}(z) < \frac{l+M}{k}.
 \end{equation}
}
 {\color{black}
  If   $l+M\le Mk$,  then $u(z)\in [\frac{l}{k},\frac{l+M}{k}]$ implies $z \in F_{k,l+M}$ and hence $\dl(z,F_{k,l+M})=0$.   Thus 
\begin{equation}\label{eq2121}
  u_k(z) \le u_{k,l+M}(z)=\frac{l+M}k+\dl(z,F_{k,l+M})  =\frac{l+M}{k}.
\end{equation}
   Combining \eqref{eq2121} and \eqref{upper2}, we have $
   u_k(z) \le \frac{l+M}{k}
$
 as desired. }

Next, we prove $u(z)\le u_k(z)$.
 {\color{black}
   For any $-k\le j\le k$ with $Mj\le l$, since $u(z) \ge \frac{l}{k} \ge \frac{Mj}{k}$, we can find $w \in \partial F_{k,Mj}$ such that $d_{\lz,F_{k,Mj}}(z)=\dl(w,z)$. Since $w \in \partial F_{k,Mj}$,
   we deduce that $u(w)=\frac{Mj} k$ and
\begin{equation}\label{eq2122}
  u_{k,Mj}(z)=\frac{Mj}{k}+d_{\lz,F_{k,Mj}}(z)=u(w)+\dl(w,z).
\end{equation}
Note that $w \in \overline U$, hence there exists a sequence $\{w_s\}_{s \in \nn}\subset U$ such that
$w_s\to w$ as $s\to\fz$.
Thus by the assumption (iii),
\begin{equation*}
  u(z)-u(w) =\lim _{s\to\fz}  [u(z)-u(w_s)] \le \lim _{s\to\fz}\dl(w_s,z)
\end{equation*}
By the triangle inequality and  $\dl(w_s,w)\le R_\lz d_{CC}(w_s,w)$ given in Lemma \ref{lem311}, we then obtain
\begin{equation}\label{eq2123}
  u(z)-u(w)  \le
   \lim _{s\to\fz}[\dl(w_s,w)+ d_\lz(w,z)]=d_\lz(w,z).
\end{equation}
Combining \eqref{eq2122} and \eqref{eq2123}, we have
\begin{equation}\label{eq2124}
  u(z)\le u(w)+  \dl(w,z) = u_{k,Mj}(z).
\end{equation}
}

On the other hand,
for any $-k\le j\le k$ with    $Mj>l$,  we have
$$u_{k,Mj}(z)\ge \frac{ Mj}{k}\ge \frac{M(l+1)}k>u(z).$$
From this  and \eqref{eq2124}, it follows that
  $$u_k(z)= \min_{j\in\{ - k, ... ,  l/M\}}u_{k,Mj}(z) \ge u(z)  $$
  as desired.
\end{proof}


\section{Proof of  Theorem \ref{radiv}}

In this section, we always suppose that
  the Hamiltonian $H(x,p)$   enjoys assumptions  (H0)-(H3).

To  prove Theorem \ref{radiv},
we need to show that
 {\color{black}
$(\Omega,\dl)$ is a pseudo-length space for all $\lz \ge \lz_H$ in the sense of Definition \ref{leng}.
In other words, define
\begin{equation*}
  \rl (x, y) := \inf \{\ell_{d_\lz} (\gz) \  | \ \gz \in \mathcal C(a, b; x, y;  \Omega )\},
\end{equation*}
where we recall the pseudo-length $\ell_{d_\lz}(\gz)$ induced by $\dl$ defined in Definition \ref{leng} and $\lz_H$ in \eqref{lh}. We have the following.}

\begin{prop}\label{len}
   For any   $\lz \ge \lz_H$,    we have     $d _\lz=\rl$.
 \end{prop}

To prove Proposition \ref{len}, we need the following approximation midpoint property of $d_\lz$.

\begin{prop} \label{p301}
For any $\lz\ge 0$,   we have
  \begin{equation}\label{eq2211}
    \inf_{z\in\Omega}\max\{\dl (x,z),\dl (z,y)\} \le \frac{1}{2}\dl (x,y) \quad\mbox{for all $x,y \in \Omega $.}
  \end{equation}

\end{prop}

\begin{proof} We prove by contradiction.
 Suppose that
\eqref{eq2211} were not true.
 There exists $x_0,y_0\in\Omega$ such that
 \begin{equation}\label{lambda}
   \inf_{z\in\Omega}\max\{\dl (x_0,z),\dl (z,y_0)\} \ge \frac{1}{2}\dl (x_0,y_0)+\ez_0:=r_0
 \end{equation}
for some $\ez_0>0$.

  Given any $\dz\in(0,\ez_0)$, define  $f(z):=f_1(z)+f_2(z)$ with
  $$f_1(z) := \min\{\dl(x_0, z)-(r_0 - \dz), 0\},\ \mbox{ }  f_2(z):= \max\{(r_0 - \dz) - \dl(z, y_0), 0\}\quad\forall z\in\Omega.$$
We claim that
 $f$ satisfies Theorem \ref{rad}(iii),
that is, for any $z\in\Omega$, there is an open neighborhood $N(z)$  such that
\begin{equation}\label{thm3iii}f(y)-f(w)\le d_\lz(w,y)\quad\forall w,y\in N(z).\end{equation}
Assume the claim \eqref{thm3iii} holds for the moment. Since we have already shown the equivalence  between (ii) and (iii) in Theorem \ref{rad},
we know that  $f$ satisfies Theorem \ref{rad}(ii), that is, $$f(y)-f(w)\le d_\lz(w,y)\quad\forall w,y\in \Omega.$$
 In particular,
 \begin{equation}\label{glob1}
   f(y_0)-f(x_0)\le \dl(x_0, y_0).
 \end{equation}
On the other hand, we have  $f_1(x_0)=-(r_0-\dz)$ and $f_2(y_0)=r_0-\dz$.
Since \eqref{lambda} implies $$\dl (x_0,y_0)=\max\{\dl (x_0,x_0),\dl (x_0,y_0)\} \ge \frac{1}{2}\dl (x_0,y_0)+\ez_0 =r_0 $$
 and $f_2(x_0)=0$ and $f_1(y_0)=0$. Therefore,
 $$ f(y_0) - f(x_0) = f_2(y_0)-f_1(x_0)=
 2r_0 - 2\dz= \dl (x_0,y_0)+2\ez_0 - 2\dz,$$
 By $\dz<\ez_0$, one has
  $$ f(y_0) - f(x_0)  >\dl (x_0,y_0),$$
which contradicts to \eqref{glob1}.

Finally we prove the above claim \eqref{thm3iii}.
Firstly, thanks to Lemma \ref{l21} and \ref{lem3.2}, $H(x,Xf_1)\le \lz$  and $H(x,Xf_2)\le \lz$ almost everywhere in $ \Omega$, and hence,  by the definition of $d_\lz$,
$$f_1(y)-f_1(w)\le d_\lz(w,y)\ \mbox{and}\ f_2(y)-f_2(w)\le d_\lz(w,y) \quad\forall w,y\in \Omega.$$

 Next, {\color{black} set $$\mbox{$\Lambda_1:=\{z\in \Omega  \ | \ d_\lz(x_0,z)<r_0\}$,\quad $\Lambda_2:=\{z\in \Omega  \ | \ d_\lz(z,y_0)<r_0 \}$.}$$
and
$$\Lambda_3:=\{z\in \Omega  \ | \ d_\lz(x_0,z)> r_0-\dz\ \mbox{ and }\  d_\lz(z,y_0)> r_0-\dz\}$$

For any $z \in \Lambda_1$, that is, $d_\lz(x_0,z)<r_0$,
 \eqref{lambda} implies $d_\lz(z,y_0) \ge r_0$. Consequently,
$$f_2(z)=\max\{(r_0 - \dz) - \dl(z, y_0), 0\}= 0$$
and hence,
  $f(z)=f_1(z).$
Consequently,
\begin{equation}\label{lz1}
  f(w)-f(y)=f_1(w)-f_1(y) \le d_\lz(y,w)\quad\forall w,y\in  \Lambda_1.
\end{equation}

Similarly, for any $z \in \Lambda_2$, that is,  $d_\lz(z,y_0)<r_0$,
  \eqref{lambda} implies  $d_\lz(x_0,z) \ge r_0$.
Consequently,
$$f_1(z)=\min\{\dl(x_0, z)-(r_0 - \dz), 0\}= 0$$
 and hence,
 $ f(z)= f_2(z).$ Consequently,
 \begin{equation}\label{lz2}
   f(w)-f(y)=f_2(w)-f_2(y) \le d_\lz(y,w)\quad\forall w,y\in  \Lambda_2.
 \end{equation}

For any $z\in\Lambda_3$, that is,
$d_\lz(x_0,z)> r_0-\dz$ and $  d_\lz(z,y_0)> r_0-\dz$, we have $f_1(z)=0=f_2(z)$
and hence
 $f(z) =0.$  Consequently
 \begin{equation}\label{lz3}
   f(w)-f(y)=0\le d_\lz(y,w)\quad\forall w,y\in \Lambda_3.
 \end{equation}

     Noticing that $\{\Lambda_i\}_{i=1,2,3}$ forms an open cover of $\Omega$, for any $z \in \Omega$, we choose
      \begin{equation}\label{nz}
N(z)= \left \{ \begin{array}{ll}
\Lambda_1  & \quad \text{if }   z \in \Lambda_1, \\
\Lambda_2 & \quad \text{if }   z \in \Lambda_2\setminus\Lambda_1, \\
\Lambda_3 & \quad \text{if }  z\in \Omega\setminus (\Lambda_1\cup\Lambda_2).
\end{array}
\right.
\end{equation}
From \eqref{lz1}, \eqref{lz2} and \eqref{lz3}, we obtain  \eqref{thm3iii} as desired with the choice of $N(z)$ in \eqref{nz}. The proof is complete.
}
\end{proof}

{\color{black}
\begin{lem} \label{dpro}

 Given any $x,y\in \Omega$, the map $\lz\in[\lz_H
 ,\fz)\mapsto d_\lz(x,y)\in[0,\fz)$ is nondecreasing
and right continuous.
\end{lem} }

\begin{proof}
  The fact that $d_\lz(x,y)$ is non-decreasing with respect to $\lz$ is obvious for any $x,y \in \Omega$ from the definition of $\dl$. Given any $x,y \in \Omega$, we show the right-continuity the map $\lz\in [\lz_H ,\fz)\mapsto d_\lz(x,y)\in[0,\fz)$. We argue by contradiction. Assume there exists $\lz_0 \ge \lz_H$ and $x,y \in \Omega$ such that
\begin{equation}\label{rad0}
  \liminf_{\lz \to \lz_0+}d_\lz(x,y)=\lim_{\lz \to \lz_0+}d_\lz(x,y) =c > d_{\lz_0}(x,y) .
\end{equation}
Let $w_\lz(\cdot):= d_\lz(x,\cdot)$. By Lemma \ref{lem3.2}, we know $\|H(\cdot,Xw_\lz)\|_{L^\fz(\Omega)}\le \lz$. Since $\{w_\lz\}_{\lz>\lz_0}$ is a non-decreasing sequence with respect to $\lz$, $\{w_\lz\}$ converges pointwise to a function $w$ as $\lz \to \lz_0+$ and for any set $V \Subset \Omega$ and $x,y \in \overline V$, we have $w_\lz \to w$ in $C^0(\overline V)$. Then applying Lemma \ref{l6.1}, we have
$$  \|H(\cdot,Xw)\|_{L^\fz(\overline V)}\le \lz  $$
for any $\lz>\lz_0$, which implies
\begin{equation}\label{rada}
  \|H(\cdot,Xw)\|_{L^\fz(\overline V)}\le \lz_0.
\end{equation}
By the definition of $w$,
we have
\begin{equation}\label{radb}
  w(x)=\lim_{\lz \to \lz_0+}d_\lz(x,x)=0, \text{ and } w(y)= \lim_{\lz \to \lz_0+}d_\lz(x,y)=c
\end{equation}
Combining \eqref{rada} and \eqref{radb} and applying Theorem \ref{rad}, we have
$$c - 0 =w(y) -w(x)  \le d_{\lz_0}(x,y),$$
which contradicts to \eqref{rad0}. The proof is complete.
\end{proof}

    We are in the position to show

\begin{proof}[Proof of Proposition \ref{len}.] {\color{black}We consider the cases $\lz>\lz_H$ and $\lz=\lz_H$ separately.}

\medskip
\textbf{Case 1. $\lz>\lz_H$.}  First,
 $\dl \le \rho_\lz  $ follows from  the triangle inequality  for $d_\lz$.

 To see   $\rho_\lz\le d_\lz$,   it suffices to prove that
  for any   $z\in \Omega$, the function
    $\rho_\lz(z,\cdot):\Omega\to\rr$ satisfies Theorem \ref{rad}(iii),
    that is,  for any $ x\in\Omega $ we can find a neighborhood $N(x)$ of $ x$ such that
    \begin{equation}\label{rhoth3iii}
\rho_\lz (z,y) -\rho_\lz (z,w) \le \dl (w,y)\quad\forall w,y \in  N(x).
\end{equation}
   Indeed, since we have already shown the equivalence of (i) and (iii) in Theorem \ref{rad},
   \eqref {rhoth3iii} implies that
   $\rho_\lz(z,\cdot)$ satisfies  Theorem \ref{rad}(i), that is,
   $\rho_\lz(z,\cdot)\in \dot W^{1,\fz}_X(\Omega)$ and $\|H(\cdot,X\rho_\lz(z,\cdot))\|_{L^\fz(\Omega)}\le\lz$.
   Taking $\rho_\lz(z,\cdot)$  as the test function in
  the definition of $\dl(z,x)$, one has
    $$\rho_\lz (z,x) \le \dl(z,x)\quad\forall x\in \Omega $$  as desired.

     To prove \eqref{rhoth3iii},  let  $z\in\Omega$ be fixed.
      {\color{black}  For any $x\in\Omega$ and any $t>0$, write
 $$  B_{\dl}^+(x, t) := \{ y \in \Omega \ | \ \dl(x,y) < t \mbox{ or } \dl(y,x) < t \}   $$
 and
 $$  B_{\dl}^-(x, t) := \{ y \in \Omega \ | \ \dl(x,y) < t \mbox{ and } \dl(y,x) < t  \} .   $$
}
 For any $x\in\Omega$, {\color{black}letting
 $r_x = \min\{\frac{R_\lz'}{10} \dcc(x, \pa \Omega),1\}$, by Corollary \ref{lem311-1}, we have
 \begin{equation}\label{Sub}
    B_{\dl}^{+}(x, 6r_x) \subset B_\dcc(x, \frac{6r_x}{R_\lz'} ) \Subset \Omega,
 \end{equation}
 where  $R_\lz'>0$ thanks to (H3).
 }
 Write  $N(x)= B_{\dl}^-(x, {\color{black}r_x})$.
 Given any $w, y  \in N(x)$, it then suffices to prove
  $\rho_\lz(z,y) -\rho_\lz (z,w)\le   \dl (w, y ) $.
 To this end, for any  $0 < \ez <  \frac12 \dl(w, y)$, we will construct a curve
 \begin{equation} \label{ccurve}
     \gz_\ez:[0,1]\to   B_{d_\lz}^+(x,6r_x) \ \mbox{with $ \gz_\ez(0)=w$, $ \gz_\ez(1)=y$ and $\ell_{d_\lz} (\gz_\ez)\le \dl(w, y)+\ez$}. \end{equation}
Assume the existence of $\gz_\ez$ for the moment.
   By the triangle inequality for $\rho _\lz$,    we have
$$\rho_\lz(z,y) -\rho_\lz (z,w)\le \ell_{\dl}(\gz_\ez )  \le \dl (w, y )+\ez $$
By sending $\ez\to0$,  this yields  $\rho_\lz(z,y) -\rho_\lz (z,w)\le   \dl (w, y ) $ as desired.

\medskip
{\bf  Construction of  a curve $\gz_\ez$ satisfying  \eqref{ccurve}}.
{\color{black}
For each $t \in \nn$, set $$D_t:= \{k2^{-t} \ | \ k \in \nn, \ 0\le k\le 2^t\}.$$
We will use induction and Proposition \ref{p301}  to construct a set
\begin{equation}\label{adj4}
  Y_t =\{y_s\}_{s \in D_t} \subset B_{d_\lz}^+(x,5r_x )
\end{equation}
 with  $y_0=w$ and $y_1=y$
so that $Y_t \subset Y_{t+1}$ ,
and   that
\begin{equation}\label{adj5}
  d_\lz(y_{j2^{-t}},y_{(j+1)2^{-t}})\le     2^{-t}(\dl (w, y )+\ez) \mbox{ for any  $0\le j\le 2^t-1$}.
\end{equation}
}
The construction  of  $\{Y_t\}_{t\in\nn}$ is postponed to the end of this proof.
Assuming that $\{Y_t\}_{t\in\nn}$ are constructed, we are able to construct $\gz_\ez$ as below.

Firstly, set $D:=\cup_{t\in\nn} D_t$ and $Y:=\cup_{t\in\nn} Y_t$. Given  any $s_1, s_2 \in D$ with $s_1 <s_2$,  there exists $t \in \nn$ such that
$s_1=l2^{-t} \in D_{t}$, $s_2=k2^{-t} \in D_{t}$ for some $l < k$ and hence
$y_{s_1},y_{s_2} \in Y_t$. Using \eqref{adj5} and the triangle inequality for $d_\lz$, we have
\begin{equation}\label{adj6}
  d_\lz(y_{s_1},y_{s_2})\le \sum_{j=l}^{k-1} \dl(y_{j2^{-t}}, y_{(j+1)2^{-t}}) \le |k-j| 2^{-t}( \dl(w,y) + \ez)= |s_1-s_2|(\dl(w,y) +\ez).
\end{equation}

{\color{black}
Next, define a map $\gz_\ez^0 : D \to Y$ by $\gz_\ez^0(s) = y_s$ for all $s \in D$.
The above inequality \eqref{adj6} implies that
\begin{equation*}
  \lim_{D \ni s' \to s}\gz_\ez^0(s') = \gz_\ez^0(s) \mbox{ for all }s \in D.
\end{equation*}
}
Since $D$ is dense in $[0, 1]$ and $\overline {B_{\dl}^+(y_0, {\color{black}5r_x})}$ is complete, it is standard to extend $\gz_\ez^0$ uniquely  to a continuous map
 $\gz_\ez: [0, 1] \to B_{\dl}^+(y_0, {\color{black}6r_x})$, that is,
$\gz_\ez(s)=\gz_\ez^0(s)$ for any $s\in D$ and
$$  \gz_\ez(s)=\lim_{D \ni s'\to s}\gz_\ez(s')=\lim_{D \ni s'\to s}y_{s'}  \mbox{ for any $s\in [0,1]\setminus D$.}  $$
Recalling \eqref{adj6}, one therefore has
$$d_\lz(\gz_\ez (s_1),\gz_\ez(s_2))\le |s_1-s_2|(\dl(x,y)+\ez)\quad\forall s_1,s_2\in [0,1],\ s_1\le s_2,$$
which gives
 $\ell_{d_\lz} (\gz_\ez)\le \dz+\ez $. Thus the curve
   $\gz_\ez$ satisfies  \eqref{ccurve} as desired.
\medskip

{\bf Construction of $\{Y_t\}_{t\in\nn}$ via induction and Proposition \ref{p301}}.
{\color{black}Since $y,w \in B_{\dl}^-(x,r_x)$, we have
 $\dl(w,x)<r_x$ and $ \dl(x,y) < r_x$, which implies
$$\dz:= \dl(w,y) \le \dl(w,x)+\dl(x,y) <2r_x.$$

We construct $Y_1=\{y_0,y_{1/2},y_1\}$ which satisfies \eqref{adj5} with $t=1$ and
\begin{equation}\label{y1sub}Y_1 \subset B_{d_\lz}^+(x,3r_x)\subset  B_{d_\lz}^+(x,5r_x).
 \end{equation}
 We set $y_0=w$ and $y_1=y$.
Noting  that Proposition \ref{p301} gives
 $$\inf_{z\in\Omega}\max\{\dl (y_0,z),\dl (z,y_1)\} \le \frac{1}{2}\dl (y_0,y_1)$$
 we choose   $y_{1/2}\in\Omega$  so that
\begin{equation}\label{induction1}
  \max\{d_\lz(y_0,y_{1/2}),d_\lz(y_{1/2},y_1) \}\le  \frac12\dz+\frac14\ez.
\end{equation}
Obviously, \eqref{induction1} gives \eqref{adj5}.
To see \eqref{y1sub}, obviously,
$$y_0,y_1 \in B_{d_\lz}^-(x,r_x) \subset B_{d_\lz}^+(x,3r_x).$$
Moreover, noting that  $0<\ez <\frac12\dz<r_x$ implies  $$\frac12\dz+\frac14\ez< \dz<2r_x$$
 and  that  $ y\in B_{\dl}^-(x,r_x)$ implies $\dl(y,x)\le r_x$, we have
$$\dl(y_{1/2},x) \le \dl(y_{1/2},y_1)+\dl(y_1,x) \le \frac12\dz+\frac14\ez + \dl(y,x) \le 3r_x,$$
which gives   $y_{1/2} \in  B_{d_\lz}^+(x,3r_x).$

In general, by induction given any $t\ge2$, assume that $Y_{t-1}=\{y_s\}_{s\in D_{t-1}}$ is constructed so that
\begin{equation}\label{yt-1sub}Y_{t-1}\subset B_{d_\lz}^+(x,(3+\sum_{l=1}^{t-2}2^{-l})r_x+\ez\sum_{l=1}^{t-1}2^{-l})
\end{equation}
and that
\begin{equation}\label{adj}
  \dl(y_{j2^{-(t-1)}}, y_{(j+1)2^{-(t-1)}}) \le 2^{-(t-1)}\bl \dz + \ez\sum_{l=1}^{t-1}2^{-l} \br \mbox{ for any $0 \le j \le 2^{t-1}-1$}.
\end{equation}
Here and in what follows, we make the convention that $\sum_{l=1}^{t-2}2^{-l} =0$ if $t=2$.
Since
 \begin{equation}\label{le5rx}(3+\sum_{l=1}^\fz2^{-l})r_x+\ez\sum_{l=1}^\fz 2^{-l}\le 4r_x+\ez<4r_x+\dz\le 5r_x, \end{equation}
the inclusion \eqref{yt-1sub} implies $Y_{t-1}\subset B_{d_\lz}^+(x,5r_x)$ and hence \eqref{adj4}.

Below, we construct $Y_t=\{y_s\}_{s\in D_t}$ satisfying  \eqref{adj5} and
 \begin{equation}\label{adj7}
   Y_{t}\subset B_{d_\lz}^+(x,(3+\sum_{l=1}^{t-1}2^{-l})r_x+\ez\sum_{l=1}^{t}2^{-l}),
 \end{equation}
Note that \eqref{le5rx} and \eqref{adj7} imply $Y_{t}\subset B_{d_\lz}^+(x,5r_x)$ and hence \eqref{adj4}.

 We define $Y_t=\{y_s\}_{s\in D_t}$ first.
 Since  $D_{t-1}\subsetneqq D_{t}$,  for
$s \in D_{t-1}$,    $y_s \in Y_{t-1}$ is defined.  It is left to define $y_s$ for $s\in D_t \setminus D_{t-1}$.
Given any $s\in D_t\setminus D_{t-1}$,  we know that
$s=j2^{-t}\in D_t \setminus D_{t-1}$ for some odd $j$ with $1\le j\le 2^t-1$.
Write $s'=(j-1)2^{-t}  $ and $s''=(j+1)2^{-t}  $. Then $s',s''\in D_{t-1}$ and hence $y_{s'}\in Y_{t-1}$ and $y_{s''}\in Y_{t-1}$ are defined.
Since Proposition \ref{p301}  gives
 \begin{equation}
 \inf_{z\in\Omega}  \max\{d_\lz(y_{s'},z),d_\lz(z,y_{s''}) \}\le  \frac12d_\lz(y_{s'},y_{s''})
 \end{equation}
 we choose $y_s\in\Omega$ such that
 \begin{equation}\label{adj3}
   \max\{d_\lz(y_{s'},y_{s}),d_\lz(y_{s},y_{s''}) \}\le  \frac12d_\lz(y_{s'},y_{s''})+2^{-2t}\ez.
 \end{equation}

 Note that \eqref{adj3} and \eqref{adj} gives  \eqref{adj5} directly. Indeed,
 for any $0 \le j \le 2^t-1$, if $j$ is odd,    applying \eqref{adj3} with
 $s=j2^{-t}$,  $ s'=(j-1)2^{-t}$ and $s''=(j+1)2^{-t}$, we deduce that
 $$\dl(y_{j2^{-t}}, y_{(j+1)2^{-t}})=d_\lz(y_{s},y_{s''})\le \frac12d_\lz(y_{s'},y_{s''})+2^{-2t}\ez.$$
 Since $s',s''\in D_{t-1}$ and $s''=s'+2^{-(t-1)}$,
applying \eqref{adj} to $y_{s'},y_{s''}$ we have
 \begin{align}\label{adj2}
   \dl(y_{j2^{-t}}, y_{(j+1)2^{-t}})
   & \le \frac12 2^{-(t-1)}\bl \dz + \ez\sum_{l=1}^{t-1}2^{-l} \br + 2^{-t}2^{-t}\ez  = 2^{-t}\bl \dz + \ez\sum_{l=1}^{t}2^{-l} \br .
 \end{align}
 If $j$ is even, then $j\le 2^t-2$. Applying \eqref{adj3} with
 $s=(j+1)2^{-t}$,  $ s'=j2^{-t}$ and $s''=(j+2)2^{-t}$,
 we deduce that
 $$\dl(y_{j2^{-t}}, y_{(j+1)2^{-t}})=d_\lz(y_{s'},y_{s})\le \frac12d_\lz(y_{s'},y_{s''})+2^{-2t}\ez.$$
 Similarly, we  also have \eqref{adj2}.

To see \eqref{adj7}, since \eqref{yt-1sub} gives
 \begin{equation}\label{adj8}
   Y_{t-1}\subset B_{d_\lz}^+(x,(3+\sum_{l=1}^{t-2}2^{-l})r_x+\ez\sum_{l=1}^{t-1}2^{-l}) \subset B_{d_\lz}^+(x,(3+\sum_{l=1}^{t-1}2^{-l})r_x+\ez\sum_{l=1}^{t }2^{-l}),
 \end{equation}
 it suffices to check
 $$Y_{t} \setminus Y_{t-1}=\{y_{j2^{-t}} \ |  \ {1\le j\le 2^t-1,\\
  \mbox{$j$ is odd}} \} \subset B_{d_\lz}^+(x,(3+\sum_{l=1}^{t-1}2^{-l})r_x+\ez\sum_{l=1}^{t}2^{-l}).$$
 For any odd number $j$ with $1 \le j \le 2^t-1$, since   $y_{(j-1)2^{-t}}  \in Y_{t-1}$,
 combining \eqref{adj2} and \eqref{adj8} and noting  $\ez <\dz <2r_x$, we obtain
 \begin{align*}
  \dl(y_{j2^{-t}},x)  & \le  \dl(y_{j2^{-t}},y_{(j-1)2^{-t}}) + \dl(y_{(j-1)2^{-t}},x) \\
    & \le
 2^{-t}\bl 2r_x + \ez\sum_{l=1}^{t}2^{-l} \br + \bl 3+\sum_{l=1}^{t-2}2^{-l} \br r_x+\ez\sum_{l=1}^{t-1}2^{-l} \\
 & \le \bl 3+\sum_{l=1}^{t-1}2^{-l} \br r_x   + \ez\sum_{l=1}^{t}2^{-l}
 \end{align*}
 which implies \eqref{adj7}.   We finish the proof of Case 1. }

 \bigskip
  \textbf{ Case 2. $\lz=\lz_H$.} Fix $x,y \in  \Omega$. For any $\ez>0$ sufficiently small, by the right continuity of the map $\lz \mapsto \dl(x,y)$ at $\lz=\lz_H$ from Lemma \ref{dpro}, there exists $\mu >\lz_H$ such that
 \begin{equation}\label{d01}
   d_\mu(x,y) < d_{\lz_H}(x,y) +\frac{\ez}2.
 \end{equation}
 By Case (i), there exists $\gz:[0,1] \to \Omega$ joining $x$ and $y$ such that
 \begin{equation}\label{d02}
   \ell_{d_\mu}(\gz) < d_\mu(x,y) +\frac{\ez}2.
 \end{equation}
 By the definition of the pseudo-length and recalling from Lemma \ref{dpro} that the map $\lz \mapsto \dl(z,w)$ is non-decreasing for all $z,w \in \Omega$, we have
 \begin{equation}\label{d03}
   \ell_{d_{\lz_H}}(\gz) \le \ell_{d_\mu}(\gz).
 \end{equation}
 Combining \eqref{d01}, \eqref{d02} and \eqref{d03}, we conclude
 $$ \ell_{d_{\lz_H}}(\gz) < d_{\lz_H}(x,y) + \ez. $$
 The proof is complete.
\end{proof}

 We are ready to prove Theorem \ref{radiv}.
\begin{proof}[Proof of Theorem \ref{radiv}] Obviously, (iii) in Theorem \ref{rad} $\Rightarrow$ (iv) in Theorem \ref{radiv}. To see the converse, let $\lz\ge 0$ be  as in (iv).
Given any $x$ and  $y,z \in N(x)$, where $N(x)$ is given in (iv)  we need to show
$$   u(y)-u(z)\le d_\lz(z,y).    $$
By Proposition \ref{len}, we know $(\Omega,d_\lz)$ is a pseudo-length space. Hence for any $\ez>0$, there exists a curve $\gz_\ez : [0,1] \to \Omega$ joining $z$ and $y$ such that
\begin{equation}\label{ez1}
  \ell_{d_\lz}(\gz_\ez) \le d_\lz(z,y) +\ez.
\end{equation}
Since $\gz_\ez \subset \Omega$ is compact, we can find a finite set $\{t_i\}_{i=0}^n \subset [0,1]$ satisfying
$$ t_0=0, \ t_n=1, \ \text{and } \gz_\ez(t_{i+1}) \in N(\gz_\ez(t_i)), \ i=0,\cdots , n-1 $$
where $N(\gz_\ez(t_i))$ is the neighbourhood of $\gz_\ez(t_i)$ in (iv).
Hence by (iv) we have
\begin{equation*}
  u(\gz_\ez(w_{i+1})) - u(\gz_\ez(w_i)) \le d_\lz(\gz_\ez(w_{i}),\gz_\ez(w_{i+1})), \ i=0, \cdots i-1.
\end{equation*}
Summing the above inequalities from $0$ to $n-1$, we have
\begin{align*}
  u(y) - u(z) & = u(\gz_\ez(w_n)) - u(\gz_\ez(w_0)) \le \sum_{i=0}^{n-1}d_\lz(\gz_\ez(w_{i}),\gz(w_{i+1}))
    \le \sum_{i=0}^{n-1} \ell_{d_\lz}(\gz_\ez| _{[w_i,w_{i+1}]}) \\
   & = \ell_{d_\lz}(\gz_\ez )
   \le d_\lz(y,z) +\ez
\end{align*}
where in the last inequality we applied \eqref{ez1}.
Letting $\ez\to0$ in the above inequality, we obtain (iii) in Theorem \ref{rad}.

Finally, \eqref{ident2} is a direct consequence of (iv)$\Leftrightarrow$(i) and thanks to Lemma \ref{dpro}, the minimum in \eqref{ident2} is achieved.
\end{proof}

{\color{black}
\begin{rem}  \label{exp1}\rm
The assumption $R_\lz'>0$ is needed in the proof of Proposition \ref{len}. Indeed, recall the construction of $\gz_\ez$ in the proof of Proposition \ref{len} below \eqref{adj6}. To guarantee $\gz_\ez$ is a continuous map, especially $\gz_\ez([0,1])$ is compact under the topology induced by $\dcc$, we need $\{\dl(x,\cdot)\}_{x \in \Omega}$ induces the same topology as the one by $\dcc$. By Corollary \ref{lem311-1}, $R_\lz'>0$ can guarantee this.

Moreover, to show $\gz_\ez \Subset \Omega$ in \eqref{Sub} in the proof of Proposition \ref{len}, for each $x \in \Omega$, we need the existence of $r_x>0$, such that
\begin{equation}\label{Sub2}
  B^+_{d_\lz}(x,r_x) \Subset \Omega.
\end{equation}
Again, by Corollary \ref{lem311-1}, $R_\lz'>0$ can guarantee this. If $R_\lz'>0$ does not hold for some $\lz>0$, in Remark \ref{d1} (ii), the example shows that \eqref{Sub2} may fail for some $x \in \Omega$.
\end{rem}
}

\section{McShane extensions and minimizers}

In this section, we always suppose that
  the Hamiltonian $H(x,p)$   enjoys assumptions  (H0)-(H3) and further that $\lz_H=0$.

Let $U\Subset \Omega$ be any domain.
Note that   the restriction of $d_\lz$ in $U$ {\color{black}may not have the pseudo-length property} in $U$,
and moreover,
Theorem \ref{rad} with $\Omega $ replaced by $U$ may not hold for the restriction of $d_\lz$ in $U$.
Thus instead of $d_\lz$, below we use intrinsic pseudo metrics
 $\{d^U_\lz\}_{\lz> 0}$  in $U $,
 which are defined via  \eqref{d311} with $ \Omega$ replaced by $ U$,  that is,
\begin{equation}\label{du}
  d^{U} _{\lambda} (x,y):= \sup\{u(y)-u(x) : u \in \dot W^{1,\infty}_{X }(U),\ \|H(\cdot, Xu)\|_{L^\infty(U)} \le \lz
  \}\quad\forall \mbox{$x,y \in  U$ and $\lz \ge 0$}.
\end{equation}

Obviously we have proved the following.
\begin{cor} \label{radu}
Theorem \ref{rad}, Theorem \ref{radiv} and Proposition \ref{len} hold with $\Omega$ replaced by $U$ and $d_\lz$ replaced by $d_\lz^U$.  In particular, $d_\lz^U$ {\color{black}has the pseudo-length property} in $U$ for all $\lz \ge 0$.
\end{cor}

 Observe that, apriori, $d^U_\lz$ is only defined in $U$ but not in $\overline U$.
Naturally, we extend $d^U_\lz:U\times U\to[0,\fz)$ as a function  $\wz d^U_\lz: \overline U\times \overline U\to[0,\fz]$  by
$$
\wz d_\lz^U(x,y)=\lim_{r\to0}\inf\{d_\lz^U(z,w) \ | \ (z,w)\in U\times U, |(z,w)-(x,y)|\le r\}.$$
Obviously, $\wz d^U_\lz(x,y)=d^U_\lz(x,y)$ for all  {\color{black}$(x,y)\in U\times U$}, and $\wz d_\lz^U$ is {\color{black}lower semicontinuous} in $\overline U\times \overline U$,
that is, for any $a\in\rr$, the set $$ \{(x,y)\in \overline U\times \overline U \ | \ \wz d^U_\lz(x,y)>a\}$$ is open in $\overline U$.
One may also note that it may happen that
  $\dlu(x,y)=+\fz$  for some   $(x,y)\notin U\times U$.
Below, {\color{black} for the sake of simplicity, we write $\wz d_\lz^U$ as $d_\lz ^U$. We define $d_{CC}^U$ by letting $H(x,p)=|p|$ in $U$ and $\lz=1$ in \eqref{du}.}

 The following property will be used later.
\begin{lem} \label{property}
Let $U \Subset \Omega$ be a   subdomain and $\lz\ge 0$.
\begin{itemize}
\item [(i)]
  For any $x,y\in \overline U$, we have $\dlu(x,y) \ge  \dl(x,y)\ge  R'_\lz d_{CC}(x,y)$.

\item [(ii)]
  For any $x\in U$ and {\color{black}$y \in  U$ with $\dcc(x,y) < \dcc(x,\pau)$}, we have
$
  \dlu(x, y)\le R_\lz d_{CC}(x,y).
$

\item [(iii)]
   For any $x\in U$, {\color{black}let $x^\ast\in\partial U$ be the point
such that $ d_{CC}(x,x^\ast)=d_{CC}(x,\partial U)$. Then  $$d_\lz^U (x,x^\ast)\le R_\lz d_{CC}(x,x^\ast)<\fz.$$}
\vspace{-0.8cm}
\item [(iv)]
   For any $x\in \overline U$ and $ y \in U$ 
   we have
    $$\mbox{$\dlu(x, z) \le \dlu(x,y) + \dlu(y,z) $ and $ d^U_\lz(z,x)
 \le d^U_\lz(z, y)+ d^U_\lz(y,x)  \quad\forall z\in\pa U$}
 .$$

\item[(v)]  Given any $z\in \partial U$, if $\dlu(x,z)= \fz$ for some $x\in \overline U$,
then   $\dlu(y,z)=+\fz$ for any $y \in \overline U$.
\item[(vi)] Given any $x,y\in \overline U\times \overline U$ the map $\lz\in[0,\fz)\mapsto d_\lz^U(x,y)\in[0,\fz]$ is nondecreasing and for ${\color{black} 0<\lz<\mu<\fz}$,
     $$d^U_\lz(x,y)<\fz \text{ if and only if } d^U_\mu(x,y)<\fz.$$
As a consequence,
$$\overline U^\ast:=\{y\in\overline U \ | \ d_{\lz}^U(x,y)<\fz\quad\mbox{for some $x\in U$ and $\lz >0$} \}$$  is well-defined independent of the choice of  $\lz  >0$.

\item[(vii)] Given any $x,y\in  \overline U^\ast\times \overline  U^\ast$, the map $\lz\in {\color{black} [0 ,\fz)}\mapsto d_\lz^U(x,y)\in[0,\fz]$ is right continuous.
\end{itemize}
\end{lem}

\begin{proof}
To see (i), for any $x,y\in U$, since the restriction
$u|_U$ is a test function in the definition of $\dlu(x,y)$ whenever
 $u$  is a test function in the definition of $\dl(x,y)$,
we have $d^U_\lz(x,y)\ge d_\lz(x,y)$.
In general,
given any    $(x,y)\in (\overline U\times\overline  U)\setminus (U\times U)$,
for any $r>0$ sufficiently small, we have
 $d^U_\lz(z,w)\ge d_\lz(z,w)$ whenever $z,w\in U$  and $|(z,w)-(x,y)|\le r$.
By the continuity of $d_\lz$ in $\Omega\times\Omega$, we have
  $$\lim_{r\to0}\inf\{d^U_\lz(z,w) \ | \ \mbox{$z,w\in U$  and $|(z,w)-(x,y)|\le r$}\}\ge d_\lz(x,y),$$
that is, $ d_\lz^U(x,y)\ge d_\lz(x,y)$.  Recall that   $d_\lz(x,y)\ge R'd_{CC}(x,y)$ comes from Lemma 2.1.

\medskip
 To see (ii), given any $y\in U$ with $\dcc(x,y)<\dcc(x,\pau)$, there is a geodesic
 $\gz$ with respect to $\dcc$
 connecting $x$ and $y$ so that $\gz\subset B_{d_{CC}}(x,d_{CC}(x,\partial U))$.
For any function $u\in \dot W^{1,\infty}_{X }(U)$ with $\|H(\cdot,Xu)\|_{L^\fz(U)}\le \lz$,
we know that
$\|Xu\|_{L^\fz(U)}\le R_\lz$.
{\color{black}Let $U' \Subset U$   and $x,y \in U'$. Thanks to Proposition \ref{mollif2}, we can find a sequence $\{u_k\}_{k\in\nn} \subset C^\fz(U')$
such that  $u_k\to u$ everywhere as $k\to\fz$
 and $\|Xu_k\|_{L^\fz(U')}\le  R_\lz + A_k(u)$ with $\lim_{k \to \fz} A_k(u) \to 0$.}
Since
$$u_k(y)-u_k(x)= \int_\gz Xu_k\cdot \dot\gz \le \int_\gz |Xu_k||\dot\gz| \le   \int_\gz {\color{black}(R_\lz+A_k(u))}|\dot\gz| = {\color{black}(R_\lz+A_k(u))} \dcc(x,y) \mbox{ for all $k \in\nn$},$$
we have
  $$u(y)-u(x)
=\lim_{k \to \fz}\{u_k(y)-u_k(x)\}\le \lim_{k \to \fz} [R_\lz+A_k(u)] \dcc(x,y) = R_\lz \dcc(x,y).$$
Taking supremum in the above inequality over all such $u$, we have $d^U_\lz(x,y)\le R_\lz \dcc(x,y)$.


\medskip
To see (iii), given any $x\in U$,  there exists  $x^\ast\in\partial U$
such that $ d_{CC}(x,x^\ast)=d_{CC}(x,\partial U)$. By (ii) and the definition of $d^U_\lz(x,x^\ast)$,
we know that  $d_\lz^U (x,x^\ast)\le R_\lz d_{CC}(x,x^\ast)<\fz$.

\medskip
To get (iv), 
for  any $ x\in\overline U$, $y\in U$ and $z\in \partial U$, choose $x_k,z_k\in U$ such that
$\dlu(x_k,y)\to \dlu(x,y)$ and $\dlu(y,z_k)\to \dlu(y,z)$  as $k\to\fz$.
Since   $$\dlu(x_k, z_k) \le \dlu(x_k,y) + \dlu(y,z_k),$$  letting $k\to\fz$ and by the lower-semicontinuous of $d^U_\lz$ we get
$$\dlu(x, z ) \le \dlu(x,y) + \dlu(y,z )$$
In a similar way, we also have $$ d^U_\lz(z,x)
 \le d^U_\lz(z, y)+ d^U_\lz(y,x) $$

\medskip
Note that (v) is a direct consequence of (iv).

\medskip
We show (vi). The fact that $d_\lz^U(x,y)$ is non-decreasing with respect to $\lz$ is obvious for any $x,y \in \overline U$. Assume $0< \lz<\mu<\fz$. Then $d^U_\mu(x,y)<+\fz$ implies $d^U_\lz(x,y)<+\fz$. Conversely, if $d_{\mu}^U(x,y)=+\fz$, we show $d_\lz^U(x,y)=+\fz$. This may happen if  at least one of $x$ and $y$ lie in $\pau$. Then for any $\{x_k\}_{k\in \nn} \subset U$ and $\{y_k\}_{k\in \nn}\subset U$ converging to $x$ and $y$, it holds
$$ \liminf_{k \to \fz}d_{\mu}^U(x_k,y_k) = +\fz $$
where we let $x_k \equiv x$ (resp. $y_k \equiv y$) if $x \in U$ (resp. $y\in U$). By (i) and (ii), we have for any $\lz\le \mu$
$$ \liminf_{k \to \fz} d_{\lz}^U(x_k,y_k) \ge R'_\lz \liminf_{k \to \fz} d_{CC} (x_k,y_k) \ge \frac{R'_\lz}{R_{\mu}} \liminf_{k \to \fz} d_{\mu}^U(x_k,y_k) = +\fz $$
where we recall that $R'_\lz>0$ for all $\lz>0$.

\medskip
Finally, we show (vii). Since we only consider $x,y \in \overline U^\ast$,
by an approximation argument, it is enough to show the right-continuity for $x,y \in U$. The proof is similar to the one of Lemma \ref{dpro}. We omit the details.
The proof of Lemma \ref{property} is complete.
\end{proof}

{\color{black}
 \begin{lem} \label{ll}
Suppose that  $ U \Subset \Omega$ and that $V$ is a subdomain of $U$.
 For any $\lz\ge0$, one has
$$d^U_\lz(x,y)\le d^V_\lz(x,y)\quad\forall x,y\in \overline V.$$
Conversely, given any $\lz>0$ and $x \in V$, there exists a neighborhood $N _\lz(x) \Subset V$ of $x$ such that
$$\mbox{ $d^U_\lz ( x,y) = d^V_\lz(x,y)$ and $d^U_\lz (y,x) = d^V_\lz(y,x)$  for any $y \in N_\lz(x)$.}$$
\end{lem}
}

\begin{proof}

For any $u \in \dot W^{1,\infty}_{X }(U)$ with $\|H(\cdot, Xu)\|_{L^\infty(U)} \le \lz$, we know that  the restriction $u|_{V} \in \dot W^{1,\infty}_{X }(V)$ with $\|H(\cdot, Xu|_V)\|_{L^\infty(V)} \le \lz$. Hence by the definition of $d^U_\lz $ and $d^V_\lz$,
 $ d^U_\lz  (x,y) \le d^V_\lz (x,y) $ for all $ x,y \in  V $ and then   for all $ x,y \in \overline V.$

Conversely, we just show $d^U_\lz (x,\cdot) = d^V_\lz(x,\cdot)$ in some neighborhood $N(x)$.  In a similar way, we can prove
  $d^U_\lz (\cdot , x) = d^V_\lz(\cdot ,x)$ in some neighborhood $N(x)$.

By Lemma \ref{property} (i), one has $\dl^V (x,y)  \ge   R'_\lz d_{CC} (x,y)  $
for any $x,y\in V$. Thus for any $r>0$,
 $d^V_\lz(x,y)\le r$ implies  $ d_{CC} (x,y)\le r/R'_\lz$.
Given any $x\in V$ and $0<r<\dcc(x,\pa V)/R'_\lz$, we therefore  have
$$N _r(x):=\{y\in V \ | \ d^V_\lz(x,y)\le r\}\subset B_{d_{CC}}(x,r/R'_\lz)\Subset V.$$
Define $u_{x,r}: \Omega \to \rr$ by
$$u_{x,r}(z):=\min\{d^V_\lz(x,z),r\}\quad\forall z\in \Omega.$$
If $ r< \dcc(x,\pa V)R'_\lz/4$, we claim  that
\begin{equation}\label{claim5.3}\mbox{$u_{x,r}\in\lip_{d_{CC}}(\Omega)$ with
$H(z,Xu_{x,r}(z))\le \lz$ for almost all $z\in \Omega$.}
\end{equation}
Assume  claim \eqref{claim5.3} holds for  the moment.
By \eqref{claim5.3},   we are able to  take     $u_{x,r}|_{U}$  as a test function in $d^U_\lz$ so that
\begin{equation*}
    u_{x,r} (z)-u_{x,r} (w)\le d^U_\lz (w,z)\quad\forall (w,z)\in U\times U.
  \end{equation*}
 On the other hand,  for any $y\in N_r(x)$, since
 $d^V _\lz(x,y)=  u_{x,r}(y)-u_{x,r}(x)$,  we get $d^V _\lz(x,y)\le d^U_\lz(x,y)$
as desired.

  Finally, we prove  the claim \eqref{claim5.3}.

 {\bf Proof of    the claim \eqref{claim5.3}.}
First,  by Lemma 3.3 and Lemma 3.4,
   the restriction $u_{x,r}|_V$ of $u_{x,r}$ in $V$ belongs to  $\dot W^{1,\fz}_X(V)$
 with $H(z,X u_{x,r}|_V(z))\le \lz$ almost everywhere in $V$, and hence
   \begin{equation}\label{eq11}
  u_{x,r}|_V(z)-u_{x,r}|_V(w)\le d^V_\lz(w,z)\quad\forall (w,z)\in V\times V.
   \end{equation}

Next,
we show $u \in \lip_\dcc(\Omega)$. Given any $w,z\in\Omega$, we consider 3 cases separately.

\medskip
 {\bf Case 1.}  $w \in \overline {B_{d_{CC}}(x,r/R'_\lz)}$ and  $z\in \overline { B_{d_{CC}}(x, r /R'_\lz)}$.
 We have
 $z\in \overline {B_{d_{CC}}(w,2 r/R'_\lz )}$. Since
 $$ d_{CC}(w,\pa V) \ge  d_{CC}(x,\pa V)-d_{CC}(x,w) \ge \frac{4r}{R_\lz'} -  \frac{r}{R_\lz'} =  \frac{3r}{R_\lz'} > \frac{2r}{R_\lz'} \ge d_{CC}(w,z),  $$
by Lemma \ref{property} (ii), $d^V_\lz(w,z)\le R_\lz d_{CC}(w,z)$, which  combined with \eqref{eq11}, gives
 $$  u_{x,r}(z)-u_{x,r} (w)\le R_\lz d_{CC}(w,z).$$

\medskip
 {\bf Case 2.} $w,z\notin  B_{d_{CC}}(x,r / R'_\lz )$. Then
$w,z \in \Omega\setminus B_{d_{CC}}(x,r / R'_\lz )$, since $u_{x,r} $ is  constant $r$ in $\Omega\setminus B_{d_{CC}}(x,r /R'_\lz)$,
   we know that
    $$  u_{x,r} (z)-u_{x,r} (w)=r-r=0\le R_\lz d_{CC}(w,z).$$

\medskip
 {\bf Case 3.}  $w \in \overline {B_{d_{CC}}(x,r/R'_\lz)}$ and  $z\notin \overline {B_{d_{CC}}(x,r/R'_\lz)}$. Then for any $\ez>0$, there exists a curve $\gz_\ez:[0,1] \to \Omega$ joining $z$ and $w$ such that
     $$\ell_{\dcc}(\gz_\ez) \le \dcc(w,z) + \ez$$
     and there exists $t \in [0,1]$ such that $\gz_\ez(t) \in \pa B_\dcc(x,r/R'_\lz)$.
      Thus using Case 1 and Case 2, we deduce
      \begin{align*}
        u_{x,r}(z)-u_{x,r} (w)& = u_{x,r}(z)-u_{x,r} (\gz_\ez(t)) +u_{x,r}(\gz_\ez(t))-u_{x,r} (w) \\
         & \le R_\lz d_{CC}(z,\gz_\ez(t)) + R_\lz d_{CC}(\gz_\ez(t),w)\\
         & \le R_\lz \ell_{\dcc}(\gz_\ez|_{[0,t]}) + R_\lz \ell_{\dcc}(\gz_\ez|_{[t,1]})   \\
         & \le R_\lz[\dcc(w,z)+\ez].
      \end{align*}
   Letting $\ez \to 0$ in the above inequality, we conclude
   $$ u_{x,r}(z)-u_{x,r} (w)\le R_\lz\dcc(w,z). $$

 Finally,  by Lemma \ref{radcclem}, $u_{x,r}\in \dot W^{1,\fz}_X(\Omega)$. Note that   $X u_{x,r}=0$ in $\Omega \setminus
     \overline {B_{d_{CC}}(x,r/R'_\lz)}$ implies $H(z,X u_{x,r}(z) )=0$ almost everywhere.
  Therefore recalling $H(z,X u_{x,r}(z) )\le \lz$ almost everywhere in $V$ and $\Omega= V \cup (\Omega\setminus
     \overline {B_{d_{CC}}(x,r/R'_\lz)})$, we conclude $H(z,X u_{x,r}(z) )\le \lz$ almost everywhere in $\Omega$.
\end{proof}

{\color{black}
\begin{lem} \label{lem5.9}
Suppose that  $ U \Subset \Omega$ and that
 $V=U\setminus\{x_i\}_{1\le i\le m}$ for some $m\in\nn$ and $\{x_i\}_{1\le i\le m}\subset U$.
 Then for any $\lz \ge 0$, one has $$\mbox{$d^V_\lz(x,y)=d^U_\lz(x,y)$ for all $x,y\in\overline U$.}$$
\end{lem}
}

\begin{proof}  
Obviously $d_\lz^U\le d^V_\lz$ in $\overline U$.
Conversely, we show $d_\lz^V\le d^U_\lz$ in $\overline U$. First, by an approximation argument, it suffices to consider $x,y \in U$.
By the right continuity of $\lz\in[0,\fz)\to d_\lz^U(x,y)$ for any $x,y\in \overline U^\ast$, up to considering $ d_{\mu+\ez}^U$ for sufficiently small $\ez>0$, we may assume that $\mu>0$.  For any $\lz>0$,
by the pseudo-length property of $d^U_\lz$ as in Proposition \ref{len},
it suffices to prove that  for any  curve $\gz:[a,b] \to  U$, one has
\begin{equation}\label{xx-4}d^V_\lz(\gz(a),\gz(b)) \le \ell_{d^U_\lz}(\gz) .\end{equation}
We consider the following 3 cases.

\medskip
{\bf Case 1. } $\gz((a,b))\subset V$ and $\gz(\{a,b\}) \subset  V$,   that is, $\gz( [a,b])\subset V$.
Recall from Lemma \ref{ll}  that,   for each $x \in V$, there exists a neighborhood $N(x)$ such that
$$
  d^V_\lz(x,y) = d^U_\lz(x,y) \quad \forall  y \in N(x).
$$
Since $\gz \subset \cup_{t \in [a,b]}N(\gz(t))$,
we can find $a=t_0=0<t_1<\cdots< t_m=b$ such that $\gz\subset \cup_{i=0}^m N(\gz(t_i))$
and $\gz([t_i,t_{i+1}])\subset N(\gz(t_i))$. By the triangle inequality, one has
\begin{equation*}
d_\lz^V(\gz(a),\gz(b))\le
 \sum_{i=0}^{m-1}d^U_\lz (\gz(t_{i}),\gz(t_{i+1}))\le
\ell_{d^U_\lz}(\gz).
\end{equation*}

\medskip
{\bf Case 2. }  $\gz((a,b))\subset V$  and $\gz(\{a,b\})\not\subset  V$.
 Applying Case 1 to $\gz|_{[a+\ez,b-\ez]}$ for sufficiently small $\ez>0$, we get
\begin{equation*}
d_\lz^V(\gz(a),\gz(b))\le \lim_{\ez\to 0}d_\lz^V(\gz(a+\ez),\gz(b-\ez))
\le \liminf_{\ez\to 0}  \ell_{d^U_\lz}(\gz|_{[a+\ez,b-\ez]})
\le  \ell_{d^U_\lz}(\gz).
\end{equation*}

\medskip
{\bf Case 3. }  $\gz((a,b))\not\subset   V$.
 Without loss of generality,  for each $x_i$, there is at most one $t \in(a,b)$  such that $\gz(t )=x_i$.
 Indeed,  let $t ^\pm$ as the maximum/mimimum $s\in(a,b)$ such that $\gz(s)=x_1$.
 Then  $ a\le t ^-\le t^+ \le  b$. If $t ^-<t ^+ $,
we consider $ \gz_1:[a,b-(t^+ -t^- )]\to U$ with $\gz_1(t)=\gz(t)$ for $t\in[a,t ^-]$, and $\gz_1(t)=\gz(t-(t^+ -t^- ))$ for $t\in [t^+ ,b]$.
Then $\ell_{d^U_\lz}(\gz_1) \le \ell_{d^U_\lz}(\gz) $.   Repeating this procedure for $x_2,\cdots,x_m$ in order, we may get a new curve
$\eta$ such that for each $x_i$, there is at most one $t\in(a,b)$  such that $\gz(t )=x_i$ and $\ell_{d^U_\lz}(\eta)\le \ell_{d^U_\lz}(\gz).$

Denote by $\{a_j\}_{j=0}^s$ with $a=a_0<a_1<\cdots  <a_s=b$ such that $\gz( \{a_1,\cdots,a_{s-1}\}) \subset U\setminus V=\{x_i\}_{1\le i\le m}$
and $\gz([a,b]\setminus\{a_1,\cdots,a_{s-1}\})\subset V$. Applying Case 2 to $\gz|_{[a_j,a_{j+1}]}$ for all $0\le j\le s-1$,
we obtain
\begin{equation*}
d_\lz^V(\gz(a),\gz(b))\le  \sum_{j=0}^{s-1} d_\lz^V(\gz(a_j),\gz(a_{j+1}))\le \sum_{j=0}^{s-1}\ell_{d^U_\lz}(\gz|_{[a_j,a_{j+1}]})  \le \ell_{d^U_\lz}(\gz)
\end{equation*}
as desired. The proof is complete.
\end{proof}

{\color{black}For any $g\in C^0(\partial U)$, write}
\begin{align*}
  \mu(g,  \partial U) := \inf \big \{ \lz \ge0 \ | \ g(y) - g(x) \le d^{U}_\lz(x, y)  \quad \forall \ x, y \in \pa{U}  \big \}.
\end{align*}
The following lemma  says the infimum can be reached.

\begin{lem} We have
  \begin{align*}
  \mu(g,  \partial U)
   & =\min \big \{ \lz\ge0 \ | \ g(y) - g(x) \le d^{U}_\lz(x, y)  \quad \forall \ x, y \in \pa{U}\cap \overline U^\ast  \big \}.
\end{align*}
\end{lem}

\begin{proof}
  First, if $x \in \pa{U}\setminus \overline U^\ast $ or $y \in \pa{U}\setminus \overline U^\ast $, we have $d^{U}_\lz(x, y) = \fz$. Hence $g(y) - g(x) \le d^{U}_\lz(x, y)$ holds trivially, which implies that
 \begin{align*}
  \mu(g,  \partial U)
    =\inf \big \{ \lz\ge0 \ | \ g(y) - g(x) \le d^{U}_\lz(x, y)  \quad \forall \ x, y \in \pa{U}\cap \overline U^\ast  \big \}.
\end{align*}
  Thanks to Lemma \ref{property}(vii), we finish the proof.
\end{proof}

If $\mu=\mu(g,\partial U)<\fz$, we define
 $$\cS^+_{g;U}(x) := \inf_{y\in\partial U} \{g(y) + d^{U}_\mu  (y, x)   \}  \quad{\rm and}\quad
  \cS^-_{g; U}(x) := \sup_{y\in\partial U} \{g(y) - d^{U}_\mu  (x, y)    \}\quad \forall x \in \overline U.$$
 Note that $\cS^\pm_{g; U}$  serve  as    ``McShane" extensions of $g$ in $U$.

\begin{lem} \label{cont} If
   $\mu=\mu(g,\partial U)<\fz$, then we have
   \begin{enumerate}
   \item[(i)]$\csgu^\pm  \in \dot W^{1,\fz}_X(U)\cap C^0(\overline U)$ with $\csgu^\pm=g$ on $\partial U$;
    \item[(ii)]  for any $ x,y\in\overline U$, \begin{equation}\label{yy-3}\csgu^\pm (y)-\csgu^\pm (x)\le d^U_\mu(x,y) ;\end{equation}
\item[(iii)]   $ \|H(\cdot,{\color{black}X\csgu^\pm} )\|_{L^\fz(U)}\le \mu.$
   \end{enumerate}
\end{lem}

\begin{proof}  By Corollary \ref{radu}, (ii) implies (iii) and $ \csgu^\pm \in \dot W^{1,\fz}_X(U)$.  Below we show $\csgu^\pm=g$ on $\partial U$,
(ii) and
$\csgu^\pm  \in  C^0(\overline U)$ in order.

\medskip
{\bf Proof for $\csgu^\pm=g$ on $\partial U$}. 
For any $x\in\partial U$, by definition  we have   $$\csgu^-(x) \ge g(x)\ge \csgu^+(x).$$ Conversely,  for  $y \in \pau$,  one has
  $$g(y) - d^U_\mu(x, y) \le g(x)\le g(y) + d^U_\mu(y, x)$$ and hence
  $\csgu^-(x) \le g(x)\le \csgu^+(x) $ as desired.

\medskip
{\bf Proof of (ii). }  We only prove (ii) for  $\csgu^-$; the  proof for $\csgu^+$ is similar.
For $x,y\in\partial U$, by the definition of $\mu$ one has
$$ \csgu^- (y)-\csgu^- (x)=g(y)-g(x) \le d^U_\mu(x,y).$$
For $x\in U$ and $y\in \pa U$, by definition
 \begin{align*}
    \csgu^-(x)  & \ge  g(y) - d^U_\mu(x, y)   =  \csgu^- (y)    - d^U_\mu(x, y)
  \end{align*}
  and hence $$ \csgu^- (y)-\csgu^- (x) \le d^U_\mu(x,y).$$

  For $x\in\overline  U$ and $ y \in  U$,
 by Lemma \ref{property}(iv), we have $$d^U_\mu(x,z)\le d^U_\mu(x,y)+d^U_\mu(y,z)   \quad\forall z\in\partial U.$$
 One then has
 \begin{align*}
    \csgu^-(y)  & =  \sup_{z  \in \pau }\{g(z) - d^U_\mu(y, z)  \} \\
    & \le  \sup_{z  \in \pau }\{g(z) - d^U_\mu
(x, z)+d^U_\mu(x,y)  \} \\
     & \le  \csgu^-(x)  +d^U_\mu(x,y),
  \end{align*}
     and hence $$ \csgu^- (y)-\csgu^- (x) \le d^U_\mu(x,y).$$

%
%
%
%

{\bf Proof of   $\csgu^\pm\in C^0(\overline U)$}. We only consider   $\csgu^-\in C^0(\overline U)$;
 the proof for $\csgu^+\in C^0(\overline U)$ is similar.
It suffices to show that  for any $x \in \pau$ and a sequence $\{x_j\} \subset U$ converging to $x$,
  \begin{equation*}
    \lim_{j \to \fz}\csgu^-(x_j)=\csgu^-(x)=g(x).
\end{equation*}

Choosing $x^\ast_j\in\partial U$ such that $d_{CC}(x_j,x^\ast_j)=\dist_{d_{CC}}(x_j,\partial U)$,
one has $$\mbox{$d_{CC}(x_j,x^\ast_j)\le d_{CC}(x_j,x)\to0$, and hence $d_{CC}(x_j^\ast,  x)\to0 $.}$$
Thanks to {\color{black}Lemma \ref{property}(iii) with $x=x_j$ and $x^\ast=x^\ast_j$ therein, we deduce}
$$d^U _\mu(x_j,x^\ast_j)\le R_\mu d_{CC}(x_i,x^\ast_j)\to0.$$
Since
$$\csgu^-(x_j)\ge g(x_j^\ast)-d^U _\mu(x_j,x^\ast_j),$$
by the continuity of $g$, we have  $$\liminf_{j \to \fz}\csgu^-(x_j)\ge g(x).$$

Assume that
$$\liminf_{j \to \fz}\csgu^-(w_j)\ge g(x)+2\ez\ \mbox{
 for some $\{w_j\}_{j\in\nn}\subset  U$ with $w_j\to x$ and some $\ez>0$.}$$
By the definition we can find $z_j\in\partial U$ such that
$$
\csgu^-(w_j)-\ez\le g(z_j)-d^U_\mu(w_j,z_j)  .$$
Thus {\color{black} for $j \in \nn$ sufficiently large, we have}
$$g(z_j)-d^U_\mu(w_j,z_j) \ge g(x)+\ez. $$
Up to some subsequence, we assume that $z_j\to z
  \in\partial U$.
Note that  $$d^U_\mu(x,z) \le\liminf_{j\to\fz} d^U_\mu(w_j,z_j)$$
By the continuity of $g$, we conclude  $$g(z )-d^U_\mu(x,z) \ge g(x)+\ez,$$ which is a contradiction with the definition  of $\mu$ and $\mu<\fz$.
\end{proof}

Write
\begin{equation}\label{minimizer}
\mathbf I(g,U) =\inf\{\|H(\cdot,Xu)\|_{L^\fz(U)} \ | \ u\in \dot W_X^{1,\fz}(U)\cap C^0(\overline U), u|_{\partial U}=g \}.
\end{equation}
{\color{black} A  function $ u:\overline U\to\rr$
 is called as a minimizer  for  $\mathbf I(g,U)$   if
 $$\mbox{$ u\in \dot W_X^{1,\fz}(U)\cap C^0(\overline U)$, $ u|_{\partial U}=g$ and $\|H(\cdot,Xu)\|_{L^\fz(U)}=\mathbf I(g,U)$.}$$}
We have the following existence and properties  for minimizers.

\begin{lem}\label{lem231}
For any  $g\in C(\partial U)$ with  $\mu(g,\partial U)<\fz$, we have the following:
\begin{enumerate}
  \item[(i)]  We have $
   \mu(g,\partial U)=\mathbf I(g,U).
 $ Both of $\csgu^\pm$ are minimizers for $\mathbf I(g,U)$.

\item[(ii)] If  $u$ is a  minimizer for $\mathbf I(g,U)$
    then
$$
 \csgu^- \le u \le \csgu^+ \quad\mbox{in $\overline U$ and }
 \| H(x,Xu) \|_{L^\fz(U)} = \mathbf I(g,U)=\mu(g,\partial U) .
$$

\item[(iii)]
If $u,v$ are minimizer for $\mathbf I(g,U)$, then $tu+(1-t)v$ with $t\in(0,1)$, $\max\{u, v\}$ and $\min\{u, v\}$ are minimizers for $\mathbf I(g,U)$.
\end{enumerate}
\end{lem}

\begin{proof}    (i)   Since   $\mu(g,\partial U) <\fz$, then  by {\color{black}Lemma  \ref{cont}}, we know that $\csgu^\pm $ satisfies  the condition required in \eqref{minimizer} and hence
  \begin{equation} \label{xx0}
  \mathbf I(g,U)\le \|H(\cdot,{\color{black}X\csgu^\pm} )\|_{L^\fz(U)}\le \mu.\end{equation}
Below we show that $\mu(g,\partial U)\le  \mathbf I(g,U) .$   Note that combining this and \eqref{xx0} we know that  $$\mathbf I(g,U)= \|H(\cdot,{\color{black}X\csgu^\pm} )\|_{L^\fz(U)}= \mu$$ and moreover,
      $\csgu^\pm$ are minimizers for ${ \bf I}(g;U)$.

For any $\lz>\mathbf I(g,U)$,
there is a  function $u\in \dot W^{1,\fz}_X(U)\cap C^0(\overline U)$ with $u=g$ on $\partial U$   such that
$  \|H(x,Xu)\|_{L^\fz(U)}\le \lz$. By Corollary \ref{radu},
$${\color{black} u(y)-u(x)\le d_\lz^U(x,y), \ \forall x,y\in U.}$$
By the continuity of $u$ in $\overline U$ and  the definition of  $d^U_\lz$ in $\overline U\times\overline U$
we have $$g(x)-g(y)\le d^U_\lz(x,y) {\color{black} \mbox{ for all } x,y  \in \pau}.$$
  Thus
$\mu(g,\partial U)\le  \mathbf I(g,U) .$

\medskip
(ii) If $u$ is a minimizer for ${ \bf I}(g,U)$, one has $$\|H(\cdot, Xu )\|_{L^\fz(U)} = \bi(g,U)=\mu.$$
By {\color{black}Corollary \ref{radu}},    $u(y) - u(x) \le d_\mu^U(x, y)$ for any $x, y \in U$   and hence, by the continuity of $u$ and the definition of $d_\mu^U$,
for all $x, y \in \overline U$.
Since $u= g$ on $\partial U$, for any $x \in U$, one has $g(y) - d_\mu^U(x,  y) \le u(x)$ for any $y \in \partial U$, which yields
$\csgu^-(x) \le u(x)$. By a similar argument, one also has  $u \le \csgu^+$ in $U$  as desired.

\medskip
(iii)  Suppose that $u_1,u_2$ are  minimizers for $\bi(g,U)$.
Set  $$u_\eta:=(1-\eta) u_1+ \eta u_2 \mbox{ for any $\eta \in [0, 1]$.}$$
 Then $u_\eta\in \dot W^{1,\fz}_X(U)\cap C^0(\overline U)$, and $u_\eta= g$ on $\partial U$,
  and by (H1),   $$H(\cdot, Xu_\eta(\cdot)) \le \max_{i=1,2}\{  H(\cdot, Xu_i(\cdot)) \}\le \mu \ \mbox{ a.e. on $U$}.$$
  This, {\color{black}combined with the definition of $\bi (g,U)$}, implies that $u_\eta$ is also a minimizer for ${ \bf I}(g,U)$.

Finally, note that
$$\mbox{$\max \{u_1, u_2\} \in   \dot W^{1,\fz}_X(U)\cap C^0(\overline U)$,   $\max \{u_1, u_2\} = g$ on $\partial U$}.$$
By Lemma \ref{l21}, one has
$$H(\cdot, X\max \{u_1, u_2\}(\cdot)) \le \max_{i=1,2}\{  H(\cdot, Xu_i(\cdot)) \}\le \mu \ \mbox{ a.e. on $U$}.$$
 We know that
 \begin{align*}
  \bi (g, U )\le \cf (\max\{u_1, u_2\}, U)
   & \le \max \{\cf (u, U), \cf (v, U)\} =  \bi (g, U ),
\end{align*}
that is,  $\max \{u_1, u_2\}$ is a minimizer for ${ \bf I}(g,U)$.
Similarly, $\min\{u_1, u_2\}$ is a minimizer for ${ \bf I}(g,U)$.
\end{proof}

We have the following improved  regularity for Mschane extension via $d_\lz$.

{\color{black}
\begin{lem} \label{lem5.7}
Suppose that $U\Subset\Omega$ and that $V \subset U$ is a subdomain.  If $g:\partial V\to\rr$ satisfies
\begin{equation}\label{xx2}
g(y)-g(x)\le d^U_\lz (x,y)\quad\forall x,y\in\partial V\end{equation}
for some $\lz \ge 0$, then  $ \mu(g,\partial U)\le \lz$ and
\begin{equation}\label{yy-1}
S^\pm_{g,V}(y)- S^\pm_{g,V}(x)\le d^U_\lz (x,y)\quad\forall x,y\in\overline V.
\end{equation}
\end{lem}
}

\begin{proof}
Since $d_\lz^U\le d^V_\lz$ in $\overline V\times\overline V$, we know that
$$g(y)-g(x)\le d^V_\lz (x,y)\quad\forall x,y\in\partial V$$
and hence
$$ \mu(g,\partial V)=\min\{\eta\ge 0 \ |  \ g(y)-g(x)\le d^V_\eta (x,y)\quad\forall x,y\in\partial V   \}\le \lz.$$

To prove  \eqref{yy-1}, by the pseudo-length property of $d_\lz^U$ as in Corollary \ref{radu},
it suffices to prove that  for any  curve $\gz:[a,b] \to  \Omega$ with $\gz(a),\gz(b)\in \overline V$, one has
\begin{equation}\label{xx-2}S^\pm_{g,V}(\gz(b))-S^\pm_{g,V}(\gz(a))\le \ell_{d_\lz^U}(\gz) .
\end{equation}
   We consider the following 4 cases.

\medskip
{\bf Case 1.}  $\gz((a,b))\subset V$ and $\gz(\{a,b\}) \subset  V$,   that is, $\gz( [a,b])\subset V$.   Noting $\mu=\mu(g,\partial V)\le \lz$, one then has $d_\mu^U \le d_\lz^V $ in $  V\times  V$.
From the definition of  $S^\pm_{g,V}$,    it follows
$$S^\pm_{g,V}(y)- S^\pm_{g,V}(x)\le d_\mu^V (x,y) \quad\forall x,y\in  V.$$
Recall from Lemma \ref{ll}  that,   for each $x \in V$, there exists a neighborhood $N(x) \Subset V$ such that
$$
  d^V_\lz(x,y) = d_\lz^U(x,y) \quad \forall  y \in N(x).
$$
We therefore have
\begin{equation}\label{ell1}
S^\pm_{g,V}(y)- S^\pm_{g,V}(x)\le   d_\lz^U  (x,y)\quad\forall x \in  U, y\in N(x).
\end{equation}

Since $\gz \subset \cup_{t \in [a,b]}N(\gz(t))$,
we can find $a=t_0=0<t_1<\cdots< t_m=b$ such that $\gz\subset \cup_{i=0}^m N(\gz(t_i))$
and $\gz([t_i,t_{i+1}])\subset N(\gz(t_i))$.
 Applying  \eqref{ell1} to $\gz(t_i)$ and $\gz(t_{i+1})$, we have
 $$S^\pm_{g,V}(\gz(t_{i+1}))-S^\pm_{g,V}(\gz(t_{i}))
 \le d_\lz^U(\gz(t_{i}),\gz(t_{i+1})).$$
Thus
 \begin{equation*}
S^\pm_{g,V}(\gz(b))- S^\pm_{g,V}(\gz(a))=\sum_{i=0}^{m-1} [S^\pm_{g,V}(\gz(t_{i+1}))-S^\pm_{g,V}(\gz(t_{i}))]\le
 \sum_{i=0}^{m-1}d_\lz^U (\gz(t_{i}),\gz(t_{i+1}))\le
\ell_{d_\lz^U}(\gz).
\end{equation*}

\medskip
{\bf Case 2.} $\gz((a,b))\subset V$  and $\gz(\{a,b\})\not\subset  V$.
 Applying Case 1 to $\gz|_{[a+\ez,b-\ez]}$ for sufficiently small $\ez>0$, we get
\begin{equation*}
S^\pm_{g,V}(\gz(b))- S^\pm_{g,V}(\gz(a))= \lim_{\ez\to 0}[S^\pm_{g,V}(\gz(b-\ez))- S^\pm_{g,V}(\gz(a+\ez))]
\le \liminf_{\ez\to 0}  \ell_{d_\lz^U}(\gz|_{[a+\ez,b-\ez]})
\le  \ell_{d_\lz^U}(\gz).
\end{equation*}

\medskip
{\bf Case 3.} $\gz((a,b))\not\subset   V$ and $\gz(\{a,b\})\subset  V$.
Set
$$ t_\ast=\min\{t\in[a,b],\gz(t)\notin V\} \mbox{ and } t^\ast=\max\{t\in[0,1],\gz(t)\notin V\}.$$
Then $\gz(t_\ast) \in \pa V$ and $\gz(t^\ast) \in \pa V$.
 Write
 $$ S^\pm_{g,V}(\gz(b))- S^\pm_{g,V}(\gz(a))=S^\pm_{g,V}(\gz(b))- S^\pm_{g,V}(\gz(t^\ast))+
 S^\pm_{g,V}(\gz(t^\ast))- S^\pm_{g,V}(\gz(t_\ast))+S^\pm_{g,V}(\gz(t_\ast))- S^\pm_{g,V}(\gz(a)).$$
 Note that
 $$S^\pm_{g,V}(\gz(t^\ast))- S^\pm_{g,V}(\gz(t_\ast))=g(\gz(t^\ast))-g(\gz(t_\ast))\le d_\lz^U  (\gz(t_\ast),\gz(t^\ast))\le \ell_{d_\lz^U}(\gz|_{[a,t_\ast]}) .$$
 By this, and applying Case 2 to $\gz|_{[a,t_\ast]}$ and $\gz|_{[t^\ast,b]}$,
 we obtain
 $$ S^\pm_{g,V}(\gz(b))- S^\pm_{g,V}(\gz(a))\le \ell_{d_\lz^U}(\gz|_{[t^\ast,b]})  + \ell_{d_\lz^U}(\gz|_{[t_\ast,t^\ast ]})+\ell_{d_\lz^U}(\gz|_{[a,t_\ast]})   =\ell_{d_\lz^U}(\gz) .$$

\medskip
{\bf Case 4.} $\gz((a,b))\not\subset   V$ and $\gz(\{a,b\})\not \subset  V$.
 If  $\gz(\{a,b\})\subset  \pa V$, then
$$S^\pm_{g,V}(\gz(b))- S^\pm_{g,V}(\gz(a))=g(\gz(b))-g(\gz(a))\le d_\lz^U  (\gz(b),\gz(a))\le \ell_{d_\lz^U}(\gz ) .$$
 If $\gz(a)\in V$ and $\gz(b)\in\pa V$,
 set   $s^\ast=\min\{s\in[a,b] \ | \ \gz(s) \in \pa V\}$.
 Obviously $a<s^\ast\le b$, and we can find a sequence of $\ez_i>0$ so that $\ez_i\to0$ as $i\to\fz$ and
  $\gz(s^\ast-\ez_i)\in V$.
  Write
 \begin{align*}
  S^\pm_{g,V}(\gz(b))- S^\pm_{g,V}(\gz(a))= S^\pm_{g,V}(\gz(b))- S^\pm_{g,V}(\gz(s^\ast ))+ S^\pm_{g,V}(\gz(s^\ast ))-S^\pm_{g,V}(\gz(a)).
\end{align*}
Since $\gz(s^\ast ),\gz(b)\in\pa V$, we have
$$S^\pm_{g,V}(\gz(b))- S^\pm_{g,V}(\gz(s^\ast))=g(\gz(b))-g(\gz(s^\ast))\le d_\lz^U  (\gz(b),\gz(s^\ast))\le \ell_{d_\lz^U}(\gz|_{[s^\ast,b]} ) .$$
Applying Case 3 to $\gz|_{[a,s^\ast-\ez_i]}$, one has
$$
 S^\pm_{g,V}(\gz(s^\ast ))-S^\pm_{g,V}(\gz(a))
 =    \lim_{i\to\fz}[S^\pm_{g,V}(\gz(s^\ast-\ez_i))-S^\pm_{g,V}(\gz(a))]\le\lim_{i\to\fz}\ell_{d_\lz^U}(\gz|_{[a,s^\ast-\ez_i]} )\le  \ell_{d_\lz^U}(\gz|_{[a,s^\ast]} ).$$
 We therefore have
  \begin{align*}
  S^\pm_{g,V}(\gz(b))- S^\pm_{g,V}(\gz(a))\le \ell_{d_\lz^U}(\gz|_{[s^\ast,b]} ) +\ell_{d_\lz^U}(\gz|_{[a,s^\ast]} )=\ell_{d_\lz^U}(\gz).
\end{align*}
 If $\gz(a)\in \pa V$ and $\gz(b)\in V$,  we could prove in a similar way. Thus \eqref{xx-2} holds and the proof is complete.
 \end{proof}

The following will be used in Section 6.
Let $U\Subset\Omega$ and $u\in \dot W^{1,\fz}_X(U)\cap C^0(\overline U) $  satisfying
\begin{equation}\label{udmu}\mbox{$u(y)-u(x)\le d^U_\mu(x,y) \quad\forall x,y\in\overline U$}
\end{equation}
   for some $0\le \mu<\fz$.
   Given any subdomain  $V\subset U$, write $h =u|_{\partial V}$  as the restriction of $u$ in $\overline V$.   Since $d^U_\mu\le  d^V_\lz$ in $\overline V\times\overline V$ as given in Lemma \ref{ll}, one has
   $$u (y)-u (x) \le d^V_\mu(x,y)  \quad\forall x,y\in\pa V$$
    and hence $\mu(u|_{\pa V}, \partial V) \le \mu$.
     Denote by $\cS_{u|_{\pa V} ,V}^\mp$   the McShane extension of
  $u|_{\pa V} $ in $V$. Define
\begin{equation}\label{upm}
u^\pm:= \left \{ \begin{array}{lll} \cS_{u|_{\pa V},V}^\pm
& \quad \text{in} & \quad V \\
u
  & \quad \text{in}&\quad \overline U \backslash V.
\end{array}
\right.
\end{equation}

{\color{black}
\begin{lem} \label{lem5.10}
Under the assumption \eqref{udmu} for some $0\le \mu<\fz$, the functions $u^\pm$ defined in \eqref{upm}   are continuous in $\overline U$ and satisfy
\begin{equation}\label{claim1}
  u^\pm(y)-u^\pm(x)\le d_\mu^U(x,y)\quad\forall x,y\in  \overline  U.
\end{equation}
In particular,  $u^\pm\in  \dot W^{1,\fz}_X(U) \cap C^0(\overline U)$ and $  \|H(\cdot,Xu^\pm)\|_{L^\fz(U)}\le \mu $.
 \end{lem}
 }

\begin{proof}  We only prove Lemma \ref{lem5.10} for $u^+$; the proof of Lemma \ref{lem5.10} for $u^-$ is similar. By Lemma \ref{cont}, $\cS^+_{h,V}\in C^0(\overline V)$ and $\cS^+_{h,V}=u$ on $\pa V$. These, together with $u\in C^0(\overline U)$ imply that   $u^+\in C^0(\overline U)$.
Moreover, by   Corollary \ref{radu}, if $u^+$ satisfies \eqref{claim1}, then
$u^+\in  \dot W^{1,\fz}_X(U)$  and $  \|H(\cdot,Xu^\pm)\|_{L^\fz(U)}\le \mu $.
Below we prove \eqref{claim1} for $u^+$ via the following 3 cases.
By the right continuity of $\lz\in[0,\fz)\to d_\lz^U(x,y)$ for any $x,y\in \overline U$, up to considering $ d_{\mu+\ez}^U$ for sufficiently small $\ez>0$, we may assume that $\mu>0$.

 {\bf Case 1.}  $x,y\in \overline U\setminus V $.
By \eqref{udmu} we have
$$u^+(y)-u^+(x)= {\color{black}u(y)-u(x)}    \le d_\mu^U(x,y).$$

   {\bf Case 2.}   $x \in \overline V  $ and $y \in V$ or $x \in   V  $ and $y \in \overline V$.
Applying Lemma \ref{lem5.7} with $(U, d_\lz^U,V,d^V_\lz,g)$ replaced  by
 $( U,d^U_\mu,V,d^V_\mu, u|_{\partial V})$,
 one has
$$ u^+(y)-u^+(x)=
 \cS_{h, V}^+(y)- \cS_{h, V}^+(x)\le d^U_\mu(x,y)\quad\forall x,y\in \overline V.$$

\medskip
   {\bf Case 3.} $x\in V$ and $y\in\overline U\setminus  V $ or $x\in \overline U\setminus   V$ and $y\in    V $.
For any $\ez>0$,  {\color{black}by Corollary \ref{radu}}, there exists a curve $\gz$ joining $x, y$ such that
 $$\ell_{d^U_\mu}(\gz)\le d_\mu^U(x,y)+\ez.$$
 Let $z\in \gz\cap \partial V$.  Applying Case 2 and Case 1,  we have
$$
u^+(y)-u^+(x)=u^+(y)-u^+(z)+u^+(z)-u^+(x)\le d^U_\mu(z,y)+d^U_\mu(x,z)\le \ell_{d^U_\mu}(\gz)\le d^U_\mu(x,y)+\ez .$$
By the arbitrariness of $\ez>0$, we have $
u^+(y)-u^+(x) \le d^U_\mu(x,y)   $ as desired.
\end{proof}

\section{Proof of Theorem \ref{t231}}

In this section, we always suppose that
  the Hamiltonian $H(x,p)$   enjoys assumptions  (H0)-(H3) and further that $\lz_H=0$.

\begin{defn}\label{d231}\rm
Let $U\Subset\Omega$ be a domain  and $g\in C^0(\pa U)$ with $\mu(g,\partial U)<\fz$.

\begin{enumerate}
\item[(i)]
   A minimizer $u  $  for ${ \bf I}(g,U)$ is called a local superminimizer   for ${ \bf I}(g,U)$
    if $u \ge \cS_{u|_{\partial V}; V}^-$ in $V$    for any subdomain $V\subset U$.

\item[(ii)]  A minimizer $u  $  for ${ \bf I}(g,U)$ is called a local subminimizer for ${ \bf I}(g,U)$  if
 $u \le \cS_{u|_{\partial V}; V}^+$ in $V$  for any subdomain $V\subset U$.
    \end{enumerate}
\end{defn}

The next lemma shows McShane extensions are  local super/sub minimizers.
\begin{lem}\label{lem234}
Let $U\Subset\Omega$ and $g\in C^0(\pa U)$ with $\mu(g,\partial U)<\fz$.
\begin{enumerate}
\item[(i)] For any subdomain  $V\subset U$, we have
\begin{equation}\label{e601}
 \cS_{h^+, V}^- \le \cS_{h^+, V}^+ \le \csgu^+ \text{ in $V$, where $h^+=\csgu^+|_{\partial V}$.}
\end{equation}
In particular, $\csgu^+$ is a local superminimizer for ${ \bf I}(g,U)$.
\item[(ii)]For any subdomain  $V\subset U$, we have
\begin{equation}\label{e601-}
 \csgu^-\le \cS_{h^-, V}^- \le \cS_{h^-, V}^+  \text{ in $V$, where $h^-=\csgu^-|_{\partial V}$.}
\end{equation}
In particular, $\csgu^-$ is a local subminimizer for ${ \bf I}(g,U)$.
\end{enumerate}
\end{lem}

\begin{proof}  We only prove (i); the proof for (ii) is similar.
Write  $\mu=\mu(g,\partial U).$
By Lemma \ref{lem231}, $\csgu^+$ is a minimizer for $\bi(g,U)$,  $\mu=\bi(g,U)=\|H(\cdot,X\csgu^+)\|_{L^\fz(U)} $, and
\begin{equation}\label{udmus}\mbox{$\csgu^+(y)-\csgu^+(x)\le d^U_\mu(x,y) \quad\forall x,y\in\overline U$}
\end{equation}
Fix any subdomain  $V\subset U$.
 Denote by $\cS_{h,V}^\pm$  the McShane extension of
  $h^+=\csgu^+|_{\partial V}$ in $V$.
Note that $\cS_{h, V}^-\le \cS_{h, V}^+ $ in $\overline V$, that is, the first inequality
 in \eqref{e601} holds.

Below we show  $\cS_{h, V}^+\le \csgu^+$ in $V$, that is, the second inequality in
 \eqref{e601}.
Let $ u^+$ be as in the \eqref{upm} with $u=\cS_{g,U}^+$, that is,
 \begin{equation*}
 u^+:= \left \{ \begin{array}{lll} \cS_{h^+,V}^+
 & \quad \text{in} & \quad V \\
 \csgu^+  & \quad \text{in}&\quad \overline U \backslash V.
 \end{array}
 \right.
 \end{equation*}
 By \eqref{udmus}, we apply Lemma \ref{lem5.10} to conclude that  $u^+\in  \dot W^{1,\fz}_X(U) \cap C^0(\overline U)$ and $  \|H(\cdot,Xu^+)\|_{L^\fz(U)}\le \mu $.
Note that $u^+=\csgu^+$  on $\pa U$ and hence, by definition of $\bi (g, U )$,
one has $\bi (g, U )\le \|H(\cdot,Xu^+)\|_{L^\fz(U)}.$ Recalling $\mu= \bi (g, U )$,
  one obtains  that  $\bi (g, U )= \|H(\cdot,Xu^+)\|_{L^\fz(U)},$ that is,
  $u^+$
is a minimizer for $\bi (g, U )$.
 By Lemma \ref{lem231}(i) again, $u^+ \le \csgu^+$ in $U$.
 Since $\cS_{h, V}^+= u^+$ in $V$,
we conclude that $\cS_{h, V}^+\le \csgu^+$ in $V$ as desired.
The proof is complete.
\end{proof}

{\color{black}
\begin{lem}\label{c0}
  Let $V \Subset \Omega$ be a  domain  and $P=\{x_j\}_{j \in \nn}\subset V$ be a dense subset of $V$. Assume $u \in C^{0}(\overline V)$ and $\{u_j\}_{j \in \nn} \subset \dot W_X^{1,\fz}(V) \cap C^{0}(\overline V)$ such that for any $j \in \nn$,
   \begin{equation}\label{pp01}
 |u_j(x_i)-u(x_i)| \le \frac1j \mbox{ for any $i=1, \cdots,j$, }
\end{equation}
  and
  \begin{equation}\label{pp02}
 \|H(\cdot,Xu_j(\cdot))\|_{L^\fz(V)} \le \lz <\fz.
  \end{equation}
  Then $u_j \to u$ in $C^0(V)$, and moreover,
  \begin{equation}\label{pp03}
    u\in \dot W^{1,\fz}_X(V ) \mbox{  and }
  \|H(x,Xu)\|_{L^\fz(V)}\le \lz.
  \end{equation}
\end{lem}}

\begin{proof} We only need to  prove $u_j \to u$ in $C^0(V)$. Note  that
 \eqref{pp03} follows from this and Lemma 3.1.

Since $u \in C^0(\overline V)$ and $\overline V$ is compact, $u$ is uniform continuous in $\overline V$,
 that is, for any $\ez>0$, there exists $h_\ez\in(0,\ez)$ such that for all
 \begin{equation}\label{u00}
    |u(x)-u(y)| \le \ez\quad \mbox{whenever $x,y \in \overline V$ with $|x-y|<h_\ez$.}
  \end{equation}
 Recalling the assumption (H3), by \eqref{pp02} one has
 \begin{equation}\label{mini5}
   \||Xu_j(\cdot)|\|_{L^\fz(V)} \le R_\lz \quad \forall j \in \nn.
 \end{equation}
By Lemma \ref{radcclem},
 $$u_j(y)-u_j(x)\le R_\lz d^V_{CC}(x,y ) \quad\forall x,y\in V.$$
Given any $K\Subset V$, recall from \cite{nsw85} that
$$d_{CC}^V(x,y)\le C(K,V)|x-y|^{1/k} \quad\forall x,y\in \overline K.$$
 It then follows
 $$|u_j(y)-u_j(x)|\le R_\lz C(K,V)|x-y|^{1/k} \quad\forall x,y\in \overline K.$$
 Given any $\ez>0$,  thanks to the density of $\{x_i\}_{i\in\nn}$ in
 $V$,  one has $\overline K\subset  \cup_{x_i\in \overline K}B(x_i,h_\ez)$.
 By the compactness of $\overline K$,  we have
 $$\overline K\subset  \cup\{B(x_i,h_\ez) \ | \ 1\le i\le i_K \ \mbox{ and }\ x_i\in\overline K\}$$
 for some $i_K\in\nn$.
 For any $j\ge \max\{i_K,1/\ez\}$ and  for any $x\in\overline K$,
  choose $1\le  i\le i_K$ such that $x_i\in \overline K$ and $|x-x_i|\le h_\ez<\ez$.
  Thus
  $| u(x_i)-u(x)|\le \ez$.
  By \eqref{pp01} we have   $| u_j(x_i)-u(x_i)|\le \frac1j\le \ez$.
Thus
$$| u_j(x)-u(x)|\le | u_j(x)-u_j(x_i)|+ | u_j(x_i)-u(x_i)|+| u(x_i)-u(x)| \le
R_\lz C(K,V)\ez^{1/k}+ 2\ez.$$
This implies that $u_j\to u$ in $C^0(\overline K)$ as $j\to\fz$.
The proof is complete.
\end{proof}

The following  clarifies the relations between absolute minimizers and
local super/subminimizers.
\begin{lem} \label{lem235}
Let $U\Subset\Omega$ and $g\in C^0(\pa U)$ with $\mu(g,\partial U)<\fz$.
Then   a function $u:\overline U\to\rr$ is an absolute minimizer for ${ \bf I}(g,U)$
if and only if it is both a local superminimizer and a local subminimizer for ${ \bf I}(g,U)$.
\end{lem}

\begin{proof}
  If $u$ is an absolute minimizer for $\bi (g, U )$, then for every subdomain $V \subset U$, $u$ is a minimizer for ${ \bf I}(u|_{\pa V},V)$. By Lemma \ref{cont}, $\cS_{u|_{\pa V}, V}^- \le u \le \cS_{u|_{\pa V}, V}^+$, that is, $u$ is both a local superminimizer and a local subminimizer for ${ \bf I}(g,U)$.

  Conversely,  suppose that  $u$ is both a local superminimizer and a local subminimizer for ${ \bf I}(g,U)$.
We need to show that $u$ is absolute minimizer for ${\bf I}(g,U)$. It suffices
   to prove that for any
  domain $V\Subset U$, $u$ is a minimizer for $\bi(u, V )$, in particular,
   $\|H(\cdot,Xu (\cdot))\|_{L^\fz(V)} \le \bi (u, V)$.


The proof consists of 3 steps.

\medskip
  {\bf Step 1.}  Given any subdomain $V \subset U$,  choose  a dense subset $\{x_j\}_{j \in \nn}$   of $V$.
 Set $V_j=V\setminus\{x_i\}_{1\le i\le j}$ and $V_0=V$. Note that
$$\partial V_j=\partial V_{j-1}\cup\{x_j\}=\partial V_0\cup \{x_i\}_{1\le i\le j} \quad \forall j\in\nn. $$
For each $j\ge0$, set
$$\mu_j=\mu(u|_{\partial V_j},\partial V_j)=\inf\{\lz\ge 0 | \  u(y) -u(x) \le d_\lz^{V_j}(x,y) \quad \forall x,y \in \pa V_j\}.$$
Since $\overline V \subset \overline U$, we have $d_\mu^U(x,y)\le d_\mu^V(x,y)$
          for all $ x,y\in \overline V^\ast$, we have
             $$  u(y) -u(x) \le d_\mu^V(x,y) \quad \forall x,y \in \overline V.  $$
          and hence  $\mu_0\le \mu$.  By Lemma \ref{lem231} (i),
          $\bi (u, V_0)=\mu_0$.
          In a similar way and by induction, for all $j\ge0$,  since $V_{j+1}\subset V_j$, we have
          \begin{equation}\label{uj4}
            \bi (u, V_{j+1})=\mu_{j+1}\le \mu_j=\bi (u, V_j)\le \mu_0=\bi (u, V_0)\le\mu=\bi (u, U).
          \end{equation}

{\bf Step 2. }
 We construct 
a sequence  $\{u_j\}_{j \in \nn}$ of functions  so that, for each $j\in\nn$,
\begin{equation}\label{pp1}
u_j\in \dot W^{1,\fz}_X(V_j)\cap C^0(\overline V)
 \mbox{ and $u_j(x )=u(x ) $   for any $ x\in\pa V_j= \pa V\cup\{x_i\}_{1\le i\le j}$  }
\end{equation} 
and
   \begin{equation}\label{pp2j}
  \mbox{
$\|H(\cdot,Xu_j(\cdot))\|_{L^\fz(V_{j-1})}=\mu_{j-1}$ for any $j \in \nn$}.
  \end{equation}

For any $j \ge 1$, since $u$ is both a local superminimizer and a local subminimizer for ${ \bf I}(g,U)$,    by Definition \ref{d231},
  $$
\cS_{u|_{\pa V_{j-1}}, V_{j-1}}^- \le u \le \cS_{u|_{\pa V_{j-1}}, V_{j-1}}^+
\mbox{ in $V_{j-1}$} .$$
{\color{black}At $x_{j}$, we have
\begin{equation}\label{uj}
  a_{j}:=\cS^-_{u|_{\pa V_{j-1}},V_{j-1}}(x_{j}) \le u(x_{j}) \le b_{j}:=\cS^+_{u|_{\pa V_{j-1}},V_{j-1}}(x_{j})
\end{equation}
  Define $u_{j}: \overline V_{j-1} =\overline V \to \rr$ by
   \begin{displaymath}
u_{j}:= \left \{ \begin{array}{lll}
\cS^+_{u|_{\pa V_{j-1}},V_{j-1}} & \quad \text{if} & \quad a_{j}=b_{j}, \\
\frac{u(x_{j})-a_{j}}{b_{j}-a_{j}}\cS^+_{u|_{\pa V_{j-1}},V_{j-1}}+ (1-\frac{u(x_{j})-a_{j}}{b_{j}-a_{j}})\cS^-_{u|_{\pa V_{j-1}},V_{j-1}} & \quad \text{if}&\quad a_{j}<b_{j}.
\end{array}
\right.
\end{displaymath}

 To see \eqref{pp1}, observe that
 Lemma \ref{lem231} gives $u_j\in \dot W^{1,\fz}_X(V_j)\cap C^0(\overline V) $.
Moreover, for any $x\in\pa V_j$, one has either $x\in\pa V_{j-1}$ or $x=x_j$.
In the case  $x\in \pa V_{j-1}$,  by Lemma \ref{lem231}one
has $$ u_{j}(x)=\cS^+_{u|_{\pa V_{j-1}},V_{j-1}}(x) = \cS^-_{u|_{\pa V_{j-1}},V_{j-1}}(x)= u(x).$$
In the case  $x=x_{j}$,  if $a_{j}=b_{j}$, then \eqref{uj} implies
$$u(x_{j}) = \cS^+_{u|_{\pa V_{j-1}},V_{j-1}} (x_{j}) =b_{j}=u(x_{j});$$
if $a_{j}<b_{j}$, then
 \begin{align*}
   u_{j}(x_{j}) & =\frac{u(x_{j})-a_{j}}{b_{j}-a_{j}}\cS^+_{u|_{\pa V_{j-1}},V_{j-1}} (x_{j})+ (1-\frac{u(x_{j})-a_{j}}{b_{j}-a_{j}})\cS^-_{u|_{\pa V_{j-1}},V_{j-1}} (x_{j}) \\
    & = \frac{u(x_{j})-a_{j}}{b_{j}-a_{j}} b_{j} + (1-\frac{u(x_{j})-a_{j}}{b_{j}-a_{j}})a_{j} \\
  &= u(x_{j}).
 \end{align*}

To see \eqref{pp2j},  applying Lemma \ref{lem231}(iii) with  $t=\frac{u(x_{j})-a_{j}}{b_{j}-a_{j}}$, we deduce that $u_{j}$ is a minimizer for $\bi(u|_{\pa V_{j-1}},V_{j-1})$, that is,
\begin{equation}\label{uj1}
  \| H(\cdot,Xu_{j}(\cdot)) \|_{L^\fz(V_{j-1})} = \bi(u|_{\pa V_{j-1}},V_{j-1})=\mu_{j-1}.
\end{equation}
}

{\bf Step 3. }  We show that, for all $j \in  \nn$,
\begin{equation}\label{pp3}
\mbox{
  $u_j(z)-u_j(y) \le d^V_{\mu}( y,z)  \quad \forall y,z \in V.$}
\end{equation}
Note that, by Corollary \ref{radu} in $V$,  \eqref{pp3} yields that
$u_j \in \dot W^{1,\fz}_X(V)$ and
 $\|H(x,Xu_j)\|_{L^\fz(V)}\le \mu$.
Applying Lemma \ref{c0} to $\{u_j\}_{j\in\nn}$ and $u$, we conclude that $u\in \dot W^{1,\fz}_X(V)$  and $\|H(x,Xu )\|_{L^\fz(V)}\le \mu$  as desired.

To see \eqref{pp3}, using \eqref{pp2j} and Corollary \ref{radu} in $V_{j-1}$, we have
\begin{equation*}
  u_j(z)-u_j(y) \le d^{V_{j-1}}_{\mu_{j-1}}(y,z) \quad \forall y,z \in  V_{j-1}.
\end{equation*}
Thus \eqref{uj4} implies
 \begin{equation*}
  u_j(z)-u_j(y) \le d^{V_{j-1}}_{\mu }(y,z) \quad \forall y,z \in  V_{j-1}.
\end{equation*}
Thanks to Lemma \ref{lem5.9} we have $d^{V_{j-1}}_{\mu }=d_\mu^V$ in $V\times V$ and hence
 \begin{equation*}
  u_j(z)-u_j(y) \le d^V_{\mu }(y,z) \quad \forall y,z \in  V_{j-1}.
\end{equation*}
By the continuity of $u_j$  in $\overline V$ we have  \eqref{pp3} and hence finish the proof.
\end{proof}

Finally, using Perron's approach, we obtain the existence of absolute minimizers.
\begin{prop} \label{p42}
Let $U\Subset\Omega$ and $g\in C^0(\pa U)$ with $\mu(g,\partial U)<\fz$.
Define
\begin{equation}\label{ug}
  U^+_g(x):=\sup\{u(x) | u :\overline U\to \rr\text{\ is a local subminimizer for ${ \bf I}(g,U)$} \}\quad\forall x\in\overline U
\end{equation}
and
$$U^-_g(x):=\inf\{u(x) | u :\overline U\to \rr\ \text{\ is a local superminimizer for ${ \bf I}(g,U)$} \}\quad\forall x\in\overline U.$$
Then $U^\pm_g$  are  absolute minimizers for ${ \bf I}(g,U)$.
\end{prop}

\begin{proof}
We only show that $U^+_g$ is an absolute minimizer for ${ \bf I}(g,U)$; similarly one can prove that
$U^-_g$ is also an absolute minimizer for ${ \bf I}(g,U)$. {\color{black} Thanks to Lemma \ref{lem235}, it suffices to show that $U^+_g$ is a minimizer for $\bi(g,U)$, a local subminimizer for $\bi(g,U)$ and a local superminimizer for $\bi(g,U)$.  }
Note that
 $\bi(g,U)=\mu(g,\partial U)<\fz. $

\medskip
\textbf{\bf Prove that  $U^+_g$ is a minimizer for $\bi(g,U)$.}
Firstly, since any local subminimizer $w$ for $\bi(g,U)$ is a minimizer for $\bi(g,U)$. We know
$$   w \in  \dot W^{1,\infty}_{X }(U) \cap C^0(\overline U), \ w|_{\pa U} =g, \mbox{ and } \|H(\cdot, Xw(\cdot))\|_{L^\fz(U)} =  \bi(g,U)<\fz.  $$
Recalling the assumption (H3), we have
 $   \| |Xw |\|_{L^\fz(U)} \le  R_{\bi(g,U)}   $
and hence $w\in\lip_{d_{CC}}( \overline U)$ with $\lip_{d_{CC}}(w,\overline U)\le R_{\bi(g,U)}$. By a direct calculation, one has
\begin{equation}\label{ugc0}
  \mbox{$U_g^+\in\lip_{d_{CC}}( \overline U)$ with $\lip_{d_{CC}}(U^+_g,\overline U)\le R_{\bi(g,U)}$     and  $U_g^+|_{\pa U}= g.$}
\end{equation}

Next, let $\{x_i\}_{i \in \nn}$  be a dense set of $U$.
For any $i,j\in\nn$, {\color{black} by the definition of $U_g^+$,} there exists a local subminimizer $u_{ij} $ for $\bi(g,U)$ such that
  $$u_{ij} (x_i) \ge U^+_g(x_i) - \frac{1}{j}.$$
Note that, by Definition \ref{d231},  $u_{ij}$ is also a minimizer for ${ \bf I}(g,U)$.

Moreover, for each $j\in\nn$, write $$u_j := \max\{u_{ij} \}_{1\le i\le j}.$$
 Lemma \ref{lem231}(iii) implies that $u_j$   is a minimizer for $\bi(g,U)$ and hence
  \begin{equation}\label{ujj0}
    u_j \in \dot W_X^{1,\fz}(U) \mbox{ and } \|H(\cdot,Xu_j(\cdot))\|_{L^\fz(U)} \le \bi(g,U)  \quad \forall j \in \nn.
  \end{equation}
   For $1\le i\le j$,
 from the definition of $U_g^+$,
 it follows that
  \begin{equation}\label{ujj1}
    \quad U^+_g(x_i) \ge u_j(x_i) \ge U^+_g(x_i) - \frac{1}{j} .
  \end{equation}

   Finally, combining \eqref{ugc0}, \eqref{ujj0} and \eqref{ujj1}, we are able to
 apply Lemma \ref{c0} to $U_g^+$ and $\{u_j\}_{j \in \nn}$  so to get
 \begin{equation}\label{uniform}
    u_j \to U^+_g  \mbox{ in } C^0(U), \ U^+_g \in \dot W^{1,\infty}_{X }(U) \mbox{ and } \|H(\cdot,XU^+_g)\|_{L^\fz(U)}\le \bi(g,U).
  \end{equation}
 Hence
 $$ \|H(\cdot,XU^+_g)\|_{L^\fz(U)}= \bi(g,U).$$
 Together with \eqref{ugc0} yields that  $U^+_g$ is a minimizer for $\bi(g,U)$.

\medskip
\textbf{Prove that  $U^+_g$ is a local subminimizer for $\bi(g,U)$}.

 We argue by contradiction.
 {\color{black} Assume on the contrary that $U^+_g$ is not a local subminimizer for $\bi(g,U)$. Then, by definition, there exists
 a subdomain $V \subset U$, and
some $x_0$ in $V$ such that
$$U^+_g(x_0) > \cS_{h^+,V}^+(x_0),$$
where $\cS_{h^+,V}$ is the McShane extension in $V$ of
$h^+=U^+_g|_{\pa V}$.
}
By the definition of $U^+_g$, there exists a local subminimizer $u$ for $\bi(g,U)$ such that
\begin{equation}\label{x0}
  U^+_g(x_0)\ge u(x_0) > \cS_{h^+,V}^+(x_0).
\end{equation}
The definition of $U^+_g$  also gives
\begin{equation}\label{ug1}
  u\le U^+_g \mbox{ in $U$.}
\end{equation}
Define
$$E := \{x\in \overline V \ | \ u(x) > \cS_{h^+,V}^+(x)\}.$$
 {\color{black}  Since both $u $ and  $\cS_{h^+,V}^+$ are continuous,      $E$ is  an open subset of $\overline V$.
 Since  $$U^+_g = \cS_{h^+,V}^+ \mbox{ on $\pa V$},$$
by \eqref{ug1}, we infer that
 $u \le  \cS_{h^+,V}^+ \mbox{ on $\pa V$}$
and hence
 $E \subset V.$
 Obviously, $x_0\in E$.
 Denote by $E_0$ the component of $E$ containing $x_0$.
 }

 Recalling that $u$ is a local subminimizer for $\bi(g,U)$, {\color{black}by Definition \ref{d231},} we have
 $u \le \cS_{u|_{\pa E_0}, E_0}^+$ in $  E$.
 Since $x_0 \in E_0$, we have
   $   \cS_{u|_{\pa E_0}, E_0}^+(x_0) \ge u(x_0 ),$
 which, combined with \eqref{x0}, gives
 \begin{equation}\label{x01}
   \cS_{u|_{\pa E_0}, E_0}^+(x_0) \ge u(x_0 ) >   \cS_{h^+,V}^+(x_0).
 \end{equation}

 {\color{black}On the other hand, we are able to apply  Lemma \ref{lem234} with
 $(U,V,g, h=\csgu^+|_{\partial V} )$ therein replaced by $(V,E_0, h^+,
h= \cS^+_{h^+,V}| _{E_0})$ here
and then obtain
 $$    \cS_{h^+, E_0}^+  \le   \cS_{h^+,V}^+ \mbox{ in $E_0$.}       $$
 Since $u|_{E_0}= \cS^+_{h^+,V}| _{E_0}=h$,  at
  $x_0 \in E$, we arrive at
 \begin{equation}\label{x02}
   \cS_{u|_{\pa E_0}, E_0}^+(x_0)  \le   \cS_{h^+,V}^+(x_0).
 \end{equation}
 }
Note that \eqref{x01} contradicts with  \eqref{x02}  as desired.

\medskip
\textbf{Prove that  $U^+_g$ is a local superminimizer for $\bi(g,U)$.} By definition, it suffices to prove   that,  for any given subdomain $V \subset U$,  we have
$\cS_{h^+, V}^-\le U^+_g$ in $V$, where we   write
$h^+=U^+_g|_{\pa V}$.

 To this end,   define $u^+$ as in \eqref{upm} with $u$ therein replaced by $U^+_g$, that is,
  \begin{equation}\label{cla}
u^+:= \left \{ \begin{array}{lll}
\cS_{h^+, V}^-  & \quad \text{in} & \quad V \\
 U^+_g & \quad \text{in} & \quad \overline U \backslash V.
\end{array}
\right.
\end{equation}
Then  $u^+$ is a minimizer for $\bi(g,U)$. Indeed,
 since $U^+_g$ is a minimizer for $\bi(g,U)$, we know that $U^+_g$ satisfies \eqref{udmu} with $\mu=\bi(g,U)$ therein. This allows us to apply   Lemma \ref{lem5.10} with $u= U^+_g$ therein  and then
conclude that   $u^+\in \dot W^{1,\fz}_X(U)\cap C^0(U)$, $u^+=U^+_g =g$ on $\pa U$,
and $\|H(x,Xu^+ )\|\le \bi(g,U) $.   Therefore, by definition of $\bi(g,U) $,
 $\|H(x,Xu^+ )\|= \bi(g,U) $, and hence   $u^+$ is a minimizer for $\bi(g,U)$.

We further claim that
  \begin{equation}\label{cllsub}
   \mbox{$u^+$   is a local subminimizer for $\bi(g,U)$.}
   \end{equation}
  Assume that  this claim holds.  Choosing $u^+$  as a test function in the definition of $U^+_g$ {\color{black}in \eqref{ug},  we know that
  $$U^+_g \ge u^+ \mbox{ in $U$}.$$
  }
  and in particular $U^+_g \ge \cS_{h^+, V}^- $ in $V$ as desired. 


  \medskip
  {\bf Proof the claim \eqref{cllsub}.}
 To  get \eqref{cllsub}, by Definition \ref{d231}, we still need to show for any subdomain $B \subset U$,
  \begin{equation}\label{deft}
    u ^+\le 
    \cS_{u^+|_{\pa B},B}^+ \text{ in $B$}.
  \end{equation}

  To prove \eqref {deft}, we argue by contradiction.
   Assume that \eqref {deft} is not correct,
   that is, \begin{equation}\label{ccc}
  W := \{x \in B \ | \ u^+(x) > \cS_{u^+|_{\pa B},B}^+ (x)\}\ne \emptyset.
   \end{equation}
   Up to considering some connected component of $W$, we may assume that $W$ is connected.
     Note that
         \begin{equation}\label{paW1}
\mbox{$u^+ = \cS_{u^+|_{\pa B},B}^+ $ on $\pa W$. }
   \end{equation}
    Consider the set
   \begin{equation}\label{defd}
     D := \{x \in B \ | \ U^+_g(x) > \cS_{u^+|_{\pa B},B}^+ (x)\}.
   \end{equation}
   By continuity, both  of $W$ and $D$ are open.
Below,  we consider   two cases: $D=\emptyset$ and $D\ne \emptyset$.

\medskip
{\bf Case $D=\emptyset$}.
  If $D$ is empty,  then we always have $U^+_g \le \cS_{u^+|_{\pa B},B}^+$  in $B$.
  Thus $$\mbox{  $U^+_g \le \cS_{u^+|_{\pa B},B}^+< u^+$ in $W$. }$$
  Since $U^+_g=u^+\in \overline U\setminus V$,  this implies
  \begin{equation}\label{incl04}
    W  \subset V.
  \end{equation}
   Since  $u^+= \cS_{h^+, V}^-$  in $\overline V$  and
   $W \subset V$ gives $ \pa W \subset \overline V$,
 we have
 \begin{equation}\label{paW2}\mbox{$u^+= \cS_{h^+, V}^-$ on $\pa W$. }\end{equation}


 Moreover, by Lemma \ref{lem234}, we know that
  $ \cS_{h^+, V}^-$  with $h^+=U^+_g|_{\pa V}$ is a local subminimizer for $\bi(U^+_g,V)$.
By the definition of local subminimizer,  and by $W\subset V$,  we have
  \begin{equation}\label{e91}
    \cS_{h^+, V}^-  \le \cS_{u^+|_{\pa  W}, W}^+    \quad \text{in } W,
  \end{equation}
  where we recall $\cS_{h^+, V}^-|_{\pa W}=u^+|_{\pa  W}$  from \eqref{paW2}.


  {\color{black}
  Applying Lemma \ref{lem234} with $(U,V,g,h^+)$ therein replaced by
  $(B,W, u^+|_{\pa B}, \cS_{u^+|_{\pa B},B}^+|_{\pa W})$ here,
  recalling $\cS_{u^+|_{\pa B},B}^+|_{\pa W}=u^+|_{\pa  W}$  from    \eqref{paW1},
  we obtain
      \begin{equation}\label{e931}
    \cS_{u^+|_{\pa  W}, W}^+
  \le   \cS_{u^+|_{\pa B},B}^+   \mbox{ in $W$.}
  \end{equation}
Combing \eqref{e91} and \eqref{e931},  by $W\subset V$, one arrives at
$$
   u^+= \cS_{h^+, V}^-  \le  \cS_{u^+|_{\pa B},B}^+ \mbox{ in $W$,}
$$
   which contradicts with \eqref{ccc}.
}

\medskip
 {\bf Case  $D\ne\emptyset$.} Up to considering some connected component of $D$, we may assume that $D$ is connected.
 By the definition of $D$ {\color{black} in \eqref{defd}}, we infer that
\begin{equation}\label{paD}U^+_g =  \cS_{u^+|_{\pa B},B}^+ \mbox{ on $\pa D$.}
\end{equation}
  {\color{black} Since  $U^+_g$ is a local subminimizer for $\bi(g,U)$ as proved above, we know}
\begin{equation}\label{e95}
  U^+_g \le \cS_{U^+_g|_{\pa D},D}^+  \text{ in }  D.
\end{equation}

{\color{black}
 Applying Lemma \ref{lem234} with $(U,V,g,h^+)$ therein replaced by
 $(B,D, u^+|_{\pa B},   \cS_{u^+|_{\pa B},B}^+|_{\pa D})  $,
 recalling  $\cS_{u^+|_{\pa B},B}^+|_{\pa D}=U^+_g |_{\pa D}$ from \eqref{paD},
  we obtain
  \begin{equation}\label{e951}
     \cS_{U^+_g |_{\pa D}, D}^+  \le     \cS_{u^+|_{\pa B},B}^+ \mbox{ in $D$.}
  \end{equation}
Combining \eqref{e95} and \eqref{e951}, we deduce
 $$  U^+_g \le  \cS_{u^+|_{\pa B},B}^+\mbox{ in $D$.}$$
  Recalling \eqref{defd}, this contradicts with  $D\ne\emptyset$.
 The proof is complete.
 }
\end{proof}

Theorem \ref{t231} is now a direct consequence of the above series of results.

\begin{proof}[Proof of Theorem \ref{t231}]
Let $g\in\lip_{d_{CC}} (\partial U)$. It suffices to show that $\mu(g,\pa U)<\fz$,
which allows us to use  Proposition \ref{p42} and then conclude the desired absolute minimizer $U_g^+ $ therein.

Taking  $0<\lz<\fz$ such that $R_\lz' \ge \lip_\dcc(g,\pa U)$, we have
$$ g(y) -g(x) \le   R_\lz'\dcc(x,y) \quad \forall x,y \in \pa U.$$
From Lemma \ref{lem311} (ii), that is,
 $R'_\lz d_{CC} \le  d^U_\lz,$
it follows that
$$ g(y) -g(x) \le   d_\lz^U(x,y) \quad \forall x,y \in \pa U.$$
and hence that
$$\mu(g,\partial U) = \inf\{\mu \ge 0 \ | g(y) -g(x) \le d_\mu(x,y)\} \le \lz<\fz.$$
The proof is complete.
\end{proof}

\section{Further discussion}


Note that Rademacher type Theorem \ref{rad} is a cornerstone when showing the existence of absolute minimizers. Indeed, Champion and Pascale \cite{cp} and Guo-Xiang-Yang \cite{gxy} established partial results similar to Theorem \ref{rad} for a special class of Hamiltonians considered in this paper to show the existence of absolute minimizers. However, their method seems to be invalid for more general Hamiltonians considered in this paper. We briefly explain the reason below.

\begin{rem}\rm
Champion and Pascale \cite{cp} showed the McShane extension is a minimizer for $H$ when $H$ is lower semi-continuous on $U \times \rn$. In fact, they defined another intrinsic distance induced by $H(x,p)$.
For every $\lz\ge 0$,
\begin{equation*}
L_\lz (x,q):=\sup_{\{p\in H_\lz(x)\}} p\cdot q, \quad  \forall \ x\in\overline U \ \mbox{and }\  q\in\rrm.
\end{equation*}
where $H_\lz(x)$ is the sub-level set at $x$, namely, $H_\lz(x)=\{p \in \rr^m \mid H(x,p) \le \lz\}$.

For $0\le a<b\le +\infty$, let  $\gz:[a,\,b]\to \overline U$ be a Lipschitz curve, that is, there exists a constant $C>0$ such that $|\gz(s)-\gz(t)|\le C|s-t|$ whenever $s,t\in[a,b]$.
The $L_\lz$-length of $\gz$ is defined by
\begin{equation*}
\ell_\lz(\gz):= \int_a^bL_\lz\lf(\gz(\theta),\gz'(\theta)\r)\,d\theta,
\end{equation*}
which is nonnegative, since $L_\lambda(x,q)\ge 0$ for any $x\in\overline U$ and $q\in\rn$.
For a pair of  points  $x,\,y\in\overline U$, the $\odl$-distance from $x$ to $y$ is defined by

\begin{equation*}
\odl(x,y):= \inf\Big\{\ell_\lz(\gz)\ |\ \gamma\in\mathcal C(a,b,x,y, \overline U)\Big\}.
\end{equation*}
Then, they prove two intrinsic pseudo-distance are equal, that is

\begin{equation}\label{e15}
  \dl(x,y)=\odl(x,y) \quad \text{for any} \ \lz >0 \ \text{and for any} \ x,y \in \overline U.
\end{equation}
Thanks to the definition of $\odl$, they can justify (i)$\Leftrightarrow$(ii) in Rademacher type Theorem \ref{rad}.

However, when asserting \eqref{e15}, we will meet obstacles in generalizing \cite{cp} Proposition A.2 since we are faced with measurable $H$.

On the other hand, Guo-Xiang-Yang \cite{gxy} provided another method to identify a weak version of \eqref{e15} for measurable Finsler metrics $H$, that is \begin{equation}\label{e16}
  \lim_{y \to x}\frac{\dl(x,y)}{\wz d_\lz(x,y)}=1 \quad \text{for any} \ \lz >0 \ \text{and for any} \ x,y \in \overline U.
\end{equation}
Here $\wz d_\lz$ induced by measurable Finsler metrics $H$ is defined in the following way.
\begin{equation} \label{q2}
\wz d_\lz(x,y):= \sup_N\inf\Big\{\ell_\lz(\gz)\ |\ \gamma\in \Gamma_N(a,b,x,y,  \overline  U)\Big\}
\end{equation}
where the supremum is taken over all subsets $N$ of $  \overline U$ such that $|N| = 0$ and $\Gamma_N(a,b,x,y,  U)$
denotes the set of all Lipschitz continuous curves $\gz$ in $  \overline U$ with end points $x$ and $y$ such that
$\mathcal H^1(N \cap \gz) = 0$ with $\mathcal H^1$ being the one dimensional Hausdorff measure.

In fact, \eqref{e16}, combined with the method in \cite{cp} will be sufficient for validating (i)$\Leftrightarrow$(ii) in Rademacher type Theorem \ref{rad}. Unfortunately, since we are coping with H\"{o}rmander vector field, a barrier arises when modifying their proofs. Indeed, their method uses a result by \cite{da} that every $x$-measurable Hamiltonian $H$ can be approximated by a sequence of smooth Hamiltonians $\{H_n\}_n$ such that two intrinsic distances $\wz d_\lz^{H_n}$ and $d_\lz^{H_n}$ induced by $H_n$ by means of \eqref{q2} and \eqref{d311} satisfy $ \lim_{n \to \fz}d_\lz^{H_n} = d_\lz$ and $ \limsup_{n \to \fz}\wz d_\lz^{H_n} \le \wz d_\lz$ uniformly on $ U\times U$ respectively. The process of the proof of \cite{da} is based on a $C^1$ Lusin approximation property for curves. Namely, given a Lipschitz curve $\gz: \ [0,1] \to U$ joining $x$ and $y$, for any $\ez >0$, there exists a $C^1$ curve $\widetilde{\gz}: \ [0,1] \to U$ with the same endpoints such that
  $$ \mathcal{L}^1 ( \{t \in [0,1]  \ | \ \widetilde{\gz}(t) \neq \gz(t) \quad \text{or} \quad \widetilde{\gz}'(t)\neq \gz'(t)\})<\ez$$
  where $\mathcal{L}^1$ denotes the one dimensional Lebesgue measure. Besides,
  $$ \|\widetilde{\gz}\|_{L^\fz} \le c\|\gz\|_{L^\fz}, $$
  for some constant $c$ depending only on $n$.
  Although this version of $C^1$ Lusin approximation property holds for horizontal curves in Heisenberg groups (\cite{l16}) and step 2 Carnot groups (\cite{ls16}),
   it fails for some horizontal curve in Engel group (\cite{l16}).

In summary, it is difficult to generalize the properties of the pseudo metric $\wz d_\lz$ not only from Euclidean space to the case of H\"{o}rmander vector fields but also from lower-semicontinuous $H(x,p)$ to measurable $H(x,p)$. Hence we would like to pose the following open problem.
\end{rem}

\begin{ques}
  Under the assumptions (H0)-(H3), does \eqref{e16} holds?
\end{ques}

\renewcommand{\thesection}{Appendix }
\newtheorem{lemapp}{Lemma \hspace{-0.15cm}}
\newtheorem{thmapp}[lemapp] {Theorem \hspace{-0.15cm}}
\newtheorem{corapp}[lemapp] {Corollary \hspace{-0.15cm}}
\newtheorem{remapp}[lemapp]  {Remark  \hspace{-0.15cm}}
\newtheorem{defnapp}[lemapp]  {Definition  \hspace{-0.15cm}}
\renewcommand{\theequation}{A.\arabic{equation}}

\renewcommand{\thelemapp}{A.\arabic{lemapp}}

\section {Rademacher's theorem  in Euclidean domains---revisit   }
In this appendix, we state some consequence  of Rademacher's theorem (Theorem \ref{rade}) for   Sobolev and Lipschitz classes,  see Lemma \ref{coin} and Lemma \ref{rade'} below.
They were well-known in the literature and also partially  motivated our Theorem \ref{rad} and Corollary \ref{global}.
For reader's convenience, we give the details.

 Recall that $\Omega\subset\rn$ is  always a domain.
 The homogeneous Sobolev space $\dot W^{1,\fz}(\Omega)$ is the collection of all functions $u\in L^1_\loc(\Omega)$ with its distributional derivative $\nabla u=(\frac{\partial u}{\partial x_i})_{1\le i\le n}\in L^\fz(\Omega)$.   We equip $\dot W^{1,\fz}(\Omega)$
with the semi-norm $$\|u\|_{\dot W^{1,\fz}(\Omega)}=\|\nabla u\|_{L^\fz(\Omega)}.$$    Write  
$\dot W^{1,\fz}_\loc(\Omega)$ as the collection of all functions $u$ in $\Omega$ so that $u\in \dot W^{1,\fz}(V)$ whenever $V\Subset\Omega$. Here and below,  $V\Subset\Omega$ means that
$V$ is bounded domain with $\overline V\subset\Omega$.
On the other hand, denote by $\lip(\Omega)$
the collection of all Lipschitz functions $u$ in $\Omega$,
that is,   all functions $u$ satisfying \eqref{lip}. We equip $\lip(\Omega)$ with the semi-norm
$$\lip(u,\Omega)=\sup_{x,y\in\Omega,x\ne y}\frac{|u(y)-u(x)|}{|x-y|}=\inf\{ \mbox{$\lz\ge 0$ satisfiying \eqref{lip}}\} .$$
Denote by $\lip_\loc(\Omega)$ the collection of all   functions $u$ in $\Omega$  so that $u\in\lip(V)$ for any $V\Subset\Omega$.
Moreover,
 denote by $\lip^\ast(\Omega)$ the  collection of all functions  $u$ in $\Omega$
  with
\begin{equation}\label{suplip}\lip^\ast(u,\Omega):=\sup_{x \in\Omega }\lip u(x)<\fz.
\end{equation}

Obviously, $ \lip(\Omega)\subset\lip^\ast(\Omega)$ with the seminorm bound
 $\lip^\ast(u,\Omega)\le \lip(u,\Omega).$
%
%
  Next, we have the following relation.
 \begin{lemapp}  \label{lemast}
 We have $  \lip^\ast(\Omega)\subset\lip_\loc(\Omega)$.
For any convex subdomain $V\subset\Omega$ and $u\in  \lip^\ast(\Omega)$, we have
\begin{equation}\label{ast}
  |  u(x)-  u(y)|\le \|u\|_{\lip^\ast(V ) } |x-y|\quad\forall x,y\in V.
\end{equation}
 \end{lemapp}

 \begin{proof} Let $u\in  \lip^\ast(\Omega)$. To see $ u\in \lip_\loc(\Omega)$, it suffices to
 prove $u\in\lip(B)$ for any ball $B\Subset\Omega$.  Given any $x,y\in B$, denote by $\gz(t)=x+t(y-x)$. Since $A_{x,y}=\sup_{t\in[0,1]} \lip u(\gz(t))<\fz$,
 for each $t\in[0,1]$ we can find $r_t>0$ such that
 $|u(\gz(s))-u(\gz(t))|\le A_{x,y}|\gz(s)-\gz(t)|=A_{x,y}|s-t||x-y| $ whenever $|s-t|\le r_t$ and $s \in [0,1]$.
 Since $[0,1]\subset\cup_{t\in[0,1]}(t-r_t,t+r_t)$ we can find an increasing sequence
  $t_i\in[0,1]$ with $t_0 = 0$ and $t_N=1$ such that $[0,1]\subset \cup_{i=1}^N(t_i-\frac12r_{t_i},t_i+\frac12 r_{t_i})$. Write $x_i=\gz(t_i)$ for $i =0 , \cdots , N$. We have
 $$|u(x)-u(y)|=|\sum_{i=0}^{N-1}[u(x_i)-u(x_{i+1})]|\le
 \sum_{i=0}^{N-1}|u(x_i)-u(x_{i+1})|\le A_{x,y}\sum_{i=0}^{N-1}|x_i-x_{i+1}|=A_{x,y}|x-y|.$$
 Noticing that $A_{x,y}  \le \lip^\ast(u,B) \le \lip^\ast(u,\Omega) < \fz$ for all $x,y \in B$, we deduce that $u \in \lip(B)$ and hence $u \in \lip_\loc(\Omega)$.

 If $V \Subset \Omega$ is convex, then for any $x,y \in \Omega$, the line-segment joining $x$ and $y$ lies in $V$. Hence similar to the above discussion, we have
 $$ |u(x)-u(y)|\le A_{x,y}|x-y|, \ \forall x,y \in V $$
 and
 $A_{x,y}  \le \|u\|_{\lip^\ast(V)}< \fz$ for all $x,y \in V$. Therefore, \eqref{ast} holds and the proof is complete.
  \end{proof}

 On the other hand, functions $\dot W^{1,\fz}_\loc(\Omega)$ admit  continuous representatives.
\begin{lemapp} \label{sub}
(i) Each $u\in  \dot W^{1,\fz}_\loc(\Omega)$
admits a unique continuous representative $\wz u$, that is,
$\wz u\in \dot W^{1,\fz}_\loc(\Omega)$    with $\wz u=u$ almost everywhere in $\Omega$. Moreover, $\wz u\in \lip_\loc(\Omega)$, and
for any convex subdomain $V\subset\Omega$, we have
$$|\wz u(x)-\wz u(y)|\le \|u\|_{\dot W^{1,\fz}(\Omega)}|x-y|\quad\forall x,y\in V.$$

(ii) Each   $u\in  \dot W^{1,\fz} (\Omega)$
admits a unique continuous representative $\wz u$, that is,  $\wz u\in  \dot W^{1,\fz} (\Omega)$   with $\wz u=u$ almost everywhere in $\Omega$. Moreover,
$\wz u\in \lip^\ast(\Omega)$ with $\lip^\ast(\wz u,\Omega)\le \|u\|_{\dot W^{1,\fz}(\Omega)}$.
\end{lemapp}

\begin{proof}   Since (ii) can be shown in a similar way as (i), we only prove (i).
 Given any  convex domain $V\Subset\Omega$, for any pair $x,y$ of  Lebesgue points  of $ u$, we have
$$
 u(y)-u(x)=\lim_{\dz\to0} [u\ast\eta_\dz(y)-u\ast\eta_\dz(x)]=\lim_{\dz\to 0}\nabla (u\ast\eta_\dz)(x+t_\dz(y-x))\cdot (y-x)$$
 where $\eta_\dz$ is the standard mollifier in $\rr^n$ with its support  ${\rm spt}\eta_\dz \subset B(0,\dz)$ and $t_\dz \in [0,1]$.
 Also, since for any $z\in V$,
 $$|\nabla (u\ast\eta_\dz)(z)|= |(\nabla  u)\ast\eta_\dz)(z)|\le \|\nabla u\|_{L^\fz(B(z,\dz))}, $$
 we deduce that for any pair $x,y$ of Lebesgue points of $ u$,
 $$ |u(y)-u(x)|\le \|\nabla u\|_{L^\fz(V)} |y-x|.$$
 If $z \in V$ is not a Lebesgue point of $u$, let $\{z_i\}_{i \in \nn} \subset V$ be a sequence of Lebesgue points of $u$ converging to $z$. We have
 $$  \lim_{i \to \fz}  |u(z_i)-u(z_{i+1})|\le  \lim_{i \to \fz} \|\nabla u\|_{L^\fz(V)} |z_i-z_{i+1}| = 0,  $$
 which implies $\{u(z_i)\}_{i \in \nn} \subset V$ is a Cauchy sequence.
 Since $\|u\|_{L^\fz(V)} < \fz$, we know $\{u(z_i)\}_{i \in \nn}$ has a limit in $\rr$ independent of the choice of the sequence $\{u(z_i)\}_{i \in \nn}$. Define
 $\wz u(z):= u(z) $ if $z \in V$ is a Lebesgue point of $u$ and $\wz u(z)= \lim_{i \to \fz} u(z_i)$ if $z \in V$ is not a Lebesgue point of $u$ where $\{z_i\}_{i \in \nn} \subset V$ is a sequence of Lebesgue points of $u$ converging to $z$.
 We know $\wz u : V \to \rr$ is well-defined and moreover,
   $$ | \wz u(y)-\wz u(x)|\le \|\nabla u\|_{L^\fz(V)} |y-x|\quad\forall x,y\in V.$$
 Thus $\wz u\in\lip(V)$ with   $\sup_{x\in V}\lip  \wz u(x)\le \lip(\wz u,V)\le \|\nabla u\|_{L^\fz(V)}$. In particular, $\wz u$ is continuous, which shows (i).

\end{proof}
Thanks to  lemma \ref{sub}, below
 for any function $u\in \dot W^{1,\fz}_\loc(\Omega)$ or $u\in \dot W^{1,\fz}(\Omega)$,
up to considering its continuous representative $\wz u$, we may assume that $u$ is continuous.
Under this assumption, Lemma \ref{sub} further gives  $\dot W^{1,\fz}_\loc(\Omega)\subset \lip_\loc(\Omega)$,
and  $\dot W^{1,\fz} (\Omega)\subset \lip^\ast(\Omega)$ with a norm bound $\lip^\ast(u,\Omega)\le \|u\|_{\dot W^{1,\fz}(\Omega)}$.
  Rademacher's theorem (Theorem \ref{rade}) tells that their  converse are also true. Indeed, we have the following.

\begin{lemapp} \label{coin}
(i)
We have
  $ \dot W^{1,\fz}_\loc(\Omega)=\lip_\loc(\Omega)$  and
  $\lip( \Omega)\subset \dot W^{1,\fz} (\Omega)= \lip^\ast(\Omega)$ with
   $\lip^\ast(u,\Omega)= \|u\|_{\dot W^{1,\fz}(\Omega)}\le \lip(u,\Omega)$.

 (ii)  If $\Omega$ is convex, then  $\lip( \Omega)= \dot
 W^{1,\fz} (\Omega)= \lip^\ast(\Omega)$ with
   $\lip^\ast(u,\Omega)= \|u\|_{\dot W^{1,\fz}(\Omega)}= \lip(u,\Omega)$.
   \end{lemapp}

\begin{proof}
(i)
If $u\in \lip_\loc(\Omega)$, applying Rademacher's theorem (Theorem \ref{rade}) to all subdomains  $V\Subset\Omega$, one has $u\in \dot W^{1,\fz}_\loc(\Omega)$
 and $ |\nabla u(x)|=\lip u(x)$ for almost all $x\in \Omega$
 (whenever $u$ is differentiable at $x$). Hence $\lip_\loc(\Omega) \subset W^{1,\fz}_\loc(\Omega)$. Combining Lemma \ref{sub}(i), we know
 $\lip_\loc(\Omega) = W^{1,\fz}_\loc(\Omega)$.

 If $u\in \lip^\ast(\Omega)$,   that is, $\lip^\ast(u,\Omega) = \sup_{x\in\Omega}\lip u(x)<\fz$.
 We have $u\in \lip_{\loc}(\Omega)$ and hence $u\in \dot W^{1,\fz}_\loc(\Omega)$
 and   $ |\nabla u(x)|= \lip u(x) \le  \lip^\ast(u,\Omega)<\fz$ for almost all $x\in \Omega$.
 Thus $u\in \dot W^{1,\fz}(\Omega)$.

By definition, it is obvious that $\lip( \Omega)\subset  \lip^\ast(\Omega)$. Hence Lemma \ref{coin} (i) holds.

(ii) By Lemma \ref{coin} (i), we only need to show $\lip^\ast(\Omega) \subset \lip( \Omega)$. Applying Lemma \ref{lemast} with $V=\Omega$ therein, \eqref{ast} becomes
$$   \frac{ |  u(x)-  u(y)|}{|x-y|} \le \|u\|_{\lip^\ast(\Omega ) } \quad\forall x,y\in \Omega.  $$
Taking supremum among all $x,y \in \Omega$ in the left hand side of the above inequality, we arrive at
$$   \|u\|_{\lip(\Omega ) }  \le \|u\|_{\lip^\ast(\Omega ) },  $$
which gives the desired result.
 \end{proof}

%
%
%

\begin{remapp} \label{eg2}\rm  (i) Lemma \ref{lemast} and Lemma \ref{coin} fail if we   relax $ \sup_{x\in\Omega}\lip u(x)$ in the definition  \eqref{suplip}  to be $\|\lip u\|_{L^\fz(\Omega)}=\esup_{x\in\Omega}\lip u(x)$.
 This is witted by  the standard Cantor function $w$ in $[0,1]$.
 Denote by  $E$ the standard Cantor set.
It is well-known that $w$ is continuous but not absolute continuous in $[0,1]$. Since  Lipschitz functions are always absolutely continuous, we know that
  $w$ is neither Lipschitz nor locally Lipschitz in $\Omega=(0,1)$,
   and hence $w\notin \lip_\loc(\Omega)$.
 On the other hand,  observe that $\Omega\setminus E$ consists of a sequence of open intervals which mutually disjoint, and $w$ is a constant in each such intervals and hence
  $\lip w(x)=0 $  therein. So we know that $\lip w(x)=0 $
  in $\Omega\setminus E$.
Since $|E|=0$, we have $\|\lip u\|_{L^\fz(\Omega)}=0$.

(ii) In general,  if $\Omega$ is not convex, one cannot expect that  $ \dot W^{1,\fz} (\Omega)\subset \lip (\Omega)$ with a norm bound.  Indeed, consider the planar domain
 \begin{equation}\label{domainu}U:=\{x=(x_1,x_2) \in \rr^2 \ | \ |x|<1 \} \setminus [0,1) \times \{0\}.
\end{equation}
Indeed,
in the polar coordinate $(r,\tz)$, let $w: U \to \rr$ be
$$w(r,\tz) := r\tz \mbox{ for all $0< r <1$ and $0< \tz< 2\pi$}.$$
 One can show that $w\in \dot W^{1,\fz}(U)$ so that $w(x_1,x_2)< \pi/3$ when $1/2\le x_1<1$ and $0<x_2<1/10$, and $w(x_1,x_2)> 5\pi/6$
when $1/2\le x_1<1$ and $-1/10<x_2<0$.  One has  $\lip (w,\Omega)=\fz$ and hence $w\notin \lip(\Omega)$.
 \end{remapp}

The example in Remark \ref{eg2} (ii) also indicates that
 the Euclidean distance does not match the geometry of domains and hence
 $\lip (\Omega)$ defined via Euclidean distance is not the prefect one to understand  $ \dot W^{1,\fz} (\Omega)$.

Instead of Euclidean distance,  for any domain $\Omega$, we consider the intrinsic distance
\begin{equation}\label{dual}
  d^\Omega_E(x,y)=\inf\{\ell(\gz) \ | \ \mbox{$\gz:[0,1]\to\Omega$ is absolute coninuous curve joining $x,y$}\},
\end{equation}
where $\ell(\gz):=\int_0^1|\dot\gz(t)|\,dt$ is the Euclidean length.
We have the dual formula.
\begin{lemapp}\label{la.4}
(i) For any $x,y\in\Omega$,
 \begin{equation}\label{intrinsic}
   d^\Omega_E(x,y)
=\sup\{u(y)-u(x) \ | \ u\in\dot W^{1,\fz}(\Omega),\ \|\nabla u\|_{L^\fz(\Omega)}\le 1\}.
 \end{equation}

(ii) If  $x,y\in\Omega$ with $|x-y|\le \dist(x,\partial\Omega)$, then  $d^\Omega_E(x,y)=|x-y|$.

(iii)  If $\Omega$ is convex, then  $d^\Omega_E(x,y)=|x-y|$ for all $x,y\in\Omega$.
\end{lemapp}

\begin{proof}
(i) Set
\begin{equation}\label{wzd}
  \wz d^\Omega_E(x,y)
=\sup\{u(y)-u(x) \ | \ u\in\dot W^{1,\fz}(\Omega),\ \|\nabla u\|_{L^\fz(\Omega)}\le 1\}.
\end{equation}
Notice that $d^\Omega_E(x, \cdot) \in \Lip^\ast(\Omega) = \dot W^{1,\fz}(\Omega)$ (Lemma \ref{coin} (i)) and $\|\nabla d^\Omega_E(x, \cdot)\|_{L^\fz(\Omega)}\le 1$ for all $x \in \Omega$. Hence letting $d^\Omega_E(x, \cdot)$ be the test function in \eqref{wzd}, we see
$$  d^\Omega_E(x,y) \le \wz d^\Omega_E(x,y) \ \forall x,y \in \Omega.$$

To see the contrary, fix $x,y \in \Omega$. Let $\{u_i\}_{i \in \nn}$ be a sequence of test functions in \eqref{wzd} such that
$$  \wz d^\Omega_E(x,y)
= \lim_{i \to \fz } (u_i(y)-u_i(x)).   $$
Let $\gz:[0,1] \to \Omega$ be an arbitrary absolute continuous curve joining $x$ and $y$. Then there exists a domain $U \Subset \Omega$ with $\gz \subset U$. Let $\{\eta_\dz\}_{\dz >0}$ be the standard mollifiers in $\rr^n$. For each $i \in \nn$, we know $u_i \ast \eta_\dz \in C^\fz(\Omega)$ and $\|\nabla (u_i \ast \eta_\dz)\|_{L^\fz(U)}\le \|\nabla u_i \|_{L^\fz(U)}  \le 1$.
Then we have
\begin{align*}
  \wz d^\Omega_E(x,y) & = \lim_{i \to \fz} [u_i(y)-u_i(x)]  \\
   & =\lim_{i \to \fz} \lim_{\dz \to 0}[u_i \ast \eta_\dz(y) - u_i \ast \eta_\dz(x)]  \\
   & =\lim_{i \to \fz} \lim_{\dz \to 0} \int_{0}^1 \nabla (u_i \ast \eta_\dz) (\gz(t)) \cdot \dot \gz (t) \, dt \\
   & \le \lim_{i \to \fz} \lim_{\dz \to 0} \int_{0}^1 |\nabla (u_i \ast \eta_\dz)(\gz(t))|  |\dot \gz (t)| \, dt \\
   & \le  \int_{0}^1  |\dot \gz (t)| \, dt \\
   & = \ell(\gz).
\end{align*}
Finally, taking infimum among all absolute continuous curves joining $x$ and $y$ in the above inequality, we conclude
$$  \wz d^\Omega_E(x,y) \le d^\Omega_E(x,y) \ \forall x,y \in \Omega.$$

(ii) If $|x-y|\le \dist(x,\partial\Omega)$, then the line-segment $\gz$ joining $x$ and $y$ is contained in $\Omega$. Letting $\gz$ be the absolute continuous curve in \eqref{dual},
$$  |x-y|  \le d^\Omega_E(x,y) \le  \ell(\gz) = |x-y| \ \forall x,y \in \Omega.$$

(iii) If $\Omega$ is convex, for all $x,y\in\Omega$, since the line-segment joining them is contained in $\Omega$,  similarly to (ii),  we have $d^\Omega_E(x,y)=|x-y|$.
The proof is complete.
\end{proof}

Note that if $\Omega$ is not convex, one cannot expect   $d^\Omega_E(x,y)=|x-y|$ for all $x,y\in\Omega$. Indeed, if $\Omega$ is given by the domain $U$ as in \eqref{domainu},
for points $(1/2,\ez)$ and $(1/2,-\ez)$ with $\ez\in(0,1/10)$, the Euclidean distance between them is $2\ez$.
However, note that any curve $\gz:[0,1]\to\Omega$ joining them must have intersection with $(-1,0)\times\{0\}$, which is call $z$. One then deduce
that $$\ell(\gz)\ge |(1/2,\ez)-z|+|(-1/2,\ez)-z|\ge 1/2+1/2=1+2\ez.$$
Thus the intrinsic distance between $(1/2,\ez)$ and $(1/2,-\ez)$ is always larger than or
equals to $1+2\ez$.

With in Lemma \ref{la.4}, we show that the Lipschitz spaces defined via the intrinsic distance perfectly match
with the Sobolev space $ \dot W^{1,\fz}(\Omega)$, see Lemma \ref{rade'} below.
Denote by $\lip_{d^\Omega_E}(\Omega) $  the collection of all Lipschitz functions $u$ in $\Omega$ with respect to $d^\Omega_E$, that is,
$$ \lip_{d^\Omega_E}(u,\Omega):=\sup\frac{|u(x)-u(y)|}{d^\Omega_E(x,y)}<\fz.$$
We also denote by $\lip^\ast_{d^\Omega_E}(\Omega)$ the  collection of all functions  $u$ in $\Omega$
  with
$$\lip^\ast_{d^\Omega_E}(u,\Omega):=\sup_{x \in\Omega }\lip_{d^\Omega_E} u(x)<\fz.$$

\begin{lemapp}\label{rade'}
We have    $ \lip_{d^\Omega_E}(\Omega)= \dot W^{1,\fz} (\Omega)= \lip^\ast(\Omega)$ and
 $$\|\nabla u\|_{L^\fz(\Omega)}=\lip_{d^\Omega_E} (u,\Omega)=\lip^\ast(u,\Omega)=\sup_{x\in\Omega}\lip_{d^\Omega_E} u(x).$$
 \end{lemapp}

 \begin{proof}
Recall that Lemma \ref{wzd} gives $d^\Omega_E(x,y)=|x-y|$ whenever $|x-y|\le \dist(x,\partial\Omega)$.
One then has   $\lip_{d^\Omega_E}(\Omega) \subset\lip_\loc(\Omega)$,
and moreover,
 $\lip u(x)= \lip_{d^\Omega_E} u(x)$ for all $x\in\Omega$, which gives $\lip^\ast_{d^\Omega_E}(\Omega)=\lip^\ast (\Omega)$.

Next, we show $\dot W^{1,\fz} (\Omega) \subset \lip_{d^\Omega_E}(\Omega)$ and $ \lip_{d^\Omega_E} (u,\Omega) \le \|\nabla u\|_{L^\fz(\Omega)}$. Let $u \in \dot W^{1,\fz} (\Omega)$.  Then $\|\nabla u\|_{L^\fz(\Omega)} =:\lz < \fz.$ If $ \lz >0$, then $\lz^{-1}u \in \dot W^{1,\fz} (\Omega)$ and $\|\nabla (\lz^{-1} u )\|_{L^\fz(\Omega)} =1$. Hence $\lz^{-1}u$ could be the test function in \eqref{intrinsic}, which implies
  $$    \lz^{-1} u(y)- \lz^{-1}u(x) \le  d^\Omega_E (x,y) \ \forall x,y \in \Omega,  $$
 or equivalently,
 $$  \frac{|u(y)- u(x) |}{\|\nabla u\|_{L^\fz(\Omega)}} \le  d^\Omega_E (x,y) \ \forall x,y \in \Omega.  $$
 Therefore, $u \in \lip_{d^\Omega_E}(\Omega)$ and $ \lip_{d^\Omega_E} (u,\Omega) \le \|\nabla u\|_{L^\fz(\Omega)}$. If $\lz=0$, then similar as the above discussion, we have for any $\lz'>0$
 $$  \frac{|u(y)- u(x) |}{\lz'} \le  d^\Omega_E (x,y) \ \forall x,y \in \Omega.  $$
 Therefore, $u \in \lip_{d^\Omega_E}(\Omega)$ and $ \lip_{d^\Omega_E} (u,\Omega) \le \lz'$ for any $\lz'>0$. Hence $ \lip_{d^\Omega_E} (u,\Omega) =0 = \|\nabla u\|_{L^\fz(\Omega)}$.

Moreover, we show $ \lip_{d^\Omega_E}(\Omega) \subset \lip^\ast(\Omega)$ and $ \lip^\ast(u,\Omega) \le \lip_{d^\Omega_E} (u,\Omega)$. Let $u \in \lip_{d^\Omega_E}(\Omega)$. Then $\lip_{d^\Omega_E} (u,\Omega)<\fz$. Since $\lip^\ast_{d^\Omega_E}(u,\Omega) \le \lip_{d^\Omega_E} (u,\Omega)$ and $\lip^\ast_{d^\Omega_E}(u,\Omega) = \lip (u,\Omega)$, we arrive at
 $$   \lip (u,\Omega) \le \lip_{d^\Omega_E} (u,\Omega) < \fz. $$
 Therefore, $u \in \lip^\ast(\Omega)$.

 Finally, recalling that $\dot W^{1,\fz} (\Omega) = \lip^\ast(\Omega)$ and $\|\nabla u\|_{L^\fz(\Omega)} = \lip (u,\Omega) $ in Lemma \ref{coin}, we finish the proof.
 \end{proof}

\bigskip
\noindent Jiayin Liu

\noindent
School of Mathematical Science, Beihang University, Changping District Shahe Higher Education Park
  South Third Street No. 9, Beijing 102206, P. R. China

{\it and}

\noindent Department of Mathematics and Statistics, University of Jyv\"{a}skyl\"{a}, P.O. Box 35
(MaD), FI-40014, Jyv\"{a}skyl\"{a}, Finland

\noindent{\it E-mail }:  \texttt{jiayin.mat.liu@jyu.fi}

\bigskip
\noindent Yuan Zhou

\noindent School of Mathematical Sciences, Beijing Normal University, Haidian District Xinjiekou Waidajie No.19, Beijing 100875, P. R. China

\noindent {\it E-mail }:
\texttt{yuan.zhou@bnu.edu.cn}


\end{document}